\documentclass[10pt,a4paper]{amsart}
\usepackage[margin=15mm]{geometry}
\usepackage[numbers]{natbib}

\usepackage{amsthm}

\usepackage{graphicx,epsfig}
\usepackage{amssymb,amsmath,amsthm}
\usepackage[all]{xypic}
\usepackage[colorlinks]{hyperref}

\newtheorem{theorem}{Theorem}[section]

\newtheorem{conclusion}[theorem]{Conclusion}

\newtheorem{corollary}[theorem]{Corollary}

\newtheorem{definition}[theorem]{Definition}
\newtheorem{lemma}[theorem]{Lemma}
\newtheorem{notation}[theorem]{Notation}
\newtheorem{problem}[theorem]{Problem}

\newtheorem{proposition}[theorem]{Proposition}
\newtheorem{remark}[theorem]{Remark}
\newtheorem{warning}[theorem]{Warning}

\newtheorem{example}[theorem]{Example}

\newcommand{\x}{\boldsymbol{x}}
\renewcommand{\u}{\boldsymbol{u}}
\newcommand{\R}{\mathbb{R}}
\newcommand{\C}{\mathbb{C}}

\newcommand{\p}{\mathbb{P}}
\newcommand{\g}{\mathfrak g}

\newcommand{\gp}{\mathfrak p}

\newcommand{\X}{\mathfrak X}

\newcommand{\gl}{\mathfrak{gl}}
\newcommand{\aff}{\mathfrak{aff}}

\newcommand{\sll}{\mathfrak{sl}}
\newcommand{\gll}{\mathfrak{gl}}

\newcommand{\Span}[1]{\left\langle#1\right\rangle}
\DeclareMathOperator{\LL}{LGr}

\DeclareMathOperator{\tr}{tr}
\DeclareMathOperator{\ad}{ad}
\DeclareMathOperator{\id}{id}
\DeclareMathOperator{\rank}{rank}
\DeclareMathOperator{\vol}{vol}

\DeclareMathOperator{\Graph}{graph}

\DeclareMathOperator{\GL}{\mathsf{GL}}

\DeclareMathOperator{\SL}{\mathsf{SL}}
\DeclareMathOperator{\PGL}{\mathsf{PGL}}
\DeclareMathOperator{\Sp}{\mathsf{Sp}}

\DeclareMathOperator{\Gr}{Gr}

\DeclareMathOperator{\Graff}{Graff}

\DeclareMathOperator{\End}{End}

\DeclareMathOperator{\Aff}{Aff}
\DeclareMathOperator{\Ann}{Ann}%

\DeclareMathOperator{\Hom}{Hom}

\DeclareMathOperator{\dd}{d}

\newcommand{\E}{\mathcal{E}}
\newcommand{\F}{\mathcal{F}}
\newcommand{\D}{\mathcal{D}}

\newcommand{\CC}{\mathcal{C}}
\newcommand{\Th}{^\textrm{th}}
\newcommand{\St}{^\textrm{st}}

\newcommand{\Nd}{^\textrm{nd}}

\newcommand{\Asf}{\mathsf{A}}

\title{Contact manifolds, Lagrangian Grassmannians and PDEs}

\author{Olimjon Eshkobilov}
 \email{olim\_0190@mail.ru}
   \address{Dipartimento di Matematica ``G. L. Lagrange'', Politecnico di Torino, Corso Duca degli Abruzzi, 24, 10129 Torino, ITALY.}
\author{Gianni Manno}
 \email{giovanni.manno@polito.it}
   \address{Dipartimento di Matematica ``G. L. Lagrange'', Politecnico di Torino, Corso Duca degli Abruzzi, 24, 10129 Torino, ITALY.}
\author{Giovanni Moreno}
 \email{gmoreno@impan.pl}
   \address{Institute of Mathematics, Polish Academy of Sciences, ul. Sniadeckich 8, 00--656 Warsaw, POLAND.}
\author{Katja Sagerschnig}
 \email{katja.sagerschnig@univie.ac.at}
   \address{INDAM--Dipartimento di Matematica ``G. L. Lagrange'', Politecnico di Torino, Corso Duca degli Abruzzi, 24, 10129 Torino, ITALY.}

\date{\today}
\begin{document}
\maketitle
\begin{abstract}
In this paper we review a geometric approach to PDEs. We mainly focus on scalar PDEs in $n$ independent variables and one dependent variable of order one and two, by insisting on the underlying $(2n+1)$--dimensional contact manifold and the so--called Lagrangian Grassmannian bundle over the latter.  This work is based on  a 30--hours Ph.D course given by two of the authors (G. M.  and G. M.). As such, it was mainly designed as a quick introduction to the subject for graduate students. But also the more demanding reader will be gratified, thanks to the frequent references to current research topics and glimpses of higher--level mathematics, found mostly in the last sections.
\end{abstract}
\setcounter{tocdepth}{1}
\tableofcontents



\section*{Introduction}

The existing literature proposes various ways to interpret (non--linear) PDEs in geometric terms. They all share a common   core idea: one must give up looking for \emph{solutions} and   focus on the \emph{infinitesimal solutions} instead. And this is the moment when differential geometry comes into play. Such was the   pioneering intuition of \`Elie Cartan, which   reached its maturity in the modern theory of Exterior Differential Systems \cite{MR1083148}. An alternative approach, where the same basic idea is captured by the natural contact structures on jet spaces \cite{Ehresmann:InTSInPL1,Ehresmann:InTSInPL2}, was developed to provide a solid theoretical background to the rapidly increasing results in the context of Integrable Systems \cite{MR2813504,MR1670044}.\par
The advantage of these approaches lies in  their extreme generality. Notwithstanding the order, the number of dependent and independent variables, and the number of equations, a (system of) PDEs can be effectively rendered in one of these geometric pictures. As such, they provide the best environment to prove general theorems.\par
At the far side, there are the applicative results, where   single PDEs (or classes of them)  are examined case--by--case and \emph{ad hoc} techniques are developed thereby. Even if some of these techniques can be explained in terms of general theorems, the link between them is often overlooked. In particular, we put a particular emphasis on the fact that, for a small (though non--trivial) class of non--linear PDEs---namely scalar with one independent variable and of order one or two---the framework itself takes a particularly friendly form. At least, in the authors' opinion.\par
This is the framework based on contact manifolds and their prolongations, already described in the excellent book \cite{MR2352610}. The present work, beside being more concise, puts a particular focus on the geometry of the Lagrangian Grassmannian bundle, casting important bridges between the theory of $2\Nd$ order scalar PDEs and purely algebraic--geometric aspects of certain  projective varieties. Such interrelationships become particularly evident in the case of parabolic Monge--Amp\`ere equations, and a few recent results \cite{MR2503974,MR2985508,Alekseevsky2014144} are reviewed here, suitably framed against a common geometric background.
\subsection*{Structure of the paper}
These notes, being essentially based on a Ph.D course, begin with the review of basic classical notions, cover gradually even more complicated topics and, in the end, touch some recent research results. Roughly speaking, the material covered in Section  \ref{secPreliminaries},  Section \ref{secContSymp} and Section \ref{sec1stOrderPDEs}  explains how to introduce geometrically $1\St$ order jets and $1\St$ order PDEs, by using rather  standard tools of contact manifolds. We discuss, in particular, the \emph{holonomy equations}, thus paving the way to the higher--order setting. In Section \ref{secPrologContMan} and Section \ref{sec.SecondOrderPDEs}     we introduce $2\Nd$ order jets and  $2\Nd$ order PDEs in terms of  Lagrangian Grassmannian bundles over contact manifolds and its one--codimensional sub--bundle.    In Section \ref{secSymPDEBack} and Section \ref{sec.characteristics} we work in the general setting of $k\Th$ order jets and $k\Th$ order PDEs, in order to better appreciate the Lie--B\"acklund theorem and to frame the notion of \emph{characteristics} against a broader background. In the remaining sections we show how the so--obtained geometric framework allows to obtain non--trivial results in some special classes of PDEs, namely Monge--Amp\`ere and homogeneous PDEs.\par
More precisely, the content of the paper is as follows.\par
In Section \ref{secPreliminaries} we introduce the necessary toolbox for studying vector distributions on smooth manifolds, with a particular emphasis on the notions of complete integrability and non--integrability. The particular case of contact manifolds is dealt with in more detail in Section \ref{secContSymp}, and the classical Darboux's theorem is established. In view of its importance in the integration of certain Cauchy problems, a fundamental discrete invariant of vector fields known as  \emph{type} is introduced in this section.  Section \ref{sec1stOrderPDEs} contains some examples of  (systems of) $1\St$ order PDEs whose solutions are particularly evident in the geometry framework based on contact manifolds. The section ends with the so--called \emph{holonomy equations}, that are tautological equations allowing to introduce the framework for $2\Nd$ order PDEs.  The possibility of constructing a geometric theory for $2\Nd$ order PDEs by means of natural structures associated to contact manifolds is explained in  Section \ref{secPrologContMan}. The central notion of these notes, the \emph{Lagrangian Grassmannian}, appears precisely when contact manifolds are prolonged in order to ``accommodate'' $2\Nd$ order PDEs. The key properties of the Lagrangian Grassmannian, namely the identification of its tangent spaces with symmetric matrices, and the so--called Pl\"ucker embedding, are duly highlighted in this section. The former allows to speak of the \emph{rank} of a tangent vector, the latter allows to introduce a distinguished class of hypersurfaces, the \emph{hyperplane sections}. The arrangement of next  Section \ref{sec.SecondOrderPDEs} mirrors that of Section \ref{sec1stOrderPDEs}, with the geometric definition of a $2\Nd$ order PDE, and an example of resolution. The first part of these notes is ideally concluded with  Section \ref{secSymPDEBack}, where the classical definition of symmetries is reviewed and the Lie--B\"acklund theorem is established.\par
The second part is opened with Section \ref{sec.characteristics}. First, the method of integration by means of the characteristic field is recalled. Then, we review the classical notion of characteristic and the Cauchy--Kowalewskaya theorem. Finally, in the particular context of scalar $2\Nd$ order PDEs, the latter are recast in the above--introduced geometric framework based on contact manifolds and their prolongations. In  Section \ref{secMAE} we finally see how the peculiar geometry of the Lagrangian Grassmannian, more precisely the presence of the Pl\"ucker embedding, allows to single out a non--trivial and interesting class of scalar $2\Nd$ order PDEs, the \emph{Monge--Amp\`ere equations}, in two independent variables. We show that, for these equations, the family of their characteristics takes a particularly simple form, which is indeed the simplest possible in the entire class of scalar $2\Nd$ order PDEs (a pair of rank--2 sub--distributions of the contact distribution on the underlying contact manifold). The study of such \emph{characteristic variety} allows to find the normal forms for parabolic Monge--Amp\`ere equations.  Section \ref{sec.multidim.MAE} represents the multi--dimensional analogue of Section \ref{secMAE}. We introduce Monge--Amp\`ere equations in $n$ independent variables and we show how to integrate  a Cauchy problem by suitably extending the technique already used in the case of $1\St$ order PDEs. This  is supported by a working example. Following the general guidelines of Section \ref{sec.multidim.MAE}, in Section \ref{secComplessa} we study  Monge--Amp\`ere equations by working on symplectic manifolds instead of contact manifolds (thus obtaining the so--called \emph{symplectic} PDEs). In particular, we employ K\"ahler manifolds and we study the geometry of the corresponding  Monge--Amp\`ere equations, among which there is the celebrated Kantorowich equation.  In the last Section \ref{secVeryClear} we propose yet another characterization of Monge--Amp\`ere equations. As  the underlying contact manifold we chose the projectivised cotangent bundle of the projective space, understood as a homogeneous contact manifold of the Lie group $\PGL$. By using representation--theoretic arguments, we show that the Monge--Amp\`ere equation $\det u_{ij}=0$, together with an its ``exotic partner'', is the \emph{only} $\PGL$--invariant $2\Nd$ order PDEs on that particular ($\PGL$--homogeneous) contact manifold.

\subsection*{Acknowledgements}
These are lecture notes based on the course ``Contact manifolds, Lagrangian Grassmannians and PDEs'', lectured by G. Manno and G. Moreno at the Department of Mathematics, University of Turin, April--May 2016.
The research of all four authors has been partially supported by the project
 ``FIR (Futuro in Ricerca) 2013 -- Geometria delle equazioni differenziali''.
Giovanni Moreno and Gianni Manno have been also partially supported by
the Marie Sk\l odowska--Curie fellowship SEP--210182301 ``GEOGRAL".  The research of Giovanni Moreno has also been partially    founded by the  Polish National Science Centre grant
under the contract number 2016/22/M/ST1/00542.
Gianni Manno, Giovanni Moreno and Katja Sagerschnig are members of G.N.S.A.G.A of I.N.d.A.M. Katja Sagerschnig is an INdAM (Istituto Nazionale di Alta Matematica) research fellow.

\section{Preliminaries}\label{secPreliminaries}
\subsection{Prerequisites}

We assume the reader to be familiar with basic definitions and results from the theory of smooth manifolds. As a  standard reference for the definitions of tangent spaces, vector fields, differential forms and De Rham differential we recommend the book \cite{Lee2012}.

\subsection{Conventions}

Even if the foundations of the theory are universally accepted, there is  still   a slight variation  of terminology throughout the    existing literature. In order to avoid unnecessary confusion,  we make a concise list of  the conventions adopted within this paper.\par
If $F$ is a smooth map between smooth manifolds, $F^*$ denotes its   pull--back (acting on functions, differential forms, tangent covectors, etc.), $F_*$ its   push--forward (acting on vector fields, tangent vectors, etc.) and $dF$ its tangent map.  The $C^\infty(M)$--module of vector fields on the smooth manifold $M$ is denoted by $\X(M)$, and it coincides with the module of sections of the tangent bundle $TM$ f $M$. The $C^\infty(M)$--module of $k$-differential forms on the smooth manifold $M$ is denoted by $\Lambda^k(M)$, and it coincides with the module of sections of the cotangent bundle $T^*M$ of $M$.
The symbol $L_X$ stands for the Lie derivative along the vector field $X$.
If $\D$ is a distribution on a manifold $M$, i.e., a sub--bundle of $TM$, we abuse the notation and write $X\in\D$ to express that the vector field is a section of $\D$ (more rigorously, $X\in\Gamma\D$).\par

Even if we do not need any sophisticated algebraic--geometric tool, we   recall  below the elementary notions which we will make use of.
 If $V$ is a linear vector space, the symbol  $\p V=\p(V)$ denotes the set of all the $1$--dimensional linear subspaces in $V$. In particular, for $V=\R^{n+1}$, we often write
$$
\R\p^n = \{ \ell\mid\ell\textrm{ is a $1$--dimensional linear subspace of }\R^{n+1}\}
$$
instead of $\p(\R^{n+1})$. If $X\subset V$ is a subset of $V$, the symbol   $\Span{X}$ denotes the linear span of $X$ in $V$.\par
Projective spaces and Grassmannian varieties are standard gadgets in algebraic geometry. For any point $\ell\in\p V$, we call \begin{equation}\label{eq.affine.neig.1}
\ell^*\otimes\frac{V}{\ell}
\end{equation}
the \emph{affine neighborhood} centred at $\ell$. Similarly, for a point $L\in\Gr(l,V)$ of the Grassmannian of $l$--dimensional linear subspaces of $V$; we call
\begin{equation}\label{eq.affine.neig.2}
L^*\otimes\frac{V}{L}
\end{equation}
the \emph{affine neighborhood} centred at $L$. That is, we regard affine neighborhoods as spaces of linear maps. This is just a coordinate--free way to express the classical affine neighborhoods of  Grassmannian varieties, see, e.g., \cite[Section 5.4]{Smith2000}.  \par
In particular, if for $V$ we take the tangent space $T_pE$ to a smooth manifold $E$ at the point $p\in E$, the Grassmannian $\Gr(l,T_pE)$ may be thought of as the fibre  at $p$ of a (non--linear) bundle, henceforth denoted by $\Gr(l,TE)$. Similarly we define the bundles $\p TE$ and $\p T^* E$. Observe that $\Gr(\dim E-1,TE)$ identifies with $\p T^*E$.\par
We will also need the ``affine version'' of the  Grassmannian variety, i.e., the collection of all the \emph{affine} subspaces of a certain dimension of a given vector space. For instance, the symbol $\Graff(1,\R^2)$ will denote the set of all the straight lines in $\R^2$ (not necessarily passing through the origin).\par

The symbol $\ltimes$ stands for the semi--direct product.\par

The Einstein convention is used throughout the paper, unless we need to stress the range of the summation index.

\subsection{Distributions, integral submanifolds and the Darboux's theorem}\label{secDistro}

It is not an exaggeration to claim that any geometric approach to non--linear PDEs comes down, in one form or another, to the study of the integral submanifolds of a distribution on a suitable space parametrising independent variables, dependent variables, and their (formal) partial derivatives. This is why we begin by reviewing   the notion of a distribution, its fundamental properties, and the key theorem of Frobenius.\par
Then we prove the Darboux's theorem, which provides the indispensable link between an abstract contact manifold and the above--mentioned ``space of variables and formal (first) derivatives''.
\begin{definition}
Let $M$ be a smooth manifold of dimension $m$. An assignment of subspaces of $T_pM$ at each point $p\in M$ is called a \emph{distribution} (on $M$). A distribution $\D$ is said to be
of \emph{constant dimension} $\dim(\D)$ if $\dim(\D_p)$ is the same for any point $p\in M$. An $n$--dimensional distribution $\D$ is said to be of class $C^\infty$ if for each point $p\in M$ there exists a neighborhood $U$ and $n$ vector fields $X_1,\dots,X_n$ such that
$$
\D_q=\Span{ ({X_1})_q,\dots, ({X_n})_q}\,, \quad \forall\,q\in U\,.
$$
\end{definition}

\begin{warning}
Unless otherwise specified, in what follows, by the term ``distribution'' we mean a distribution of constant dimension of class $C^\infty$.
\end{warning}

\begin{remark}\label{rem.distr.dual}
An $n$--dimensional distribution $\D$ on an $m$--dimensional smooth manifold $M$ can be locally described by $m-n$ independent differential $1$--forms $\theta^1,\dots,\theta^{m-n}$ that annihilate $\D$:
\begin{equation}\label{eq.distr.1.forms}
\D=\ker(\theta^1)\cap\ker(\theta^2)\cap\dots\cap\ker(\theta^{m-n})\,.
\end{equation}
In general, this does not hold globally: for instance, for a codimension--1 distribution $\D$, the set of $1$--forms that annihilate $\D$ is a line bundle over $M$, since if $\alpha$ annihilates $\D$ also $f\alpha$ annihilates $\D$, where $f$ is a nowhere vanishing function on $M$. If such a line bundle is trivial, then one can construct globally a $1$--form annihilating $\D$. In this case, the distribution is called \emph{co--orientable} (see \cite[Section 8.5]{book:925352}).
\end{remark}

\begin{definition}
A system of type $\{\theta^1=\theta^2=\cdots = \theta^{m-n}=0\}$, where $\theta^i\in\Lambda^1(M)$, is called a \emph{system of Pfaffian equations}.
\end{definition}

\begin{warning}
 Unless otherwise specified, we are not concerned with global, i.e., topological, aspects of distributions. In particular, this allows us to \emph{always} regard a distribution as the common kernel of a certain set of differential forms. But the reader should bear in mind that these forms are defined only \emph{locally} in the neighborhood of each point.
 \end{warning}

\begin{definition}
An \emph{integral submanifold} of a distribution $\D$ on $M$ is a submanifold $N$ of $M$ such that $T_pN\subseteq \D_p$ $\forall p\in N$. An integral submanifold is said of \emph{maximal dimension} if $T_pN=\D_p$ $\forall p\in N$.
\end{definition}
In terms of $1$--forms, if a distribution $\D$ on $M$ is given by \eqref{eq.distr.1.forms}, then a submanifold $N\overset{\iota}{\hookrightarrow} M$ is integral  of $\D$ if and only if all the $1$--forms $\theta^i$ vanish on $N$, i.e.,
$$
\iota^*(\theta^1)=0\,,\,\iota^*(\theta^2)=0\,,\dots\,,\iota^*(\theta^{m-n})=0\, .
$$
In this case we have that
\begin{equation}\label{eq.dtheta.bla}
(d\theta^i)|_N=d(\theta^i|_N)=0\,,\quad i=1,\dots,m-n.
\end{equation}

\begin{definition}\label{defPrelimSymm}
Let $\mathcal{D}$ be a distribution on $M$. A diffeomorphism  $F:M\to M$ (resp., a vector field $X\in\X(M)$) is called a \emph{symmetry} (resp., \emph{infinitesimal symmetry}) of $\mathcal{D}$ if and only if $F_\ast (\mathcal{D})=\mathcal{D}$ (resp., $[X,\mathcal{D}]\subset \mathcal{D}$) or, equivalently, $F^*(\theta^k)=h^k_i\theta^i$ (resp.,  $L_X(\theta^k)=h^k_i\theta^i$), where $h^k_i\in C^\infty(M)$ and $\theta^k$ are defined by \eqref{eq.distr.1.forms}. An infinitesimal symmetry $X$ of $\D$ is called \emph{internal} or \emph{characteristic} if it belongs to $\D$.
\end{definition}
\begin{remark}
The name ``characteristic'' referred to an internal symmetry will be clear later on, starting from Section \ref{sec.characteristics}, when we shall study the characteristics (in the sense of Cauchy--Kowalewskaya) of a PDE. Basically, one can use such a kind of vector field to ``enlarge'' a Cauchy datum to a fully--dimensional solution of a given PDE.
\end{remark}

\begin{example}
Let us consider, in $\mathbb{R}^3=\{(x,y,z)\}$, the system of $1$--forms
\begin{equation}\label{eq.sys.ex}
\left\{
\begin{array}{l}
\theta^1=xdx+ydy\,,
\\
\theta^2=xydy+dz\, .
\end{array}
\right.
\end{equation}
System \eqref{eq.sys.ex} defines a $1$--dimensional distribution $\mathcal{D}$ (via formula \eqref{eq.distr.1.forms}) on $\mathbb{R}^3\smallsetminus\ell$, where $\ell=\{x=0\,,y=0\}$. In fact, $\mathcal{D}=\Span{ -y/x \partial_x + \partial_y -xy\partial_z}$. On $\ell$ the kernel of the system \eqref{eq.sys.ex} is $2$--dimensional. The integral curves of the form $(x,y(x),z(x))$ are
$$
\left(x\,,\,\mp\sqrt{-x^2+k_2}\,,\, \frac{1}{3}x^3+k_2\right)\, .
$$
\end{example}
\begin{example}
Let us consider, in $\mathbb{R}^3$, the $1$--form $dz-ydx=0$. Even if it describes a $2$--dimensional distribution on $\mathbb{R}^3$, its integral manifolds cannot be $2$--dimensional. Indeed, if we look for integral surfaces that are the graphs of a function $z=z(x,y)$, then condition $dz(x,y)-ydx=0$ leads to the system
$$
\frac{\partial z}{\partial x}=y\,,\quad \frac{\partial z}{\partial y}=0\, ,
$$
which is not compatible.
\end{example}
The above examples motivate the following definition.
\begin{definition}
A distribution $\D$ on $M$ is said to be \emph{completely integrable} if through  any $p\in M$ there passes an integral submanifold of maximal dimension. Otherwise, the distribution is said to be \emph{non--integrable}. If $\D$ is a codimension--1 distribution, then $\D$ is said to be \emph{completely non--integrable} if its maximal integral submanifolds are of the smallest possible dimension.
\end{definition}
We shall explain better the complete non--integrability condition in Section \ref{sec.non.int}.

\begin{definition}
Let $\D$ be a distribution. Let $[\D,\D]:=\{[X,Y]\,,X,Y\in\D\}$. Let us define $\D':=\D+[\D,\D]$. The distribution $\D'$ is called the \emph{derived} distribution of $\D$.
\end{definition}

\begin{definition}
A distribution $\D$ such that $[X,Y]\in\D$ for any vector fields $X,Y\in\D$ is called \emph{involutive}. In other words, $\D$ is involutive if and only if $\D'=\D$.
\end{definition}

\begin{theorem}[Frobenius]\label{th.Frobenius}
A distribution $\D$ is completely integrable if and only if it is {involutive}.
\end{theorem}

\begin{proof}[Proof of the necessary part]
Let $\dim\D=n$. Let $p\in M$. By hypothesis there exists an $n$--dimensional integral submanifold $N$ of $\D$ passing through  $p$. Let us consider a system of coordinates $(x^1,\dots,x^m)$ in a neighborhood $U$ of $p$ such that, in this  neighborhood,
\begin{equation}\label{eq.p.000}
p=(0,\dots,0)
\end{equation}
and $N$ is described by
$$
N=\{x^{n+1}=\cdots=x^m=0\}\,.
$$
Thus, $\partial_{x^1},\dots,\partial_{x^n}$ are tangent to $N$ and so, by hypothesis
\begin{equation}\label{eq.D.loc.expr}
\D=\Span{ \partial_{x^1},\dots,\partial_{x^n} }\,.
\end{equation}
Let $X,Y\in\D$. We have that
$$
X=\sum_{i=1}^m a^i\partial_{x^i}\,, \quad Y=\sum_{i=1}^m b^i\partial_{x^i}\, ,
$$
where $a^i,b^i\in C^\infty(U)$ and
\begin{equation}\label{eq.a0.b0}
\left\{
\begin{array}{l}
a^{n+1}(x^1,\dots,x^r,0,\dots,0)=\cdots=a^{m}(x^1,\dots,x^r,0,\dots,0)=0\, ,
\\
\\
b^{n+1}(x^1,\dots,x^r,0,\dots,0)=\cdots=b^{m}(x^1,\dots,x^r,0,\dots,0)=0\,,
\end{array}
\right.
\end{equation}
since $X$ and $Y$ are in \eqref{eq.D.loc.expr}. Now
$$
[X,Y]_p=\sum_{i,j=1}^m\left(a^i(p)\frac{\partial b^j}{\partial x^i}(p) - b^i(p)\frac{\partial a^j}{\partial x^i}(p)\right)\partial_{x^j}|_p
$$
and, taking into account \eqref{eq.p.000} and \eqref{eq.a0.b0}, we have that $[X,Y]_p\in \D_p$.
\end{proof}
\begin{proof}[Sketch of the proof of the sufficient part]
As a first step, in a neighborhood $U$ of a point $p\in M$, it is possible to find $n$ commuting vector fields $X_i$ spanning $\D$. Let $\{\varphi^i_t\}$ be the local $1$--parametric group associated to the vector field $X_i$. Since $[X_i,X_j]=0$ $\forall\,i,j$, we have that $\varphi^i_t\circ\varphi^j_s=\varphi^j_s\circ\varphi^i_t$.
Then the   map
$$
\varphi:(t_1,\dots,t_n)\in\mathbb{R}^n\mapsto \varphi^1_{t_1}(\varphi^1_{t_2}(\dots(\varphi^r_{t_n}(p))\dots))\in U
$$
is   well defined and its image   is an $n$--dimensional integral submanifold of $\D$.\par
A complete proof of the sufficient part can be found, e.g., in  \cite[Section 2.3(c)]{book:8247}.
\end{proof}

\begin{corollary}
Let $\mathcal{D}$ be an $n$--dimensional integrable distribution on an $m$--dimensional manifold $M$. Let $p\in M$. Then there exists a system of coordinates $(x^1,\dots,x^m)$ of $M$ about $p$ such that $\mathcal{D}=\Span{ \partial_{x^1},\dots,\partial_{x^n}  }$. Dually, there exist functions $f^1,\dots,f^{m-n}\in C^\infty(M)$ such that, locally, $\D=\ker(df^1)\cap\ker(df^2)\cap\dots\cap\ker(df^{m-n})$.
\end{corollary}
We underline that, if $\D=\{\theta^i=0\}$ is a completely integrable distribution, then, in view of \eqref{eq.dtheta.bla} (with $N$ a (maximal) integral manifold), we have that
\begin{equation}\label{eq.cond.dual.int}
d\theta^i=\lambda^i_j\wedge\theta^j\,, \quad \lambda^i_j\in\Lambda^1(M)\, .
\end{equation}
The Frobenius Theorem \ref{th.Frobenius} says that condition \eqref{eq.cond.dual.int} is also sufficient for the distribution $\D$ to be completely integrable.

For future purposes we need the following lemma, whose proof is straightforward.
\begin{lemma}\label{Lemma_Calcolo_Derivato}
Let $\mathcal{P}$ be a $k$-dimensional
distribution on a smooth manifold $M$ and let $I_{\mathcal{P}}$ be the
corresponding Pfaffian system, i.e., the set of $1$--form that annihilates $\mathcal{P}$. Then the Pfaffian system associated with the
derived distribution $\mathcal{P}^{\prime}$ is
\[
I_{\mathcal{P}}^{\prime}=\{\rho\in I_{\mathcal{P}}\text{ s.t. } L_X(\rho)\in
I_{\mathcal{P}} \,\, \forall X\in\mathcal{P}\}.
\]
\end{lemma}

\subsection{Codimension--1 distributions, dimension of its integral manifolds and complete non--integrability}\label{sec.non.int}
The simplest Pfaffian systems are those generated by a \emph{single} $1$--form, thus corresponding  to codimension--1 distributions. They deserve special attention, since the integrability properties of the distribution are captured by a unique $1$--form.\par
Let $\D=\{\theta=0\}$ be a codimension--1 distribution on $M$. In view of \eqref{eq.dtheta.bla}, where $N$ is an integral manifold of $\D$, we have that $d\theta(X,Y)=0$, $\forall\,X,Y\in T_pN$, implying that $T_pN$ is an isotropic space for $(d\theta)|_{\D_p}$. The dimension of $T_pN$ cannot be greater than $m-k-1$, where $m$ is the dimension of $M$ and $2k$ is the rank of $(d\theta)|_{\D_p}$. Thus, we have an upper bound for the dimension of the integral manifolds of $\D$ accordingly to the rank of $(d\theta)|_{\D_p}$. The latter measures, in a sense, the ``integrability'' of distributions: the smaller this rank, the bigger the dimension of maximal integral submanifolds. Thus, the dimension $k$ of integral submanifolds of a distribution $\D$ of codimension $1$ on an $m$--dimensional manifold $M$ always fulfils
\begin{equation*}
\frac{m}{2}-\frac{\delta(m)}{2}\leq k \leq m-1\,,
\end{equation*}
where $\delta(m)$ is equal to $0$ if $m$ is even and is equal to $1$ if $m$ is odd. In particular, the dimension of the maximal integral submanifolds of a $2m$--dimensional distribution $\D$ on a $(2m+1)$--dimensional manifold $M$ can range from $m$ ($\D$ is a completely non--integrable distribution) to $2m$ ($\D$ is a completely integrable distribution). The distribution $\D$ is completely non--integrable if and only if
\begin{equation}\label{eq.dual.codim.1.non.int}
\theta\wedge (d\theta)^m\neq 0\, ,
\end{equation}
where $(d\theta)^m$ is the $m$--fold wedge product of $d\theta$ with itself. In fact, the rank of $(d\theta)|_{\D}$ is equal to $2k$ if and only if $((d\theta)|_{\D})^k\neq 0$ but $((d\theta)|_{\D})^{k+1}= 0$. In the case $k=m$, $((d\theta)|_{\D})^m\neq 0$ implies \eqref{eq.dual.codim.1.non.int}, being $\D=\ker(\theta)$.

\begin{remark}\label{re.no.internal}
The above results imply that a completely non--integrable distribution $\D$ cannot have infinitesimal internal symmetries. In fact, towards a contradiction, if $X\in\D$ is a symmetry of $\D$, then $f\theta =L_X(\theta)=X\lrcorner d\theta$ for some $f\in C^\infty(M)$ (for the first equality see Definition \ref{defPrelimSymm},  whereas for the second one recall that $X\in\D$ and the Cartan formula $L_X(\alpha)=X\lrcorner d\alpha + d (X\lrcorner\alpha)$ of the Lie derivative), so that $\theta\wedge X\lrcorner d\theta\wedge (d\theta)^{m-1}=0$, thus contradicting \eqref{eq.dual.codim.1.non.int}.
\end{remark}

\section{Contact and symplectic manifolds}\label{secContSymp}
What makes the first derivatives  of an unknown function $u$ special among all--order derivatives is that they come exactly in the \emph{same number} as the independent variables. This statement of the obvious has indeed deeper ramifications, which were extensively exploited in geometric mechanics. Such a ``duality'' between independent variables and (formal) first derivatives (of a scalar function) can be captured via a so--called \emph{symplectic space}, an even--dimensional linear space equipped with an operator defining the duality. A symplectic space is  but the ``infinitesimal analog'' of a symplectic manifold, which is introduced below.

\subsection{Symplectic manifolds}
Symplectic manifolds are ubiquitous in geometric mechanics, as they naturally
parametrise positions and momenta. In the present context, a symplectic manifold is used as a mere tool  to facilitate the inception of Darboux coordinates into a contact manifold.
\begin{definition}\label{def.sympl.man}
A \emph{symplectic manifold} is a pair $(W,\omega)$ where $W$ is an $2n$--dimensional manifold and $\omega$ is a closed non--degenerate $2$--form on $W$, i.e.,
$$
d\omega=0\quad\text{and}\quad \quad \forall\,\xi\in T_pW\,\,\exists\,\eta\in T_pW\,\,|\,\,\omega(\xi,\eta)\neq0\,.
$$
A diffeomorphism which preserves $\omega$ is called a \emph{symplectic transformation} (a.k.a. \emph{symplectomorphism}). A \emph{Lagrangian submanifold} of $(W,\omega)$ is an $n$-dimensional submanifold of $W$ on which $\omega$ vanishes.
\end{definition}

\begin{remark}\label{rem.conf.symp}
By taking into account the notation of Definition \ref{def.sympl.man}, also $(W,f\omega)$ is a symplectic manifold, with $f\in C^\infty(W)$ a nowhere vanishing function, with the same Lagrangian submanifolds. Thus, the property to be Lagrangian pertains more to a \emph{conformal} symplectic manifold rather than to a symplectic manifold itself. A similar reasoning, of course, applies also to symplectic vector spaces.
\end{remark}

\begin{example}[The cotangent bundle]
Let $W:=T^*N$, for some smooth manifold $N$. The cotangent bundle $T^*N$ plays a key role as the phase space in geometric mechanics, and it possesses a canonical 1--form $\beta$, called \emph{Liouville} $1$--form. In the standard coordinates $(x^i,u_i)$ of $T^*N$,
\begin{equation}\label{eq.Liuoville}
\beta=u_idx^i\, ,
\end{equation}
($u_i$ are the momenta canonically conjugated to $x^i$). The   differential  $\omega:=d\beta$ equips  $W$ with the structure of a symplectic manifold. Let us define $\beta$ in a coordinate--independent way.
The 1--form $\beta\in\Lambda^1(N)$ can be defined by setting
\begin{equation}\label{eq.taut.cotangent}
\beta_\alpha(\xi):=\alpha({\pi_N}_*(\xi))
\end{equation}
for any covector $\alpha\in T^*N$, where $\pi_N:T^*N\longrightarrow N$ is the cotangent bundle.
The reader may immediately verify that the so--defined $\beta$ coincides, in local coordinates, with \eqref{eq.Liuoville}.
\end{example}
Next Theorem \ref{th.Darboux.sympl} says that all $2n$--dimensional symplectic manifolds are locally equivalent to the cotangent bundle $T^*\mathbb{R}^n$.
\begin{theorem}[Darboux]\label{th.Darboux.sympl}
Let $(W,\omega)$ be a symplectic manifold. Then for any point $p\in W$ there exists a neighborhood $U\subset W$ of $p$ where one can choose coordinates $(x^i,u_i)$ such that
\begin{equation}\label{eq.appoggio.xipi}
\omega|_U=\sum_{i=1}^n dx^i\wedge du_i\,.
\end{equation}
\end{theorem}
\begin{proof}
 See, e.g, \cite[Section 43B]{book:693821}.
\end{proof}
\begin{remark}\label{rem.choose.arb.coord}
We omitted the proof of Theorem \ref{th.Darboux.sympl} as we prefer to focus on a similar theorem for contact manifolds (see Theorem \ref{th.Darboux.contact} below). Nevertheless, it is important to stress that one of the coordinates functions $x^i,u_i$ of \eqref{eq.appoggio.xipi} can be chosen arbitrarily.
\end{remark}

\subsection{Contact manifolds}\label{sec.contact.manifolds.2}

We introduce now the main actor of this paper. A contact manifold is essentially defined as a very special Pfaffian system, generated by a unique ``generic'' $1$--form, up to a scalar factor. Then, the Darboux's Theorem \ref{th.Darboux.contact} below shows that such an abstract definition corresponds locally to a space with coordinates $x^1,\ldots, x^n,u,u_1,\ldots,u_n$, that is, independent variables, one dependent variable, and (formal) derivatives of the latter with respect to the former ones.

\begin{definition}\label{defContMan}
A \emph{contact manifold} is a pair $(M,\CC)$ where $M$ is an odd--dimensional manifold and $\CC$ is a completely non--integrable codimension--1 distribution on $M$. A diffeomorphism of $M$ which preserves $\mathcal{C}$ is called a \emph{contact transformation} (a.k.a. \emph{contactomorphism}). Locally, $\CC$ is the kernel of a $1$--form $\theta$ (defined up to a conformal factor), that we call a \emph{contact form}.
\end{definition}

\subsubsection{The Darboux's theorem for contact manifolds}

In this section we prove the Darboux's theorem for contact manifolds by ``simplectifying'' them and then by using Theorem \ref{th.Darboux.sympl}. This idea was explained in \cite{book:693821}, where one can find more details on this subject.
\begin{definition}\label{def.symplectification}
The \emph{symplectification} of a contact manifold $(M,\CC)$ is the symplectic manifold $(\widetilde{M},d\beta)$, where $\widetilde{M}$ is the line bundle
\begin{equation*}
\widetilde{\pi}:\widetilde{M}:=\bigcup_{p\in M}\{ \alpha\in T^*_pM\,\,|\,\,\ker\alpha=\CC_p\} \to M
\end{equation*}
and $\beta\in \Lambda^1(\widetilde{M})$ is the ``tautological'' $1$-differential form of $\widetilde{M}$:
\begin{equation}\label{eq.taut.general}
\beta_{\alpha}(\xi):=\alpha(\widetilde{\pi}_*(\xi))\,, \quad\alpha\in\widetilde{M}\,.
\end{equation}
\end{definition}
Recall that we have already met the tautological differential form \eqref{eq.taut.general} when we discussed the cotangent bundle, cf. \eqref{eq.taut.cotangent}.

\smallskip
Below we prove that the symplectification $(\widetilde{M},d\beta)$ of a contact manifold $(M,\CC)$ is actually a symplectic manifold.

\medskip\noindent
Of course $d\beta$ is a closed $2$-form, so we have to prove only its non--degeneracy. Locally, if $\theta$ is a contact form, then the canonical $1$-form $\beta$ is equal to
$$
\beta=f\, \widetilde{\pi}^*(\theta)\,, \quad \text{for some}\,\, f\in C^\infty(\widetilde{M})
$$
(here the function $f$ plays the role of vertical coordinate),
so that
$$
d\beta=df\wedge\widetilde{\pi}^*(\theta) + f\,d\widetilde{\pi}^*(\theta) = df\wedge\widetilde{\pi}^*(\theta) + f\,\widetilde{\pi}^*(d\theta)\,.
$$
Let $\xi$ be a non--zero vertical vector with respect to $\widetilde{\pi}$, i.e., $\widetilde{\pi}_*(\xi)=0$. Then
$$
d\beta(\xi,\eta)=df(\xi)\widetilde{\pi}^*(\theta)(\eta)\,.
$$
Since, in view of the verticality of $\xi$, $df(\xi)\neq 0$, then, in order to have $d\beta(\xi,\eta)\neq 0$, it is enough to choose a vector $\eta$ such that $\pi_*(\eta)\notin\CC_p$ ($\pi_*(\eta)\in T_pM$).

\smallskip\noindent
By the same reasoning, if $\xi$ does not project through $\widetilde{\pi}$ on $\CC$, it is enough to choose a non--zero vertical vector $\eta$ to have $d\beta(\xi,\eta)\neq 0$.

\smallskip\noindent
So, assume that $\xi$ is not vertical and that projects on the contact hyperplane $\CC$ through $\widetilde{\pi}$. So, we have
$$
d\beta(\xi,\eta)=df(\xi)\widetilde{\pi}^*(\theta)(\eta) + fd\theta\big(\widetilde{\pi}_*(\xi),\widetilde{\pi}_*(\eta)\big)\,.
$$
Since $\pi_*(\xi)$ is not zero and lies in the contact hyperplane, in view of the fact that $d\theta$ is not degenerate on the contact hyperplane, there exists a vector $\eta$ projecting on the contact hyperplane (i.e., $\widetilde{\pi}^*(\theta)(\eta)=0$) such that $d\theta\big(\widetilde{\pi}_*(\xi),\widetilde{\pi}_*(\eta)\big)\neq 0$, so that $d\beta(\xi,\eta)\neq 0$.

\begin{theorem}[Darboux for contact manifolds]\label{th.Darboux.contact}
Let $(M,\CC)$ be a $(2n+1)$-dimensional contact manifold. Let $\theta$ be a contact form, i.e., locally, $\CC=\ker\theta$. Then for any point $p\in M$ there exists a neighborhood $U\subset M$ of $p$ where one can choose coordinates $(x^i,u,u_i)$ such that
$$
\theta|_U=du-\sum_{i=1}^n u_idx^i\,.
$$
\end{theorem}
\begin{proof}
Let us consider the symplectification $\widetilde{\pi}:\widetilde{M}\to M$ of the contact manifold $M$. Let $\beta\in\Lambda^1(\widetilde{M})$ be the $1$-differential form constructed in Definition \ref{def.symplectification}. Let us fix a contact form $\theta$ in a neighborhood $U$ of a point $p\in M$: such a form defines, locally, a hypersurface $V$ in $\widetilde{M}$ which non--degenerately projects on $U$ through $\widetilde{\pi}$. Let $q$ be the point of $V$ such that $\widetilde{\pi}(q)=p$. By Theorem \ref{th.Darboux.sympl} there exists, in a neighborhood of $\widetilde{M}$ containing $q$, a system of coordinates $(\widetilde{x}^0,\widetilde{u}_0,\dots,\widetilde{x}^n,\widetilde{u}_n)$ such that
\begin{equation}\label{eq.dbeta}
d\beta=\sum_{i=0}^{n}d\widetilde{x}^i\wedge d \widetilde{u}_i\,.
\end{equation}
Since we can choose a coordinate arbitrarily (see Remark \ref{rem.choose.arb.coord}), we choose $\widetilde{u}_0$ in such a way that $V=\{\widetilde{u}_0=0\}$. Thus, in a (possibly) smaller neighborhood, in view of \eqref{eq.dbeta}, we have that $\beta=-\sum_{i=0}^{n}\widetilde{u}_i d\widetilde{x}^i + d\widetilde{u}$, so that
\begin{equation}\label{eq.tilde.res}
\beta|_V=-\sum_{i=1}^{n}\widetilde{u}_i d\widetilde{x}^i + d\widetilde{u}\,,
\end{equation}
where $\widetilde{x}^i$, $\widetilde{u}$ and $\widetilde{u}_i$ in  \eqref{eq.tilde.res} in what follows are to be considered restricted to $V$.
Now we define the functions $(x^i,u,u_i)$ as
$$
\widetilde{\pi}^*(x^i)=\widetilde{x}^i\,, \quad \widetilde{\pi}^*(u)=\widetilde{u}\,, \quad \widetilde{\pi}^*(u_i)=\widetilde{u}_i\,.
$$
It remains to prove that the above--defined $(x^i,u,u_i)$ are coordinates in a neighborhood of $p$. To this end we will show that $(\widetilde{x}^1,\dots,\widetilde{x}^n,\widetilde{u},\widetilde{u}_1,\dots,\widetilde{u}_n)$ are coordinates on the hypersurface $V$. Since $(\widetilde{x}^0,\dots,\widetilde{x}^n,\widetilde{u}_1,\dots,\widetilde{u}_n)$ are coordinates on $V$, it is enough to show that $\widetilde{u}$ depends non--trivially on $\widetilde{x}^0$, i.e., $\frac{\partial\widetilde{u}}{\partial \widetilde{x}^0}\neq 0$. We shall proceed as follows. Let $\upsilon$ be a coordinate vertical vector field with respect to $\widetilde{\pi}$. We have that $\big(d\beta(\partial_{\widetilde{x}^0},\upsilon)\big)|_V\neq 0$. Indeed, on the contrary,
\begin{equation}\label{eq.cond.appoggio}
d\beta(\partial_{\widetilde{x}^0}|_z,\eta)= 0 \,\,\forall\,\eta\in T_z\widetilde{M}\,, \,\,z\in V\,,
\end{equation}
since $d\beta(\partial_{\widetilde{x}^0}|_z,\xi)= 0\,\,\forall\,\xi\in T_zV$ and the aforementioned $\upsilon$ is transversal to $V$. But condition \eqref{eq.cond.appoggio} cannot hold as $d\beta$ is non--degenerate. Thus, $\big(d\beta(\partial_{\widetilde{x}^0},\upsilon)\big)|_V\neq 0$ which, in view of the fact that $[\partial_{\widetilde{x}^0},\upsilon]=0$ and $\beta(\upsilon)=0$, implies $\upsilon\big(\beta(\partial_{\widetilde{x}^0})\big)\neq 0$ on $V$, from which we infer $\beta(\partial_{\widetilde{x}^0})$, whence, in view of \eqref{eq.tilde.res}, $\frac{\partial\widetilde{u}}{\partial \widetilde{x}^0}\neq 0$, as wanted.
\end{proof}

\begin{definition}
The coordinates $(x^i,u,u_i)$ constructed in Theorem \ref{th.Darboux.contact} are called \emph{Darboux} or \emph{contact} coordinates.
\end{definition}
Let $(x^i,u,u_i)$ be contact coordinates on an $(2n+1)$-dimensional contact manifold. The locally defined vector fields
\begin{equation}\label{eq.contact.local}
D_{x^i}:=\partial_{x^i} + u_i \partial_u\,,\quad  \partial_{u_i}
\end{equation}
span the contact distribution $\CC$.
\begin{remark}\label{rem.ConfSympSpace}
Let $(M,\mathcal{C})$ be a $(2n+1)$-dimensional contact manifold. Let $\theta$ be a contact form. The restriction $\omega:=(d\theta)|_{\CC}$ defines on each hyperplane $\CC_p$, $p\in M$, a conformal symplectic structure as $\big(d(f\theta)\big)|_{\CC}=f(d\theta)|_{\CC}$ for any $f\in C^\infty(M)$ . Lagrangian planes of $\CC_p$ depend only on such conformal symplectic structure (see also Remark \ref{rem.conf.symp}) and they  are the tangent spaces to the maximal submanifolds of $\CC$ (that have dimension $\frac{\dim M - 1}{2}$ in view of the discussion at the beginning of Section \ref{sec.non.int}). We call these submanifolds \emph{Lagrangian} submanifolds according to the definition we gave in the context of symplectic manifolds. They are often called \emph{Legendrian} submanifolds.
\end{remark}
Since the contact distribution $\mathcal{C}$ is
completely non--integrable, it does not admit any internal infinitesimal symmetry (see Remark \ref{re.no.internal}), so that
the flow of any vector field $Y\in\mathcal{C}$
deforms $\mathcal{C}$. The sequence of iterated Lie derivatives
\begin{equation}\label{X_i(U)}
\theta,\,L_Y(\theta),\,L_Y^{2}(\theta),\dots,\,L_Y^{2n-1}(\theta)
\end{equation}
gives a measure of this deformation \cite{MR722524}.
\begin{definition}
Let $Y\in\mathcal{C}$. The \emph{type} of $Y$ is the rank of the system \eqref{X_i(U)}.
\end{definition}
Note that the type does not depend on the choice of $\theta$. Also, the type of the line distribution $\Span{Y}$ is well defined, rather than the type of vector field $Y$.

Since $\omega=(d\theta)|_{\CC}$ is non--degenerate (see Remark \ref{rem.ConfSympSpace}), it induces an isomorphism $\CC\simeq\CC^*$, that we shall continue to denote by $\omega$. Any vector field $Y\in\mathcal{C}$ is the image of a $1$-form $\sigma\in \mathcal{C}^*$ via $\omega$, i.e., $Y=Y_\sigma:=\omega^{-1}(\sigma)$.
\begin{definition}\label{def.involution}
A vector field $Y\in\mathcal{C}$ such that $Y=Y_{df}$ is called \emph{Hamiltonian}. Two functions $f$ and $g$ on $M$ are in \emph{involution} if $d\theta(Y_{df},Y_{dg})=0$ (or equivalently, if $Y_{df}(g)=0$).
\end{definition}
A straightforward computation shows that, in a system of contact coordinates $(x^i,u,u_i)$, a Hamiltonian vector field can be written as
\begin{equation}\label{eq.hamiltonian.local}
Y_{df}=\sum_{i=1}^{n} \,\, f_{u_i}D_{x^{i}} -
D_{x^{i}}(f)\partial_{u_{i}}.
\end{equation}
where $D_{x^i}$ are defined in \eqref{eq.contact.local}. A Hamiltonian vector field $Y_{df}$ satisfies the following properties:
\begin{equation}\label{eq.proprieta.Xf}
df(Y_{df})=Y_{df}(f)=0\,, \quad \theta(Y_{df})=0\,, \quad Y_{df}(\theta)=df-\frac{\partial f}{\partial u}\theta.
\end{equation}
Properties \eqref{eq.proprieta.Xf} imply:
\begin{itemize}
\item $Y_{df}$ is a vector field of type $2$;
\item $Y_{df}$ is a characteristic symmetry for the distribution
$Y_{df}^{\perp}=\{\theta=0,\,df=0\}$. In other words, $Y_{df}$ coincides with the
classical characteristic vector field of the $1\St$ order PDE $f(x^i,u,u_i)=0$ where $u_i = \partial u/\partial x^i$. See also Section \ref{sec.characteristics}, where such a vector field is used to construct a solution of a $1\St$ order PDE starting from a non--characteristic Cauchy datum.
\end{itemize}
\begin{notation}
 We shall denote the Hamiltonian vector field associated with a function $f\in C^\infty(M)$h by $Y_{df}$ and by $Y_f$.
\end{notation}

\smallskip
The next proposition is a consequence of the Darboux's theorem \ref{th.Darboux.contact}. We shall use it especially for getting the normal form of (multidimensional) Monge--Amp\`ere equation with integrable characteristic distribution (see Section \ref{secNormFormMAEs} later on).
\begin{proposition}\label{prop.int.one.way}
Let $(M,\mathcal{C})$ be a $(2n+1)$--dimensional contact manifold. Let $\mathcal{D}$ be an $n$-dimensional integrable
sub-distribution of $\mathcal{C}$. Then there exists a contact
chart $(x^i,u,u_i)$ of $M$ where $\mathcal{D}$ takes the form
\begin{equation}\label{eq.int.dist.vert}
\mathcal{D}=\langle\partial_{u_1}\,, \dots\,,\partial_{u_n}\rangle.
\end{equation}
\end{proposition}
\begin{proof}
Let us choose a contact form $\theta$.
Since $\mathcal{D}$ is integrable, we can find $n+1$ functions
$\{f^i\}_{i=1,...,n+1}$ such that $\mathcal{D}$ is described by
\begin{equation}\label{eq.2}
\mathcal{D}=\{df^1=0\,,\,\, df_2=0\,,\dots \,, df^n=0, \,\,
df^{n+1}=0\}ker(df^1)\cap\cdots\cap\ker(df^n)
\end{equation}
Since by hypothesis $\mathcal{D}$ is inscribed in $\mathcal{C}$,
the contact form $\theta$ must depend on $df^1,\,\, df^2\,,\dots \,, df^n\,,\,df^{n+1}$, i.e.,
\begin{equation}\label{eq.3}
\lambda \theta = df^{n+1} + \sum_{i=1}^{n} a_i df^i
\end{equation}
for some $\lambda, a_1,...,a_n\in C^\infty(M)$. Note that since
$\theta$ is a contact form, all $n+1$ differentials must appear on
the right side of \eqref{eq.3}, otherwise the condition \eqref{eq.dual.codim.1.non.int} would not be satisfied.
From \eqref{eq.3} we obtain that
\[
x^i=f^i\,, \quad u=f^{n+1}\,, \quad u_i=-a_i,
\]
are contact coordinates on $M$. In these coordinates, \eqref{eq.2} takes the form
\[
\mathcal{D}=\{dx^1=0\,,\,\, dx^2=0\,,\dots \,, dx^n=0, du=0\}\,,
\]
i.e., \eqref{eq.int.dist.vert}.
\end{proof}

\subsubsection{The Legendre transformation}\label{secLegTrans}

Let us fix a system of contact
coordinates $(x^i,u,u_i)$ so that $\theta=du-u_idx^i$.
The functions
\begin{equation}\label{eq.Legendre}
\tilde{u}=u-\sum_{i=1}^n u_i x^i\,, \quad \tilde{x}^i=u_i\,, \quad \tilde{u}_i=-x^i \,, \qquad  i=1,\dots,n\,,\end{equation}
define a local contact transformation $\Psi:\,(x^i,u,u_i)\to (\tilde{x}^i,\tilde{u},\tilde{u}_i)$ since $\theta=d\tilde{u}-\tilde{u}_id\tilde{x}^i$. Such transformation is known as the \emph{Legendre transformation}. The action of $\Psi$ on vector fields interchanges the
roles of $D_{x^{i}}$ and $\partial_{u_{i}}$:
\begin{equation*}
\Psi_{*}\left(  \partial_{u}\right)  =
\partial_{\tilde{u}} ,\quad\Psi_{*}\left(
D_{x^{i}}\right)  =
-\partial_{\tilde{u}_{i}},\quad\Psi_{*}\left(
\partial_{u_{i}}\right)  = D_{\tilde{x}^{i}}.
\end{equation*}
Sometimes it is useful to define a ``partial'' Legendre transformation. For
instance, we can divide the index $i=1,\dots,n$ into $\alpha=1,\dots,m$ and
$\beta=m+1,\dots,n$, and define
\begin{equation}\label{eq.partial.legendre}
\tilde{u}:=u-\sum_{\alpha=1}^m u_{\alpha}x^{\alpha} \,, \quad
\tilde{x}^{\alpha}:=u_{\alpha} \,, \quad
\tilde{u}_{\alpha}:=-x^{\alpha} \,, \quad
\tilde{x}^{\beta}:=x^{\beta}  \,, \quad
\tilde{u}_{\beta}:=u_{\beta} \,, \quad\alpha=1,\dots,m\,,\quad \beta=m+1,\dots,n\,.
\end{equation}
It is easy to realize that \eqref{eq.partial.legendre} is also a (local) contact transformation. In this case, only the first $m$ coordinates $x^{\alpha}$ and $u_{\alpha}$ are interchanged.

\subsection{The projectivised cotangent bundle}\label{sec.proj.cot.bundle}
Let $E$ be an $(n+1)$--dimensional smooth manifold. Consider the Grassmannian bundle
\begin{equation*}
\pi:J^1(E,n):=\Gr(n,TE)=\bigcup_{p\in E}\Gr(n,T_pE)\to E
\end{equation*}
of tangent hyperplanes to $E$. It is also known as the \emph{manifold of contact elements of $E$} or as the \emph{space of first jets of hypersurfaces of $E$}. In view of the line--hyperplane duality, it is easy to realise that $J^1(E,n)$ also coincides with $\mathbb{P}T^*E$, the projectivised cotangent bundle.
\begin{remark}\label{rem.Gianni.Giovanni}
The projectivised cotangent bundle $\mathbb{P}T^*E$ of an arbitrary manifold $E$ is the chief example of a contact manifold (see Proposition \ref{prop.HpContact} below). It is worth stressing that, by its very definition, $\mathbb{P}T^*E$ possesses an intrinsic notion of ``verticality'', that is, directions degenerately projecting onto $E$. However, in general, there is no ``verticality'' inherent to $E$. Still, there are important contexts where $E$ comes equipped with its own notion of ``verticality'', for instance a fiber bundle structure $\rho:E\longrightarrow \mathcal{X}$ of rank $1$. In this case, $\mathbb{P}T^*E$ acquires a notion of ``horizontality'' as well. More precisely, within $\mathbb{P}T^*E$ there appears an open and dense subset, usually denoted by $J^1(\rho)$ (known as the \emph{first jet of the bundle $\rho$}), made of the ``horizontal'' tangent hyperplanes, i.e., tangent to the  graphs of local sections of $\rho$. If $E$ is equipped with a trivial vector bundle structure $\rho:E\simeq N\times\R\to N$, then the bundle $J^1(\rho)$ is the \emph{first jet of functions on $N$} (for the relationship between these notions, see also Section \ref{sub.Graff1R2}).
\end{remark}
\begin{proposition}\label{prop.HpContact}
The distribution
\begin{equation}\label{eq.defCHp}
H_p\in \Gr(n,T_pE)\to \mathcal{C}_{H_p}:=(d\pi)_{H_p}^{-1}(H_p)
\end{equation}
is a contact one.
\end{proposition}
\begin{proof}
Local coordinates $(x^1,\ldots,x^{n},u)$ on $E$ induce local coordinates $(x^1,\ldots,x^{n},u,u_1,\ldots,u_{n})$ on $\p T^* E$, in such a way that   $H_p=\Span{\partial_{x^1}+u_1\partial_u\,,\ldots,\,\partial_{x^{n}}+u_{n}\partial_u}$ has coordinates $(x^1,\ldots,x^{n},u,u_1,\ldots,u_{n})$. Then from \eqref{eq.defCHp} it follows immediately that
\begin{equation*}
\mathcal{C}_{H_p}=H_p\oplus\Span{\partial_{u_1} ,\ldots,\partial_{u_{n}} }=\ker\theta_{H_p}\, ,
\end{equation*}
with
\begin{equation}\label{eqThetaPTStarM}
\theta=du-u_idx^i\, .
\end{equation}
Recalling the condition \eqref{eq.dual.codim.1.non.int} (also Theorem \ref{th.Darboux.contact}), one immediately sees that the $1$--form \eqref{eqThetaPTStarM} is a contact form.
\end{proof}
Let us fix a hypersurface $S$ of $E$. The natural map
\begin{equation}\label{eq.jS}
j^1S:p\in S\mapsto T_pS\in J^1(E,n)\,.
\end{equation}
allows us to ``prolong'' $S$ to $J^1(E,n)$.

\subsubsection{The Grassmannian of affine lines in $\R^2$}\label{sub.Graff1R2}
We study now the particular case of \eqref{eq.defCHp} when $M$ is the affine space $\R^2$. This example reveals important features of the projectivised cotangent bundle $\p T^*\R^2$, which hold also for $\p T^*M$, albeit less visibly. To begin with, observe that
\begin{equation*}
J^1(\R^2,1):=\Gr(1,T\R^2) =\p T^*\R^2=\R^2\times\R\p^1\,.
\end{equation*}
An important aspect of $J^1(\R^2,1)$ is that it can be constructed simply by using functions from $\R$ to $\R$. More precisely,
let $f,g\in C^1(\R)$, and define
\begin{equation}\label{eq.rel.equ.initial}
f\sim_x^1 g\Leftrightarrow f(x)=g(x)\textrm{ and }f'(x)=g'(x)\, .
\end{equation}
$\sim_x^1 $ is an equivalence relation. We shall denote by $[f]^1_x$ an equivalence class and by $J^1(1,1)$ the set of such equivalence classes:
\begin{equation}\label{eq.J1.1.1}
J^1(1,1):=\{[f]^1_x\,\,|\,\, f\in C^1(\R)\}
\end{equation}
(first jet of functions on $\R$, a more general construction is given in Section \ref{sec1stOrderPDEs}, see formula \eqref{eq.rel.equ.more.gen} and Definition \ref{def.Michoriana}).
 It holds
 \begin{equation}\label{eq.per.scuole.medie}
J^1(\R^2,1)\simeq\overline{\coprod_{x\in\R}C^1(\R)/\sim_x^1}\, .
\end{equation}
Indeed, if we let $(x,u,u_1)\in J^1(1,1)$  and choose a function $f\in C^1(\R)$ such that $f(x)=u$ and $f'(x)=u_1$, then $[f]_x^1$ is well--defined.
The space $J^1(1,1)$ may be dubbed ``the space of Taylor series coefficients of order $\leq 1$'' and it is an open dense subset of $J^1(\R^2,1)$.

\smallskip\noindent
There is a natural projection $J^1(\R^2,1)\longrightarrow\Graff(1,\R^2)$,
where
$$
\Graff(1,\R^2)= \{ \ell\mid\ell\textrm{ is an $1$--dimensional affine subspace of }\R^2\}\,.
$$

Let us denote by $\vec{\ell}=(\x,\ell)$
an element of $J^1(\R^2,1)$. To have an efficient description of the tangent bundle $TJ^1(\R^2,1)$, one should understand what object is the ``infinitesimal'' element $\delta\vec{\ell}$ of $J^1(\R^2,1)$. Roughly,
$\delta\vec{\ell}=(\delta\x,\delta\ell)$.
More precisely,
\begin{equation}\label{eqTangSpaceRPTR}
T_{\vec{\ell}}(\R^2\times\R\p^1)=T_{\x}\R^2\oplus T_\ell(\R\p^1)\, .
\end{equation}
A fundamental observation concerning \eqref{eqTangSpaceRPTR} is that it contains the canonical subspace
\begin{equation}\label{eqPrimaDEfCi}
T_{\x}\R^2\oplus T_\ell(\R\p^1)\supset\ell\oplus T_\ell(\R\p^1)=:\CC_{\vec{\ell}}\, .
\end{equation}
Now we can specialise Proposition \ref{prop.HpContact} to the present context.
\begin{proposition}\label{prop.J1R2cont}
The correspondence
$$
\vec{\ell}\in J^1(\R^2,1) \longmapsto \CC_{\vec{\ell}}\in T_{\vec{\ell}}J^1(\R^2,1)
$$
is a contact distribution (of dimension $2$) on $J^1(\R^2,1)$.
\end{proposition}
Before commencing the proof of Proposition \ref{prop.J1R2cont}, let us further inspect the tangent geometry of $J^1(\R^2,1)$. Let $
\x=(x,u)$
be coordinates on $\R^2$. Then
$\X(\R^2)=\Span{\partial_x,\partial_u}$, and
$
T_{\x}\R^2=\Span{\partial_x|_{\x},\partial_u|_{\x}}
$.
Accordingly,
 \begin{equation}\label{eqCorrdSbaglSuJ121}
 \vec{\ell} = (\x,\ell)=(x,u,l_1,l_2)\, ,
\end{equation}
where $\ell=(l_1,l_2)$.
Observe that \eqref{eqCorrdSbaglSuJ121} are not coordinates, for there is ambiguity in the $l_i$'s. However,
\begin{equation}\label{eqCoordOnJ1}
J^1(\R^2,1)=\overline{\{(x,u,1,u_1)\}}\stackrel{\textrm{loc.}}{\equiv} \{(x,u,u_1)\}=J^1(1,1)\, ,
\end{equation}
(see also \eqref{eq.per.scuole.medie}) where (see Fig. \ref{fig.TgAlfa})
\begin{equation}\label{eqDEfPi}
u_1:=\frac{l_2}{l_1}=\tan\alpha\, .
\end{equation}
\begin{figure}
{\epsfig{file=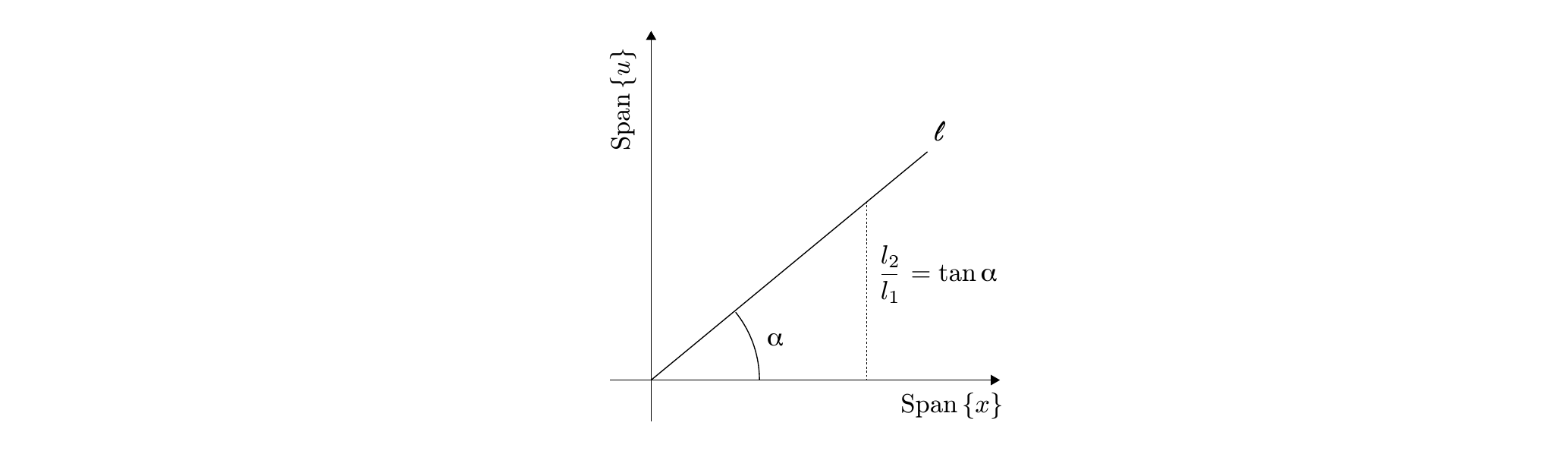,width=\textwidth}}
\caption{Save   for the vertical, all lines in $\R^2=\{(x,u)\}$ are unequivocally labeled  by $\tan\alpha$.\label{fig.TgAlfa}}
\end{figure}
Let us recall the definition of the affine neighborhood centred at $\ell$ (see \eqref{eq.affine.neig.1})
\begin{equation*}
\Hom\left(\ell,\frac{\R^2}{\ell}\right)\stackrel{\textrm{non--canonically}}{\equiv}\ell^*\otimes\ell^\perp\stackrel{\textrm{open and dense}}{ \subset }\R\p^1\, ,
\end{equation*}
and the corresponding embedding:
\begin{equation*}
\ell^*\otimes\ell^\perp\ni h\longmapsto\Graph(h):=\ell+h(\ell)\, .
\end{equation*}
(Indeed, the zero element of $\ell^*\otimes\ell^\perp$ is mapped precisely into $\ell$.)
Observe that, in particular, for the line $ \Span{x}$, one has
\begin{equation*}
T_\ell \R\p^1 \equiv T_\ell \Hom\left(\Span{x},\frac{\R^2}{\Span{x}}\right)\stackrel{\textrm{}}{\equiv}\Span{x}^*\otimes\Span{u}\, .
\end{equation*}
This means that $\ell$ is identified with a homomorphism
$
h:\Span{x}\longrightarrow\Span{u}
$,
 and then   an ``infinitesimal deformation'' $\delta\ell$  of $\ell$ is identified with \emph{another} homomorphism
$
\delta h:\Span{x}\longrightarrow\Span{u}
$,
 in the sense that $\delta h$   is tangent to the curve of homomorphisms
\begin{equation*}%
\epsilon\longmapsto h+\epsilon\delta h\equiv\Graph(h+\epsilon\delta h)=\Span{x+h(x)+\epsilon \delta h(x)}.
\end{equation*}
Observe that the identification line$\leftrightarrow$homomorphism is necessary, because $\ell+\epsilon\delta\ell$ makes no sense, whereas $h+\epsilon\delta h$ does.
\begin{remark}
 If $\ell$ belongs to the affine neighborhood
 \begin{equation}\label{eqAffNeighUno}
\Span{x}^*\otimes\Span{u}\, ,
\end{equation}
then the homomorphism $h$ corresponding to $\ell$ is precisely the number $u_1$ defined by \eqref{eqDEfPi}. Accordingly, $\delta h=\delta u_1$, and thus $\partial_{u_1}|_\ell$ is the natural generator of $T_\ell\R\p^1$.
This means that, if $f\in C^\infty(\R\p^1)$, then
\begin{equation*}
\partial_{u_1}|_\ell(f)=\left.\frac{\dd}{\dd\epsilon}\right|_{\epsilon=0}f(\Graph(h+\epsilon\delta h))\, .
\end{equation*}
\end{remark}
\begin{proof}[Proof of Proposition \ref{prop.J1R2cont}]
Directly from \eqref{eqPrimaDEfCi} it follows that
\begin{equation*}
\CC_{\vec{\ell}}=\Span{(\partial_x+u_1(\ell)\partial_u)_{\x}}\oplus\Span{\partial_{u_1}|_\ell}=\Span{(\partial_x+u_1 \partial_u)_{(\x,\ell)},\partial_{u_1}|_{(\x,\ell)}}=\Span{\partial_x+{u_1} \partial_u ,\partial_{u_1}}_{\vec{\ell}}
\end{equation*}
Let
$\pi:J^1(\R^2,1)\simeq\R^2\times\R\p^1\to\R^2$
be the projection on the first factor. Then
$d_{\vec{\ell}}\pi\,(\CC_{\vec{\ell}})=\ell$,
and actually
\begin{equation*}
\CC_{\vec{\ell}}=(d_{\vec{\ell}}\pi)^{-1}(\ell)\,,
\end{equation*}
which is coherent with Proposition \ref{prop.HpContact}.
Let now $\ell^\perp=(-u_1,1)$,
i.e.,
\begin{equation*}
\ell=\ker(du-u_1dx)_{\x}\, .
\end{equation*}
Then the covector
$\theta_{\vec{\ell}}:=(d_{\vec{\ell}}\pi)^*(du-u_1dx)_{\x}=(du-u_1dx)_{\x}\circ (d_{\vec{\ell}}\pi)$
is such that $\CC_{\vec{\ell}}=\ker \theta_{\vec{\ell}}$.
The field of covectors $\vec{\ell}\longmapsto\theta_{\vec{\ell}}$
defines the contact $1$--form $\theta $.
Obviously, $\theta=du-u_1dx$, i.e., a contact form,
so that $J^1(\R^2,1)$ is a three--dimensional contact manifold.
\end{proof}

\subsubsection{The jet space $J^1(\R^3,2)$}\label{sub.JUno21}

We introduce now the $5$--dimensional contact manifold
\begin{equation}\label{eqJayUnoDueUno}
J^1(\R^3,2):=\p T^*\R^3=\Gr(2,T\R^3)=\R^3\times\R\p^2
\end{equation}
with the $4$-dimensional contact distribution
\begin{equation}\label{eqPrimaDEfCi4D}
\CC_{\vec{H}}:=H\oplus T_H(\R\p^2)\subset T_{\x}\R^3\oplus T_H(\R\p^2)\,,
\end{equation}
at the point
$
\vec{H}:=(\x,H)=(\x,\ker\alpha)=(\x,[\alpha])
$,
with $\alpha\in T_{\x}^*\R^3$. To see that \eqref{eqPrimaDEfCi4D} is actually a contact distribution, we need the fact that the $2$--dimensional subspaces $H\subset\R^3$ are in one--to--one correspondence with the lines $H^\perp\subset\R^3$, which in turn can be identified with equivalence classes $[\alpha]\in\p(\R^3)$ of covectors on $\R^3$. We also need the
affine neighborhood
$
\Span{x^1,x^2}^*\otimes\Span{u}\subset\R\p^2
$,
identifying
\begin{equation*}
\Span{x^1,x^2}^*\otimes\Span{u}\ni h\equiv(u_1,u_2)\leftrightarrow \Graph(h)=\{ (x^1,x^2,u)\mid u=u_1x^1+u_2x^2\}\, .
\end{equation*}
Observe that
$
\Graph(h)=\Span{x^1+u_1u,x^2+u_2u}
$.
Since $\Span{x^1,x^2}^*\otimes\Span{u}$ is linear,
\begin{equation*}
T_{H}(\Span{x^1,x^2}^*\otimes\Span{u})\equiv \Span{x^1,x^2}^*\otimes\Span{u}\,.
\end{equation*}
In view of the natural projection
$\pi:\R^3\times\R\p^2\longrightarrow\R^3$,
one can define also $\CC_{\vec{H}}$ as follows:
\begin{equation*}
\CC_{\vec{H}}:=(d_{\vec{H}}\pi)^{-1}\pi(H)\, ,
\end{equation*}
which is in accordance with Proposition \ref{prop.HpContact}.
Let now $H\equiv(u_1,u_2)$, in the sense that
$
H=\Graph(h)=\Span{x^1+u_1u,x^2+u_2u}
$.
Then it is easy to see that $H=\ker\alpha$, where
$
\alpha=du-u_1dx^1-u_2dx^2
$.
To see this, just think of the matrix
\begin{equation*}
\left(\begin{array}{ccc}1 & 0 & u_1 \\0 & 1 & u_2 \\-u_1 & -u_2 & 1\end{array}\right)\, .
\end{equation*}
Then
$
\CC_{\vec{H}}=\ker(\alpha\circ d_{\vec{H}}\pi)=\ker(du-u_1dx^1-u_2dx^2)
$, i.e., the distribution $\CC$ is the kernel of a contact $1$--form.

\subsection{Homogeneous contact manifolds}

In this section we provide the general definition of a  homogeneous contact manifold (see Definition \ref{defHomContMan} below). In \ref{J^1_homog} we illustrate it by developing in detail the very special case of the three--dimensional contact manifold $J^1(\R^2,1)$. Its ``compact analog'', that is the projectivised cotangent bundle of the projective plane, is dealt with in \ref{subsubProj3d}. This example is important, because   its multi--dimensional analog will be made use of in the last  Section \ref{secVeryClear}.

For a Lie group $G$ and a closed subgroup $H\subset G$, the space $G/H$ of left cosets is called a \emph{homogeneous space}.  Any group action $G\times M\to M$, $(g, p)\mapsto g\cdot p$,   identifies the orbit $G \cdot p=\{g\cdot p: g\in G\}$ with the homogeneous space $G/G_p$, where the closed subgroup
$G_p:=\{ g\in G: g\cdot p=p\}$ is called the \emph{stabilizer}, or  \emph{isotropy subgroup}, of $p\in M$ in $G$. If $G\cdot p= M$, the action is called \emph{transitive}, and  $M$ is called a \emph{homogeneous manifold}.

Recall that $\GL_n$ denotes the general linear group of (real) invertible $n$ by $n$ matrices, which acts by linear transformations on $\mathbb{R}^n$, $(A,\textbf{x})\mapsto A\textbf{x}$. Its subgroups
$$\SL_n:=\{A\in\GL_n\vert\, \mathrm{det}(A)=1\}$$
and
$$\mathsf{O}_n:=\{A\in\GL_n\vert\, AA^t=1\}$$
are called  the \emph{special linear group} and the \emph{orthogonal group}, respectively. The quotient
$$\PGL_n=\GL_n/\left\langle\mathbb{I}_n\right\rangle$$ of the general linear group by its center (consisting of all scalar multiples of the identity matrix) is called the \emph{projective linear group}.  We shall also consider the \emph{affine group}
$$\mathsf{Aff}_n:=\GL_n\ltimes\mathbb{R}^n=\left\{ \left(\begin{array}{c|c}B & v   \\\hline 0 & 1   \end{array}\right)\vert\, B\in\GL_n, \, v\in\mathbb{R}^n\right\},$$
which acts naturally on $\mathbb{R}^n$ by affine motions $(A,\textbf{x})\mapsto B\textbf{x}+v$, and its subgroup of \emph{Euclidean motions}
$$\mathsf{E}_n:=\mathsf{O}_n\ltimes\mathbb{R}^n=\left\{ \left(\begin{array}{c|c}B & v   \\\hline 0 & 1   \end{array}\right)\vert\, B\in\mathsf{O}_n, \, v\in\mathbb{R}^n\right\}.$$

The Lie algebras of these groups will be denoted by $\mathfrak{gl}_n$, $\mathfrak{sl}_n$, $\mathfrak{o}_n$, $\mathfrak{aff}_n$, $\mathfrak{e}_n$, respectively.
\begin{example}
The spaces $\mathbb{R}^2$ and $\mathbb{R}\mathbb{P}^2$ are naturally homogeneous manifolds.
\begin{enumerate}
\item
Evidently, the natural actions of $\mathsf{E}_2$and $\mathsf{Aff}_2$ on $\mathbb{R}^2$ are transitive and we have
\begin{equation*}
\R^2=\mathsf{E}_2/\mathsf{O}_2=\mathsf{Aff}_2/\GL_2.
\end{equation*}
\item The standard action of the special linear group $\mathsf{SL}_3$ on $\mathbb{R}^3$ induces a transitive action on $\mathbb{R}\mathbb{P}^2$  and we can identify
\begin{equation*}
\mathbb{R}\mathbb{P}^2=\mathsf{SL}_3/P\, ,
\end{equation*}
where $P\subset \mathsf{SL}_3$ is the subgroup stabilizing a line $\ell\subset\mathbb{R}^3$.
\end{enumerate}
\end{example}

\begin{definition}\label{defHomContMan}
A \emph{homogeneous contact manifold} is a contact manifold $(M,\mathcal{C})$ that admits a transitive group action by contact transformations.
\end{definition}

\subsubsection{$J^1(\R^2,1)$ as a homogeneous contact manifold}\label{J^1_homog}

The following proposition shows that the action of $\Aff(2)$ on $\R^2$ lifts to a transitive action  on $J^1(\R^2,1)$ (see also Section \ref{sub.Graff1R2}). In particular, this realizes $J^1(\R^2,1)$ as a homogeneous contact manifold.

\begin{proposition}
The natural action of $\Aff_2$ on $\R^2$ lifts to a transitive action by contact transformations on $J^1(\R^2,1)$ and we have
\begin{equation*}
J^1(\R^2,1)=\R^2\times\R\p^1=\Aff_2/(\Aff_2)_{\vec{\ell}}\, ,
\end{equation*}
where the isotropy subgroup
$\Aff(2)_{\vec{\ell}}$ is isomorphic to the  subgroup
$$P=\left\{ \left(\begin{array}{cc}\ast & \ast \\0 & \ast\end{array}\right)\in\GL_2\right\}$$
of upper--triangular matrices in $\GL_2$.
\end{proposition}
\begin{proof}
Let us block--represent $A\in \Aff_2$ as
\begin{equation}\label{Aff2}
\left(\begin{array}{c|c}B & v   \\\hline 0 & 1   \end{array}\right)\, .
\end{equation}
Then the action on $J^1(\R^2,1)=\R^2\times\R\p^1$ is given by
\begin{equation}\label{eqActionOfA}
\vec{\ell}=(\x,\ell)\stackrel{A}{\longmapsto}(B\cdot\x+v, B\cdot \ell),
\end{equation}
which is evidently transitive on.
To compute the stabilizer of an element $\vec{\ell}=(\x,\ell)$ in $\Aff_2$ we note that
$A\vec{\ell}=\vec{\ell}$ if and only if $B\cdot\x+v=\x\, \textrm{ and } B\cdot \ell=\ell$.
Since the action is transitive, we can choose $\ell=\Span{(1,0)}$ and $\x=(0,0)$, so that the stabilizer is given by $A$  as in \eqref{Aff2} with $B$ of the form
\begin{equation*}
B=\left(\begin{array}{cc}\ast & \ast \\0 & \ast\end{array}\right)
\end{equation*}
and $v=0$. It follows that
\begin{equation*}
\Aff_2/(\Aff_2)_{\vec{\ell}}=
\R^2\times (\GL(2)/P)=\R^2\times\R\p^1=J^1(\R^2,1)\, .
\end{equation*}
Since the lift is natural, the lifted action preserves the contact structure on $J^1(\R^2,1)$.
\end{proof}

\begin{remark}
The affine motions  $\Aff_2$ are our first example of \emph{point transformations} (see also Definition \ref{def.point.trans} later on).
\end{remark}
Let us now describe the action of $A$ in terms of the coordinates \eqref{eqCoordOnJ1}. Directly from \eqref{eqActionOfA} it follows that
\begin{equation}\label{eqAltraAzioneDiAA}
(x,u,u_1)\stackrel{A}{\longmapsto}\left(B\cdot\x+v, \frac{b_2^1+b_2^2u_1}{b_1^1+b_1^2u_1}\right)\,.
\end{equation}
Recall that the Lie algebra of $\Aff_2$ is given by
\begin{equation*}
\aff_2=\Span{\partial_x,\partial_u,x\partial_x,x\partial_u,u\partial_x,u\partial_u}=\R^2\oplus\gll_2.
\end{equation*}
We  now wish to represent $\aff_2$ as an algebra of vector fields on $J^1(\R^2,1)$.
The ``long way''  to do that is:
\begin{equation*}
X\Rightarrow\textrm{ flow }\{A_t\}\Rightarrow\textrm{ prolongation }\{\widetilde{A}_t\}\Rightarrow\widetilde{X}:=\left.\frac{\dd}{\dd t}\right|_{t=0}\widetilde{A}_t^*.
\end{equation*}
\begin{example}
 Starting with $X=\partial_x$, we obtain
 \begin{equation*}
A_t= \left(\begin{array}{c|c}\mathrm{id} & \left(\stackrel{t}{0}\right) \\ \hline 0 & 1\end{array}\right)\,,
\end{equation*}
whence
$(x,u,u_1)\stackrel{A_t}{\longmapsto} (x+t,u,u_1)$,
and finally
\begin{equation}\label{eqEsempioDelCazzo}
\widetilde{\partial_x}=\partial_x\, .
\end{equation}
\end{example}
\begin{example}
 Starting with $X=x\partial_x$, we obtain
 \begin{equation*}
A_t= \left(\begin{array}{c|c} \exp \left(\begin{array}{cc}t & 0 \\0 & 0\end{array}\right) &0\\\hline 0 & 1\end{array}\right)\,,
\end{equation*}
whence
\begin{equation*}
(x,u,u_1)\stackrel{A_t}{\longmapsto} \left( \left(\begin{array}{cc}e^t & 0 \\0 & 1\end{array}\right)\left(\begin{array}{c}x \\u\end{array}\right)   , e^{-t}u_1\right)\,,
\end{equation*}
and, finally,
\begin{equation}
\widetilde{x\partial_x}=x\partial_x-u_1\partial_{u_1}\, .\label{eqEsempioDelCazzo2}
\end{equation}
\end{example}
There is also the ``smart way''.  First, note that for  $X\in\aff_2$ the lift $\widetilde{X}\in\X(J^1(\R^2,1))$ is projectable, i.e.,  $\widetilde{X}\stackrel{\textrm{loc.}}{=}X+f\partial_{u_1}$.
On the other hand,  the lifted field $ \widetilde{X}$ must be a symmetry of $\CC$. Indeed, we can also verify  directly from \eqref{eqAltraAzioneDiAA} that
\begin{equation*}
A^*(x) =  b_1^1x+b_1^2u\, ,\quad
A^*(u) =  b_2^1x+b_2^2u\, ,\quad
A^*(p) = \frac{ b_2^1+b_2^2u_1}{b_1^1+b_1^2u_1}\, ,
\end{equation*}
whence
\begin{equation*}
A^*(du-u_1dx)=\frac{\det B}{b_{11}+b_{21}u_1}(du-u_1dx)\, .
\end{equation*}
This allows us to determine $\widetilde{X} = X+f\partial_{u_1}$ as follows. Let
$X=g\partial_x+h\partial_u$.
Then
 \begin{equation*}
L_{\widetilde{X}}(du-u_1dx)=(h_x-u_1g_x-f)dx+(h_u-u_1g_u)du
\end{equation*}
is proportional to $du-u_1dx$ if and only if
\begin{equation*}
h_x-u_1g_x-f=-u_1(h_u-u_1g_u)\, .
\end{equation*}
\begin{conclusion}\label{conConlcusioneSollevamento}
The lift  of  $X=g\partial_x+h\partial_u $ to $J^1(\R^2,1)$ is given by $\widetilde{X} = X+f\partial_{u_1}$, where
\begin{equation*}
f=(\partial_x+u_1\partial_u)(h)-u_1(\partial_x+u_1\partial_u)(g)\, .
\end{equation*}
\end{conclusion}
\begin{example}
 In the case $X=\partial_x$ we have $g=1$ and $h=0$, whence $f=0$ (i.e., \eqref{eqEsempioDelCazzo}). In the case $X=x\partial_x$ we have $g=x$ and $h=0$, whence $f=-u_1$ (i.e., \eqref{eqEsempioDelCazzo2}).
\end{example}
\subsubsection{The homogeneous contact manifold $\mathbb{P}T^*\mathbb{R}\mathbb{P}^2$}\label{subsubProj3d}
Finally, we observe that also $\mathbb{P}T^*\mathbb{R}\mathbb{P}^2$ is naturally a homogeneous contact manifold.
\begin{proposition}\label{propPTRP2homogeneous}
The natural action of $\mathsf{SL}_3$ on $\mathbb{R}\mathbb{P}^2$ lifts to a transitive action by contact transformations on $\mathbb{P}T^*\mathbb{R}\mathbb{P}^2$. We have
\begin{equation*}
\mathbb{P}T^*\mathbb{R}\mathbb{P}^2=\mathsf{SL}_3/Q,
\end{equation*}
where
$$Q=\left\{ \left(\begin{array}{ccc} \ast & \ast&\ast  \\  0 & \ast&\ast\\0&0&\ast   \end{array}\right)  \right\}\subset \mathrm{SL}_3.$$
\end{proposition}
\begin{proof}
To see that the action is transitive, note that $\mathbb{P}T^*\mathbb{R}\mathbb{P}^2$ can be identified with the \emph{flag manifold}
$$
\left\{ (\ell, V)\,\,\vert\,\, \ell\subset V\subset\mathbb{R}^3\right\}$$ of lines $\ell$ contained in $2$--planes $V$ in $\mathbb{R}^3$.
Then also the claim about the form of the isotropy subgroup follows, since $Q$ is the stabilizer of the flag $\left\langle e_1\right\rangle\subset\left\langle e_1,e_2\right\rangle\subset\mathbb{R}^3$. Since the construction of the contact distribution on the projectivized cotangent bundle $\mathbb{P}T^*\mathbb{R}\mathbb{P}^2$ is natural (see Section \ref{sec.proj.cot.bundle}), the lifted action preserves the contact structure.
\end{proof}
\begin{remark}\label{remPGL3} Similarly, the natural action of $\mathsf{PGL}_3$ on $\mathbb{R}\mathbb{P}^2$ lifts to a transitive action by contact transformations on $\mathbb{P}T^*\mathbb{R}\mathbb{P}^2$.
\end{remark}

\section{(Systems of) PDEs of $1\St$ order}\label{sec1stOrderPDEs}
Since a contact manifold $M$, in virtue of Darboux's theorem, is locally parametrised by $(x^i,u,u_i)$, where $\theta=du-u_idx^i$ is a contact form, it is easy to convince oneself that a hypersurface $\mathcal{F}\subset M$, being cut out by a relation between those coordinates, can be interpreted as a $1\St$ order PDE imposed on the function $u$. The delicate point is to recover the dependency of the \emph{formal parameter} $u$ upon the independent variables $x^i$'s. To this end, one needs to replace the intuitive idea of a function $u=u(x^1,\ldots x^n)$ with the more geometric picture of an \emph{integral submanifold}. Locally, the latter is nothing but the graph of the former. \par
It is very important for the reader to always bear in mind that the distinction between the $x^i$'s, the $u$ and the $u_i$'s, for a general contact manifold, \emph{has no geometric content}, since, e.g.,  a Legendre transformation \eqref{eq.Legendre} may switch the roles of the former ones and the latter ones, even partially (see \eqref{eq.partial.legendre}). If the considered contact manifold has more properties, for instance it is the projectivised cotangent bundle $\mathbb{P}T^*E$ (see Section \ref{sec.proj.cot.bundle}, in particular Remark \ref{rem.Gianni.Giovanni}) then one can require that contact transformations preserve the bundle structure on $E$: this way interchanging $x^i$'s with $u_i$'s is no longer possible.
\begin{remark}
 This is the appropriate moment to stress that our choice of putting contact manifolds at the foundations of the theory of PDEs will reveal its utility only later on this these notes, when certain subtle properties of the prolongation of a contact manifold will be used. Nevertheless, the reader must always be aware that $1\St$ order jets of scalar functions constitute a particular case of a (non--compact) contact manifold (see also Remark \ref{rem.Gianni.Giovanni}) and that, conversely, every contact manifold, in view of the Darboux's theorem \ref{th.Darboux.contact}, contains a jet space as an open sub--contact manifold.
\end{remark}

\subsection{Systems and scalar PDEs of $1\St$ order}\label{sec.system.scalar.pdes}

\begin{definition}\label{def.1stOrderPDEs}
A \emph{scalar $1\St$ order partial differential equation} ($1\St$ order PDE) with one unknown function and
$n$ independent variables is a hypersurface of a $(2n+1)$--dimensional contact manifold $(M,\mathcal{C})$. A \emph{solution} of $\mathcal{F}$
is an $n$--dimensional integral submanifold of $\mathcal{C}$ contained in
$\mathcal{F}$.
\end{definition}
We have just seen (Definition \ref{def.1stOrderPDEs}) that $1\St$ order PDEs can be interpreted as hypersurfaces of contact manifolds.
In order to define  a system of $1\St$ order PDEs one should replace the contact manifold with a more general geometric object. We shall not introduce this most general geometric object here  as it goes beyond the purposes of this paper. Instead, we shall define systems of $1\St$ order PDEs as submanifolds of a Grassmannian bundle. This is a natural generalization of the projectivized cotangent bundle, a particular contact manifold, which we have already studied in detail in Section \ref{sec.proj.cot.bundle}.

To this aim, we define the bundle
\begin{equation}\label{eq.def.J1.E.n}
\pi:J^1(E^{n+m},n):=\Gr(n,TE^{n+m})=\bigcup_{p\in E^{n+m}}\Gr(n,T_pE^{n+m})\to E^{n+m}\,,
\end{equation}
where $E=E^{n+m}$ is an $(n+m)$--dimensional manifold. The manifold $J^1(E^{n+m},n)$ has a canonical distribution that generalizes the contact distribution (see \eqref{eq.defCHp}) in the case $m=1$. Recall that a point $H_p\in \Gr(n,T_pE^{n+m})$ is a $n$--dimensional subspace of $T_pE^{n+m}$. The canonical distribution on $J^1(E^{n+m},n)$ can be defined by
\begin{equation}\label{seconddef} H_p\in\Gr(n,T_pE^{n+m})\to\CC_{H_p}:=(\pi_*)^{-1}(H_p).\end{equation}
There is another way to understand this distribution. We can define it at a point $j^1S(p)\in J^1(E,n)$ (for the definition of $j^1S(p)$, see \eqref{eq.jS}),
where $S$ is an $n$--submanifold of $E$, as
\begin{multline}\label{eq.C.general}
\CC_{j^1S(p)}=\{\text{the smallest vector subspace of $T_{j^1S(p)}J^1(E,n)$ containing }
\\
\text{$T_{j^1S(p)}(im(j^1S))$ for every $n$--dimensional submanifold $S\subset E$}\}\,.
\end{multline}
It is straightforward to see that, in local coordinates \begin{equation}\label{eq.coodinates.J1.general.gianni}
(x^1,\dots,x^n,u^1,\dots,u^m,u^1_1,\dots,u^m_n)=(\x,\u,\u_{i})
\end{equation}
 of $J^1(E^{n+m},n)$, the distribution $\CC$ \eqref{seconddef} is given by
\begin{equation}\label{eq.distr.C.general.local}
\CC=\Span{ \partial_{x^i}+u^j_i\partial_{u^j}\,,\,\partial_{u^j_i} }=: \Span{ D_{x^i}\,,\,\partial_{u^j_i} }\,.
\end{equation}
\begin{definition}\label{def.Cartan.distribution}
The distribution \eqref{seconddef} is called the \emph{Cartan} distribution on $J^1(E^{n+m},n)$.
\end{definition}

\begin{definition}\label{def.system.1stOrderPDEs}
A \emph{system of $m$ $1\St$ order PDEs} with $m$ unknown functions and
$n$ independent variables is a codimension--$m$ submanifold $\mathcal{F}$ of $J^1(E,n)=\Gr(n,TE)$, where $E$ is an $(n+m)$-dimensional manifold. A \emph{solution} of $\mathcal{F}$
is an $n$--dimensional integral submanifold of $\mathcal{C}$ (defined by \eqref{seconddef}--\eqref{eq.C.general}) contained in
$\mathcal{F}$.
\end{definition}
\begin{remark}
We shall consider only systems of PDEs where the number of equations is equal to the number of unknown functions. These systems are called \emph{determined}. A motivation for considering them, among others, is that one can apply the Cauchy--Kowalewskaya theorem (see Section \ref{sezioneZarinista}).
\end{remark}
In terms of coordinates \eqref{eq.coodinates.J1.general.gianni},  a system of $m$ $1\St$ order PDEs $\mathcal{F}$  can be locally described as
$$
\F=\{f^k(\x,\u,\u_{i})=0\}\,,\quad k=1\dots m\,,
$$
with $f^k\in C^\infty(M)$. A solution $N\subset E$  parametrized by
$x^{1},\dots,x^{n}$, can be written as
\begin{equation*}
N\equiv
\begin{cases}
\displaystyle{u^j=\phi^j(x^{1},\dots,x^{n})}\, ,\,\,\,j=1\dots m\,,\\
\\
\displaystyle{u^j_i=\frac{\partial \phi^j}{\partial x^{i}}(x^{1},\dots,x^{n})}\, , \,\,\,i=1\dots n\,,\,\,\,j=1\dots m\,,
\end{cases}
\end{equation*}
where the functions $\phi^j$ satisfy
\[
f^k\left( x^{i},\phi^j,\frac{\partial\phi^j}{\partial x^{i}}\right) =0\,, \,\,\,k=1\dots m\,,
\]
which coincides with the classical notion of a solution.

\medskip

Now we specify definition \eqref{eq.def.J1.E.n} in the case that $E$ has more properties, for instance when $E=N\times M$. This will be useful later, when, for instance, we discuss some example of (systems of) ODEs and for introducing  holonomy equations (see Section \ref{sec.HolEqsJ111}) to construct jets of higher order. To this aim, we introduce, similarly to \eqref{eq.rel.equ.initial}, an equivalence relation. Given two maps $F,G:N\to M$, of class at least $C^1$, and a point $p\in N$, we define
\begin{equation}\label{eq.rel.equ.more.gen}
F\sim_p^1 G\Leftrightarrow T_{(p,F(p))} (\Graph{F})=T_{(p,G(p))} (\Graph{G})\,
\end{equation}
and we denote by $[F]_p^1$ an equivalence class. Such a notation is standard in the study of spaces of differentiable mappings \cite{MR583436}.
\begin{definition}\label{def.Michoriana}
The \emph{space of $1\St$ jets of maps} of $C^1$ functions from $N$ to $M$ is
\begin{equation}\label{eq.J1.N.M}
J^1(N,M):=\{  [F]_p^1\mid F:N\to M\}\, .
\end{equation}
The \emph{$1\St$ jet extension} of $F$ is the map
\begin{equation}\label{eq.j1.F}
j^1F:p\in N\mapsto [F]^1_p\in J^1(N,M)\,.
\end{equation}
\end{definition}
It is easy to see that $J^1(N,M)\subset J^1(N\times M,n)$ is an open subset (cf. also Remark \ref{rem.Gianni.Giovanni}). The map \eqref{eq.j1.F} can  be seen as a particular case of \eqref{eq.jS}.
A special case of \eqref{eq.J1.N.M} is when $N=\R^n$ and $M=\R^m$. In this case we shall use the notation
\begin{equation}\label{eq.J1.n.m}
J^1(n,m):=J^1(\R^n,\R^m)
\end{equation}
(which generalizes definition \eqref{eq.J1.1.1}).

\subsection{Scalar ODEs of $1\St$ order}\label{sec.Scalar.ODEs.order1}
A scalar ODE  of $1\St$ order (in one unknown function), accordingly to Definition \ref{def.1stOrderPDEs}, is a hypersurface of a $3$--dimensional contact manifold. For our purposes, we specify below this definition together with the definition of its solutions in the case when the contact manifold is $J^1(1,1)$. In fact, the space $J^1(1,1)$ (see Section \ref{sub.Graff1R2} for more details) is a  contact manifold suitable for studying local properties of scalar ODEs.
\begin{definition}\label{defDefinizioneEquatione}
A \emph{scalar ODE of $1\St$ order (in one unknown function)} is a hypersurface $\F\subset J^1(1,1)$. The $1\St$ \emph{jet prolongation} of a function $f\in C^1(\R)$ is the curve
\begin{equation}\label{eq.j1f.for.dummies}
j^1f:\R\ni x\longmapsto [f]_x^1=(x,f(x),f'(x))\in J^1(1,1)\, .
\end{equation}
A solution to $\F$ is a function $f\in C^1(\R)$ such that $j^1f$ takes its values in $\F$.
\end{definition}
Observe that $j^1f$ is of class $C^0$, and it is a section of the bundle \begin{equation}\label{eq.bundle.J11.R}
[f]^1_x\in J^1(1,1)\mapsto x\in\R\,.
\end{equation}
\begin{proposition}
 A curve $\gamma:\R\to J^1(1,1)$ is integral of the contact distribution $\CC$ if and only if it is a section of the bundle \eqref{eq.bundle.J11.R} and it is also the $1\St$ jet prolongation $j^1f$ of a function $f\in C^1(\R)$.
\end{proposition}
\begin{proof}
One way is obvious. Let now $\gamma:t\in\R\longmapsto (x(t),u(t), u_1(t))\in J^1(1,1)$.
Then, $t=x$ because $\gamma$ is a section, whence
$\gamma:x\in\R\longmapsto (x,u(x), u_1(x))\in J^1(1,1)$.
Furthermore,
$\dot{\gamma}(x)=\partial_x + u'(x)\partial_u+u_1'(x)\partial_{u_1}$
is contained into $\CC$ if and only if $(du-u_1dx)(\dot{\gamma}(x))=u'(x)-u_1(x)=0$,
i.e., if and only if $u_1=u'$, as wished.
\end{proof}
\begin{definition}\label{def.holonomic}
 We say that $\gamma$ is \emph{horizontal} if it is a section of \eqref{eq.bundle.J11.R}. We call it \emph{holonomic} if it is  the $1\St$ jet prolongation of a function $f\in C^1(\R)$, i.e., if it is of type \eqref{eq.j1f.for.dummies}.
\end{definition}
\begin{corollary}\label{corCiEffe}
 Solutions to $\F$ are in one--to--one correspondence with horizontal (with respect to \eqref{eq.bundle.J11.R}) integral curves of the restricted distribution $$\CC_\F:=\CC\cap T\F\,.$$
\end{corollary}
\subsubsection{An example of how to solve a $1\St$ order ODE}\label{sec.an.example.solve.ode}
In the particular case when the distribution $\CC_\F$ is $1$--dimensional, Corollary \ref{corCiEffe} says precisely that finding solutions amounts at finding a suitable vector field spanning  $\CC_\F$ and then integrating it. The charm of this example is that one is allowed to scale at will the vector field in order to simplify its integration.
The equation
\begin{equation}\label{eqPrimEsempio1}
(3x-2u(x))u'(x)=u(x)
\end{equation}
(see \cite[Examples 1.2 and 2.2]{MR1670044} and Fig. \ref{Fig.sol.sode}) is traditionally solved by the substitution
\begin{equation}\label{eqPrimEsempioSUB}
u=xv\, .
\end{equation}
From \eqref{eqPrimEsempioSUB} it follows that
\begin{equation}\label{eqPrimEsempio2}
xv'=u'-v\, .
\end{equation}
Replacing \eqref{eqPrimEsempio1} into \eqref{eqPrimEsempio2}, we get
\begin{equation*}
xv'=\frac{u}{3x-2u}-v\, ,
\end{equation*}
and we finally obtain a new equation in the variable $v$:
\begin{equation}\label{eqPrimEsempio3}
v'=\frac{2(1-v)v}{x(2v-3)}\, .
\end{equation}
The next step is ``to multiply by $\frac{(2v-3)}{2(1-v)v}dx$'' the equation \eqref{eqPrimEsempio3}, thus obtaining
\begin{equation}\label{eqPrimEsempio4}
\frac{2v-3}{2(1-v)v}dv=\frac{1}{x}dx\, .
\end{equation}
Then both sides of \eqref{eqPrimEsempio4} are integrated:
\begin{equation*}
\int \frac{2v-3}{2(1-v)v}dv=\int\frac{1}{x}dx\, ,
\end{equation*}
viz.
\begin{equation*}
\frac{1}{2}(\log|v-1|-3\log|v|)+c_1=\log|x|+c_2\, .
\end{equation*}
Hence, by exponentiating, one gets
\begin{equation*}
e^{2c_1}\frac{v-1}{v^3}=x^2e^{2c_2}\, ,
\end{equation*}
i.e., the implicit solution
$
e^{2c_2}v^3x^2+e^{2c_1}(v-1)=0
$
of the equation \eqref{eqPrimEsempio1}. By \eqref{eqPrimEsempioSUB} we can bring back the last solution into the form containing $u$:
$
e^{2c_2}u^3+e^{2c_1}(u-x)=0
$.
Let us now discover the geometry behind this procedure.\par
First of all, by Definition \ref{defDefinizioneEquatione}, we regard \eqref{eqPrimEsempio1} as the submanifold
$
\F=\{f=0\}\subset J^1(1,1)
$,
where
$
f=(3x-2u)u_1-u
$.
Observe that $(x,u)$ can be taken as local coordinates on $\F$, since
\begin{equation}\label{eqEsempioTRIANGOLINO}
u_1=\frac{u}{3x-2u}\,,
\end{equation}
taking into account that  $f=0$.
Hence, the restricted contact distribution $\CC_\F$ is given by the restriction $(du-u_1dx)|_\F$ of the contact form. So, in view of \eqref{eqEsempioTRIANGOLINO}, we obtain
\begin{equation*}
(du-u_1dx)|_\F=du-\frac{u}{3x-2u}dx\, .
\end{equation*}
According to the definition of solutions given in \ref{defDefinizioneEquatione}, a solution of  \eqref{eqPrimEsempio1} is an integral $1$--dimensional manifold of the line distribution
\begin{equation}\label{eqEsepioDistrRistretta}
\F\ni p \longmapsto \ker (du-u_1dx)_p\, .
\end{equation}
Suppose for a moment that $(du-u_1dx)|_\F$  is exact, i.e., that
$
(du-u_1dx)|_\F= d\varphi$, for some $\varphi\in C^\infty(\R^2)$.
Then a vector field $X\in\X(\F)$ belongs to $\CC_\F$ if and only if
$
d\varphi(X)=0\Leftrightarrow X(\varphi)=0\Leftrightarrow\dot{\gamma}(\varphi)=0$ for all trajectories $\gamma$ of $X$,
i.e., if and only if its trajectories are contained into the level surfaces of $\varphi$. But
$
\dim\varphi^{-1}(k)=1
$,
hence the trajectories coincide with the level surfaces. Unfortunately,
\begin{equation*}
d((du-u_1dx)|_\F)=d\left(du-\frac{u}{3x-2u}dx\right)=\frac{3x}{(2u-3x)^2}dx\wedge du\neq 0\, ,
\end{equation*}
so $(du-u_1dx)|_\F$ cannot be exact, not being even closed.\par
The key observation is that \eqref{eqEsepioDistrRistretta} does not change if $(du-u_1dx)|_\F$ is multiplied by any nowhere vanishing function $g\in C^\infty(\F)$.
Hence, we may look for a function $g\in C^\infty(\E)$ such that
$
g(du-u_1dx)|_\F
$
becomes closed and, hence, exact (in view of its role, the function $g$ is known as the ``integrating factor'') .\par
This, in principle, is not an easier task than solving the original equation, since
\begin{equation*}
d\large(g(du-u_1dx)|_\F\large)=\left( g_x+g_u\frac{u}{3x-u}+g\frac{3x}{(2u-3x)^2} \right)dx\wedge du=0
\end{equation*}
if and only if $g$ satisfies the $1\St$ order PDE
\begin{equation}\label{eqEsempioPDEintegrabilitaAssociata}
g_x+g_u\frac{u}{3x-u}+g\frac{3x}{(2u-3x)^2}=0\, .
\end{equation}
The above--mentioned ``trick'' to solve \eqref{eqPrimEsempio1} provides in fact a solution to \eqref{eqEsempioPDEintegrabilitaAssociata}. Indeed, rewrite \eqref{eqPrimEsempio4} as a differential form
\begin{equation*}
\eta=\frac{1}{x}dx+\frac{3-2v}{2(1-v)v}dv\in\Lambda^1(\F)
\end{equation*}
and use \eqref{eqPrimEsempioSUB} to bring it back to the coordinates $(x,u)$:
\begin{equation*}
\eta=\frac{1}{2(u-x)}dx+\frac{3x-2u}{2u(u-x)}du\, .
\end{equation*}
It is immediate to see that $(du-u_1dx)|_\F$, multiplied by the function
\begin{equation*}
g:=\frac{2u-3x}{2u(u-x)}\, ,
\end{equation*}
gives precisely $\eta$, and moreover, that $g$ is a solution of \eqref{eqEsempioPDEintegrabilitaAssociata}.
Observe that, written in terms of $v$,
\begin{equation*}
g=\frac{2v-3}{2(v-1)v}\, ,
\end{equation*}
i.e., precisely the ``magic factor'' we multiplied with the equation \eqref{eqPrimEsempio3}.\par
\begin{figure}
 \centerline{\epsfig{file=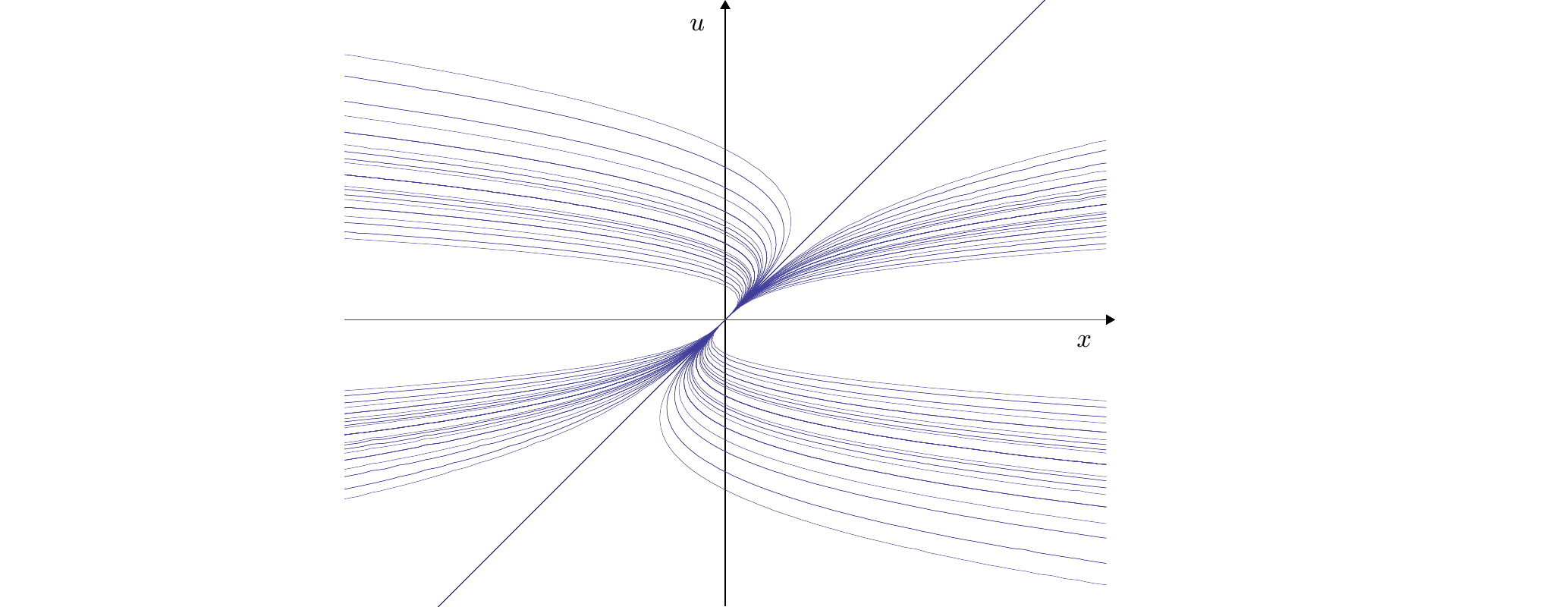,width=\textwidth}}
 \caption{Solutions of the equation $(3x-2u(x))u'(x)=u(x)$. \label{Fig.sol.sode}}
\end{figure}

\subsection{Systems of $1\St$ order ODEs in several unknown functions: trajectories of vector fields}\label{secVectFields}
Let us pass now to systems of $1\St$ order ODEs. We consider the $(2m+1)$--dimensional space
$
J^1(1,m)=\{(t,u^1,\ldots,u^m,u^1_t,\ldots,u^m_t)\}
$ (see definition \eqref{eq.J1.n.m}). Here we denote the independent variable by ``$t$'' as it will represent the external parameter on curves, that in the present case turn out to be trajectories of a given vector field.
Even if this example does not fit the general framework of contact manifolds (in fact $J^1(1,m)$ is  a contact manifold only for $m=1$), we develop it nonetheless, due to its importance in describing trajectories of vector fields. The Cartan distribution $\CC$ (see Definition \ref{def.Cartan.distribution}) is $(m+1)$--dimensional and it is spanned by (see formula \eqref{eq.distr.C.general.local})
$$
\CC=\Span{D_t, V_i}\,,\quad D_t=\partial_t+u^i_t\partial_{u^i}\,,\quad  V_i=\partial_{u^i_t}\,.
$$
All definitions contained in Definition \ref{defDefinizioneEquatione} can be easily generalized to the present case, i.e., to systems of ODEs, by using \eqref{eq.J1.N.M}--\eqref{eq.J1.n.m}, but we shall not do it to not overload the notation.

\smallskip
We observe that the natural correspondence
 \begin{eqnarray}
J^1(1,m)\ni [\gamma]_t^1 &\longmapsto & (t, \dot{\gamma}(t))\in  \R\times T\R^m\label{eqIdJ1UnoEnneTM}
\end{eqnarray}
is a diffeomorphism.
In order to lower the dimension of $\CC$ to 1, it is not sufficient to intersect $\CC$ with a hypersurface $\F=\{f=0\}$, but rather with the intersection of $m$ hypersurfaces, i.e., a \emph{system} of $m$ equations
\begin{equation}\label{eqSistOfEqs}
\F:=\{ f^1=0\, ,\ldots, f^m=0\}\, ,
\end{equation}
such that at any point $p\in\F$, the differentials $d_pf^1$, \ldots, $d_p f^m$ are linearly independent. One can introduce the vector--valued function $\vec{f}:=(f^1,\ldots, f^m)$,
and rewrite \eqref{eqSistOfEqs} as $
\F=\{\vec{f}=0\}\subset J^1(1,m)$.
\begin{example}
 The $m$ functions $
f^i=g^i(u^1,\ldots, u^m)-u^i_t
$
are such that their differentials $df^i$ are linearly independent and, hence, $\dim\F=m+1$. Moreover, $\vec{f}$ does not depend on $t$, so $\F=\R\times\F_0$, where $\F_0$ is the ($m$--dimensional) submanifold of $T\R^m$ determined by the same equations $\vec{f}=0$. Hence, in view of identification \eqref{eqIdJ1UnoEnneTM}, the solutions of $\F$ are nothing but curves $\gamma(t)=(\gamma^1(t),\ldots,\gamma^m(t))$ such that their image is contained into $\F_0$, i.e., such that
\begin{equation*}
g^i(\gamma^1(t),\ldots,\gamma^m(t))=\frac{\dd \gamma^i}{\dd t}(t)\, ,\quad \forall i=1,\ldots, m\, .
\end{equation*}
In other words, solutions of $\F$ are trajectories of the vector field
$
X=
g^i(u^1,\ldots, u^m)\frac{\partial}{\partial u^i}
$
on $\R^m=\{  (u^1,\ldots, u^m)\}$.
\end{example}

\subsection{The holonomy equations in $J^1(J^1)$}\label{sec.HolEqsJ111}
Besides the equation of the trajectories of a vector field, examined in Section \ref{secVectFields} above, there is another important example of a system of $1\St$ order ODEs in several unknowns, which is worth discussing. Indeed, such a \emph{canonical} system provides a way   to  introduce the $2\Nd$ order jet space $J^2(1,1)$  of functions $f:\R\to\R$ and hence it has some pedagogical value. A more formal  definition of $J^2(1,1)$ will be given later on (see  \eqref{eqDefJayDueEnneUno}).\par The bottom line here is to interpret $J^2(1,1)$  as an ``iterated jet space'', where    suitable holonomy conditions have been imposed (see  \cite{MorenoCauchy} for a systematic treatment of iterated jet spaces). These conditions take the form of ``tautological'' canonical $1\St$ order PDEs.\par
Let us consider \eqref{eq.J1.N.M} with  $N=\R$ and $M=J^1(1,1)=\{x,u,u_1\}$.  Then
$
J^1(\R, J^1(1,1))=\{ (t,x,u,u_1,\dot{x},\dot{u},\dot{u_1})\}
$
is the $7$-dimensional space of the $1\St$ order jets of function from $\R$ to $J^1(1,1)$ (see Definition \ref{eq.J1.N.M}), being $t$ the independent variable of $J^1(\R, J^1(1,1))$ and $\dot{x},\dot{u},\dot{u_1}$ the fibre coordinates of the bundle $J^1(\R, J^1(1,1))\to\R\times J^1(1,1)$.

\smallskip\noindent
There are three canonical differential equations in $J^1(\R, J^1(1,1))$, listed below:
\begin{enumerate}
\item the ``horizontality'' equation, i.e., $x=t$;
\item the consequence of the above, i.e., $\dot{x}=1$;
\item the ``holonomy'' equation, i.e., $\dot{u}=u_1$.
\end{enumerate}
Let us denote by $\E\subset J^1(\R, J^1(1,1))$ the intersection of these three equations. Then,
\begin{equation}\label{eqPrimaVoltaCompaionoDerivateSeconde}
\E=\{(x,u,u_1,u_{11})\}\, ,\quad\textrm{ where }u_{11}:=\dot{u_1}.
\end{equation}
In view of the definition of solutions given in Definition \ref{defDefinizioneEquatione}, a solution to $\E$ is a map
\begin{equation}\label{eqCurvaInJ1J111}
\R\ni t\mapsto (t,x(t),u(t), u_1(t), \dot{x}(t),\dot{u}(t),\dot{u_1}(t))\in J^1(\R, J^1(1,1))
\end{equation}
such that
\begin{eqnarray}
\dot{x}(t) &=&\frac{\dd x}{\dd t}(t)\, ,\label{eqOlonomiaDelCacchio1}\\
\dot{u}(t) &=&\frac{\dd u}{\dd t}(t)\, ,\label{eqOlonomiaDelCacchio2}\\
\dot{u_1}(t) &=&\frac{\dd u_1}{\dd t}(t)\, ,\label{eqOlonomiaDelCacchio3}
\end{eqnarray}
and, in addition the three conditions above are satisfied. Then \eqref{eqCurvaInJ1J111} reads
\begin{equation*}
x\mapsto \left(x,u(x), \frac{\dd u}{\dd t}(x), 1, \frac{\dd u}{\dd t}(x), \frac{\dd^2 u}{\dd t^2}(x)\right)\, ,
\end{equation*}
i.e., in view of \eqref{eqPrimaVoltaCompaionoDerivateSeconde},
\begin{equation}\label{eqFunzioneBrutalmenteSollevata}
x\mapsto \left(x,u(x), \frac{\dd u}{\dd t}(x),   \frac{\dd^2 u}{\dd t^2}(x)\right)\, .
\end{equation}
\begin{remark}
 The construction of $\E$ just presented is a possible way to introduce the $2\Nd$ order jet space $J^2(1,1)$. That is, we used the old trick of ``regarding $2\Nd$ order equations as systems of $1\St$ order equations.''
\end{remark}
Let us try to realise $J^2(1,1)$ as a geometric construction associated with $J^1(1,1)$.
To this end, let
$
\gamma(x)=(x,u(x),u'(x))
$
be a holonomic horizontal curve in $J^1(1,1)$ (see Definition \ref{def.holonomic}) and observe that, by definition,
\begin{equation}\label{eqDefSpanDotGAmmaBUONA}
\Span{\dot{\gamma}(x)}=\Span{\partial_x+u'(x)\partial_u+u''(x)\partial_{u_1}}\in\p\CC_{\gamma(x)}\, .
\end{equation}
This suggests to define
\begin{equation}\label{eqPrimaDefinizioneDiJayUnoUnoUno}
J^1(1,1)^{(1)}:=\p\CC=\coprod_{\vec{\ell}\in J^1(1,1)}\p\CC_{\vec{\ell}}
\end{equation}
 which is a bundle over $J^1(1,1)$ with abstract fibre $\R\p^1$.
\begin{definition}\label{defPrimoCasoDiProlungamento}
 $J^1(1,1)^{(1)}$ is the \emph{prolongation} of the contact manifold $(J^1(1,1),\CC)$.
\end{definition}
A more abstract definition of  prolongation of any submanifold of a contact manifold will be given in Section \ref{secPrologContMan}. Now we use the affine neighborhood (see \eqref{eq.affine.neig.1})
\begin{equation*}
\Span{\partial_x+u_1\partial_u}^*\otimes\Span{\partial_{u_1}}\subset \p\CC_{\vec{\ell}}
\end{equation*}
in order to introduce the new coordinate $
u_{11}\leftrightarrow \Span{\partial_x+u_1\partial_u+u_{11}\partial_{u_1}}$.
Then, in particular, \eqref{eqDefSpanDotGAmmaBUONA}
reads
$
(x,u(x),u'(x),u''(x))
$,
i.e.,
\begin{equation*}
\R\ni x\longmapsto \Span{\dot{\gamma}(x)}\in J^1(1,1)^{(1)}\, ,
\end{equation*}
which is precisely the map \eqref{eqFunzioneBrutalmenteSollevata} in the coordinates $(x,u,u_1,u_{11})$ of $J^1(1,1)^{(1)}$.
\begin{remark}
From a local  standpoint,  one  can choose among three ways to define the same thing: the canonical equation $\E$, the prolongation $J^1(1,1)^{(1)}$, or the ``classical'' definition of $J^2(1,1)$, given by \eqref{eqDefJayDueEnneUno} later on.
\end{remark}
The three holonomy equations \eqref{eqOlonomiaDelCacchio1}, \eqref{eqOlonomiaDelCacchio2} and \eqref{eqOlonomiaDelCacchio3} are usually written in terms of ``contact forms'', viz.
\begin{equation*}
dx-\dot{x}dt=0\, , \quad du-\dot{u}dt=0\, , \quad du_1-\dot{u_1}dt=0\, ,
\end{equation*}
which, restricted to $\E\equiv J^2(1,1)$, yields
\begin{equation}\label{eqContForm2ndorder1}
du-u_1dx=0\, ,\quad
du_1-u_{11}dx=0\,,
\end{equation}
i.e., a system of two ``contact forms''.\par
The definition of a $2\Nd$ order ODE and its solutions are the obvious analogues of those contained in Definition \ref{defDefinizioneEquatione}. They will be given in Section \ref{sec.SecondOrderPDEs} later on.
\section{The prolongation of a contact manifold}\label{secPrologContMan}
\subsection{The Lagrangian Grassmannian $\LL(n,2n)$ and the prolongation $M^{(1)}$}\label{sec.LagGrassPrologM}

We wish to define $J^2(2,1)$ in an analogous way as we defined $J^2(1,1)$ by \eqref{eqPrimaDefinizioneDiJayUnoUnoUno}. We invite the reader to carry out the same construction in the case of two independent variables and compute the resulting dimension: one gets 9  instead of 8. The latter is the predicted dimension for the space of $2\Nd$ order jets of scalar functions in two variables. Such a discrepancy suggests that  \eqref{eqPrimaDefinizioneDiJayUnoUnoUno} is not the correct general definition of a prolongation, and the purpose of this section is precisely that of rectifying it.\par
We introduce here the notion of prolongation of a contact manifold, generalizing the Definition \ref{defPrimoCasoDiProlungamento} given above---the one which allowed us to correctly define $J^1(1,1)^{(1)}$  through \eqref{eqPrimaDefinizioneDiJayUnoUnoUno}.
This way, we will be able to introduce $J^2(n,1)$ for any value of $n$, as the prolongation of the contact manifold $J^1(n,1)$.\par

Let $(M,\CC)$ be a $(2n+1)$--dimensional contact manifold. By recalling the definition \eqref{eq.def.J1.E.n}, we consider the Grassmann bundle of $n$--planes in $TM$
\begin{equation*}
\Gr(n,TM)=\bigcup_{p\in M}\Gr(n,T_pM)\to M\,.
\end{equation*}
By definition,  $\Gr(n,TM)$ is a natural bundle over $M$, whose rank equals the dimension of $\Gr(n,2n+1)$, which is $n(n+1)$.  Inside $\Gr(n,TM)$ there is the obvious sub--bundle
\begin{equation*}
\Gr(n,\CC):=\bigcup_{p\in M}\Gr(n,\CC_p)
\end{equation*}
of $n$--dimensional tangent planes to $M$ contained in the contact distribution $\CC$. Its rank is $n^2$. This bundle is not yet the correct place to study $2\Nd$ order PDEs in $n$ independent variables, since---loosely speaking---the new fiber coordinates $p_{ij}$ are not symmetric in the indices $(i,j)$.   In order to define the prolongation of $M$, that is the environment for $2\Nd$ order PDEs, we need one last   reduction of the fibres of $\Gr(n,\CC_p)$. This is possible thanks to the (conformal) symplectic structure on the contact hyperplanes (see Remark \ref{rem.ConfSympSpace}). For a linear symplectic space $(W,\omega)$, with $\dim W=2n$, the \emph{Lagrangian Grassmannian} is the subset
\begin{equation}\label{eq.DefLGLineare}
\LL(W):=\{ L\in\Gr(n,W)\mid \omega|_L\equiv 0\}\,
\end{equation}
of the Grassmannian manifold $\Gr(n,W)$ of $n$--planes in $W$. Observe that \eqref{eq.DefLGLineare} is invariant under a rescaling of $\omega$, so that it would be more appropriate to say that $\LL(W)$ is associated to the \emph{conformal} symplectic space $(W,[\omega])$. Indeed, each $\CC_p$ is a conformal symplectic space (see Remark \ref{rem.ConfSympSpace}).\par
In order to realise $\LL(W)$ as a smooth projective manifold, it suffices to convince oneself that, for any $n$--plane $L\subset W$, the notion of the \emph{volume} of $L$, that is,
\begin{equation*}
\vol (L):=[l_1\wedge\cdots\wedge l_n]\in\p\Lambda^n(W)\, ,\quad \Span{l_1,\ldots, l_n}=L\, ,
\end{equation*}
is well--defined. Then one can embed $\LL(W)$ into $\p\Lambda^n(W)$, just by sending each $L\in \LL(W)$ into its volume $\vol (L)\in \p\Lambda^n(W)$. Remarkably, $\p\Lambda^n(W)$ is not the smallest projective space $\LL(W)$ is embeddable in.
\begin{definition}\label{defPluck}
 The projective subspace $\p\Lambda_0^n(W)\subset \p\Lambda^n(W)$, where
 \begin{equation*}
\Lambda_0^n(W):=\{ \alpha\in \Lambda^n(W)\mid \omega\lrcorner\alpha=0\}
\end{equation*}
is called the \emph{Pl\"ucker embedding space} for $\LL(W)$. The map
\begin{equation*}
\LL(W)\ni L\longmapsto \vol(L)\in \p \Lambda_0^n(W)
\end{equation*}
is called the \emph{Pl\"ucker embedding}.
\end{definition}
The Pl\"ucker embedding realises the so--called \emph{minimal projective embedding}, that is, it is the simplest way to realise $\LL(W)$ as a projective variety.
\begin{definition}
 The $\left((2n+1)+\frac{n(n+1)}{2}\right)$--dimensional manifold
 \begin{equation*}
\underset{\rank=\frac{n(n+1)}{2}}{\underbrace{M^{(1)}}}=\LL(\CC):=\bigcup_{p\in M}\LL(\CC_p)\subset \underset{\rank=n^2}{\underbrace{\Gr(n,\CC)}}\subset \underset{\rank=n^2+n}{\underbrace{\Gr(n,TM)}}
\end{equation*}
  is called the \emph{first prolongation} (or \emph{Lagrangian Grassmannian bundle}) over $M$.
\end{definition}
Let $N\subseteq M$ be a Lagrangian submanifold. Then, by definition, $T_pN$ is Lagrangian in $\CC_p$ for all $p\in N$, that is, a point of $M^{(1)}$ projecting over $p$.

\begin{remark}[Coordinates on $M^{(1)}$]\label{rem.coordinates.M1}
Contact coordinates $(x^i,u,u_i)$ on $M$ induce coordinates $(x^i,u,u_i,u_{ij})$ (that we keep calling ``contact'') on $M^{(1)}$.
Recall that the $D_{x^i}$'s and the $\partial_{u_i}$'s span (locally) the distribution $\CC$ (see \eqref{eq.contact.local}).
Then one observes that $\Gr(\CC_p)$ contains the (open) affine neighborhood
\begin{equation}\label{eqAffNeighCiPi}
\Span{D_{x^i}\mid i=1,\ldots,n}^*\otimes \Span{\partial_{u_i}\mid i=1,\ldots,n}\, ,
\end{equation}
which can be parametrised by $n\times n$ matrices $u_{ij}$. Indeed, \eqref{eqAffNeighCiPi} is a particular case of  \eqref{eq.affine.neig.2}, with $V=\CC_p$ and $L=\Span{\partial_{u_i}\mid i=1,\ldots,n}$, and the $n$--plane
\begin{equation*}
\Span{D_{x^i}+u_{ij}\partial_{u_j}\mid i=1,\ldots,n}
\end{equation*}
is Lagrangian if and only if $u_{ij}$ is symmetric.
In other words, in the affine neighborhood \eqref{eqAffNeighCiPi},  the subset $\LL(\CC_p)$ of $\Gr(\CC_p)$ corresponds precisely to \emph{symmetric} $n\times n$ matrices $u_{ij}$.
\end{remark}

\begin{definition}\label{eq.def.prolog.lag.submanifold}
We define the prolongation
$V^{(1)}\subset\LL(W)$ of a subspace $V \subset W$ by:
\begin{equation}\label{eq.prolongation.of.V}
V^{(1)}:=\left\{
\begin{array}{c}
L\in \LL(W)\,\,|\,\, L\supseteq V,\,\,\text{if}\,\,\dim(V)\leq n \,, \\
\\
L\in \LL(W)\,\,|\,\, L\subseteq V,\,\,\text{if}\,\,\dim(V)\geq n\,. \\
\end{array}
\right.
\end{equation}
The prolongation $N^{(1)}$ of an integral submanifold $N$ of $\CC$ is by definition equal to $\bigcup_{p\in N}(T_pN)^{(1)}$. In particular,
the \emph{prolongation} of the Lagrangian submanifold $N$ is the tangent bundle $TN$, understood as an $n$--dimensional submanifold of $M^{(1)}$.
\end{definition}
In order to provide a simple example, let us consider the $(2n+1)$--dimensional manifold $J^1(n,1)$ (see definition \eqref{eq.J1.n.m}). Similarly as we did in \eqref{eq.rel.equ.initial}--\eqref{eq.J1.1.1} (see also \eqref{eq.rel.equ.more.gen}--\eqref{eq.J1.n.m}), we introduce $J^2(n,1)$. Let $f,g\in C^2(\R^n)$. Then
\begin{equation*}
f\sim_{\x}^2 g\Leftrightarrow f(\x)=g(\x)\,, \,\, (df)_{\x}=(dg)_{\x}\,, \,\, (\mathrm{Hess}\,f)_{\x}=(\mathrm{Hess}\,g)_{\x}\,,
\end{equation*}
i.e., functions $f$ and $g$ are equivalent at $\x$ if their Taylor expansions  is equal up to second order, i.e., they have a contact of order $2$ at $\x$. As usual, we denote by $[f]^2_{\x}$ an equivalence class.
\begin{definition}
The space $J^2(n,1)$ is the union
 \begin{equation}\label{eqDefJayDueEnneUno}
J^2(n,1):=\{[f]^2_{\x}\,\,|\,\,f\in C^2(\R^n)\}
\end{equation}
of the above--introduced equivalence classes.
\end{definition}
Observe that the map
\begin{equation*}
J^2(n,1)\ni [f]_{\x}^2  \,\,\longmapsto\,\, T_{[f]_{\x}^1}(\Graph j^1f)\in(J^1(n,1))^{(1)}
\end{equation*}
is injective, and its image is open and dense. Hence,   $(J^1(n,1))^{(1)}$ can be safely (at least locally) identified with the space $J^2(n,1)$ of $2\Nd$ jets of maps from $\R^n$ to $\R$. Since any contact manifold $M$, of dimension $(2n+1)$ is locally of the form $J^1(n,1)$ thanks to Theorem \ref{th.Darboux.contact}, it follows that  one can develop the theory of $2\Nd$ order scalar PDEs in $n$ independent variables, by studying hypersurfaces $\E\subset M^{(1)}$.\par

By recalling the local description of the contact distribution $\CC$ (see for instance \eqref{eq.contact.local}) one can say that $\CC$ consists of two parts: a ``horizontal'' (spanned by $D_{x^i}$) and a ``vertical'' one (spanned by $\partial_{u_i}$) of the same dimension. Generally, these parts are not canonically defined (see the discussion at the beginning of Section \ref{sec1stOrderPDEs}). If the considered contact manifold has some extra structure (for instance a bundle structure), then the vertical part of the contact distribution is canonical (see Remark \ref{rem.Gianni.Giovanni}).
Anyway, both of these parts are Lagrangian with respect to the symplectic form on $\CC$. As it turns out, the datum of the horizontal and the vertical part is sufficient to characterise, up to a conformal factor, the symplectic form on $\CC$.
\begin{proposition}\label{propCRnRnStar}
Up to a non--zero factor, there exists  a unique symplectic form $\omega_0$ on the space $\R^n\oplus\R^{n\ast}$, such that both  $\R^{n}$ and $\R^{n\ast}$ are Lagrangian with respect to $\omega_0$, and $\omega_0(v,\alpha)=\alpha(v)$.
Moreover, for any symplectic vector space $(W,\omega)$, there exists an isomorphism $W\simeq \R^n\oplus\R^{n\ast}$, which pulls back $\omega_0$
to a non--zero multiple of $\omega$.
\end{proposition}
\begin{proof}
 Follows from the uniqueness of a symplectic matrix in dimension $2n$.
\end{proof}
\begin{definition}\label{DefDopo_propCRnRnStar}
 The isomorphism $W\simeq \R^n\oplus\R^{n\ast}$ above is a \emph{bi--Lagrangian decomposition} of $(W,\omega)$.
\end{definition}
In view of the above Proposition \ref{propCRnRnStar}, from now on we use the symbol $\LL(n,2n)$ instead of $\LL(W)$, if there is no need to specify the symplectic space $(W,\omega)$.
\begin{remark}
 Denote by $\mathcal{V}$ the vertical distribution on $J^1(n,1)\to J^0(n,1)=\R^{n+1}$. Then $\mathcal{V}$ is an $n$--dimensional (completely integrable) distribution contained in $\CC$. Let $H\in J^1(n,1)$ be a point, understood as a tangent $n$--plane to $\R^{n+1}$. The correspondence
 \begin{eqnarray*}
\mathcal{V}_H\ni \xi &\longrightarrow & \omega_H(\xi,\, \cdot\, )\in H^*
\end{eqnarray*}
is an isomorphism. Hence, in view of Proposition \ref{propCRnRnStar},
we can assume
$
\CC_H\simeq H\oplus H^*
$.
\end{remark}
We generalise now the construction of the   affine neighborhood \eqref{eqAffNeighUno} to arbitrary values of $n$.
\begin{proposition}\label{prop.ApAffLGR}
The image of the map
 \begin{eqnarray}
\R^{n\ast}\otimes \R^{n\ast}\ni h  &\longrightarrow & \Graph h\in  \Gr(n,W)\, ,\label{eq.ApAffLGRn}
\end{eqnarray}
is an open and dense subset, such that
$
\omega|_{\Graph h}\equiv 0$ $\Leftrightarrow$ $h$ is symmetric, with  $W\simeq \R^n\oplus\R^{n\ast}$.
\end{proposition}
\begin{proof}
Fix the basis
$e_1,\ldots, e_n\in\R^n$
and its   dual
$\epsilon^1,\ldots, \epsilon^n\in\R^{n\ast}$,
and observe that
$\omega=e_i\wedge\epsilon^i$.
The elements $L\in \Gr(n,W)$ non--trivially projecting on $\R^n$ form an open and dense subset, and, for each of them, there is an $n\times n$ matrix $U=\|u_{ij}\|$, such that $L=\Span{ e_i+u_{ij}\epsilon^j \mid i=1,2,\ldots, n }$.
In other words, $L=\Graph h$, where the matrix of $h$ is precisely $U$.
Finally,
$\omega|_{\Graph h}\equiv 0$ if and only if $\omega( e_i+u_{ij}\epsilon^j, e_{i'}+u_{i'j'}\epsilon^{j'})=u_{i'i}-u_{ii'}=0$ for all $  i,i'$.
\end{proof}
Recall the definition  \eqref{eq.DefLGLineare} of $\LL(W)$. Proposition \ref{prop.ApAffLGR} provides an alternative definition of the Lagrangian Grassmannian.
\begin{corollary}\label{cor.AltDefLGR}
 The closure in $ \Gr(n,W)$ of the subset $S^2\R^{n\ast} $ of $\R^{n\ast}\otimes \R^{n\ast}$ is the  {Lagrangian Grassmannian} $\LL(W)$.
\end{corollary}
\begin{corollary}\label{cor.CanIdentTanLGR}
 There is a local natural identification
 \begin{eqnarray}\label{eq.giovanni.fundamental}
 T(S^2\R^{n\ast})    &\longrightarrow &  T\LL(n,2n) \nonumber \, ,\\
 (h,\delta h) &\longmapsto & \left.\frac{\dd}{\dd\epsilon}\right|_{\epsilon=0}\Graph(h+\epsilon\delta h)\, .
\end{eqnarray}
\end{corollary}
\begin{proof}
 Just observe that the map \eqref{eq.ApAffLGRn} restricts to a map from the space of symmetric matrices to the Lagrangian Grassmannian, whose image is open. As such, it induces an isomorphisms between the tangent space at $h$ to the former and the tangent space at $\Graph h$ to the latter.
\end{proof}
 In other words, at the point $h\equiv L$, the tangent space $T_L\LL(n,2n)$ inherits the identification with $S^2\R^{n\ast}$, due to the fact that $T_hS^2\R^{n\ast}$ identifies with $S^2\R^{n\ast}$ (in this perspective, the above symbol $\delta h$ denotes an elements of $S^2\R^{n\ast}$, understood as a tangent vector at $h$ to $S^2\R^{n\ast}$ itself). Finally, since $L$ canonically identifies with $\R^n$, via the projection, one can says also that
$
T_L\LL(n,2n)\equiv S^2L^*
$.
This motivates the  introduction of  a canonical bundle over $\LL(n,2n)$.
\begin{definition}
 The bundle $\mathcal{L}\to\LL(n,2n)$ defined by
$
\mathcal{L}^{-1}(L):=L
$
is called the \emph{tautological bundle} over $\LL(n,2n)$.
\end{definition}
Corollary \ref{cor.CanIdentTanLGR} immediately leads to the next fundamental theorem. We shall give an alternative proof of it. Essentially we will go on the other way around, i.e., we shall construct a  bundle isomorphism which may be thought of as a ``global inverse'' (i.e., not defined only on an affine neighborhood)
of the map \eqref{eq.giovanni.fundamental}.
\begin{theorem}
 There is a canonical bundle isomorphism
 \begin{equation}\label{eqCanBundIso}
T\LL(n,2n)\equiv S^2\mathcal{L}^*\, .
\end{equation}
\end{theorem}
\begin{proof}
Let $W$ a $2n$--symplectic space and $\LL(n,2n)=\LL(W)$.
Let $\dot{L}_{0} \in T_{L_{0}}\LL(n,2n)$ and
$\phi_{t}$ be an  $1$--parameter  subgroup of $\Sp(W)$, the symplectic group of $W$, such that $\dot
L_0 =  \left.\frac{d\phi_t(L_0)}{dt}\right\vert_{t=0} $. We define a symmetric bilinear
form $g^{\dot L_0}$ on $L_0$ by
\begin{equation*}
g^{\dot{L}_{0}}(v,w)\overset{\text{def}}{=}\omega\left(  \left.  \frac
{d\phi_{t}(v)}{dt}\right\vert _{t=0},w\right)  ,~~v,w\in L_{0}.
\end{equation*}
It does not depend on the $1$--parametric  group $\phi_t$ whose orbit at $L_0$ has
tangent vector  $\dot L_0$. Indeed,  any other  such $1$--parameter
group can be written as $\widetilde{\phi}_t = \phi_t \circ h_t + o(t)$ where
$h_t $ belongs  to the  stabilizer in $\Sp(V)$  of the point
$L_0$. Since $\widetilde{\phi}(v)= \phi_t(h_t(v))= \phi(t,h_t(v))$, we have that
$$
\left.\frac{d\widetilde{\phi}_{t}(v)}{dt}\right\vert _{t=0}=
\left.\frac{d\phi_{t}(v)}{dt}\right\vert _{t=0} + \left.\frac{dh_{t}(v)}{dt}\right\vert
_{t=0}
$$
and
$$
\omega\left(  \left.  \frac {d\widetilde{\phi}_{t}(v)}{dt}\right\vert
_{t=0},w\right) = \omega\left(  \left.  \frac
{d\phi_{t}(v)}{dt}\right\vert _{t=0},w\right)\,,
$$
since $\omega|_{L_0} =0$.
\end{proof}
In other words, tangent vectors to $\LL(n,2n)$ can be understood as symmetric $n\times n$ matrices $\delta U$. Also  the elements of $\LL(n,2n)$  themselves can be (locally, in an affine neighborhood) identified with  symmetric $n\times n$ matrices  $U$.
So, the correct way to understand
$
U+\delta U
$
as an ``affine tangent vector'' to $\LL(n,2n)$, is to think of it as the velocity at zero of the curve
\begin{equation*}
\R\ni\epsilon\longmapsto U+\epsilon\delta U\in \LL(n,2n)\, .
\end{equation*}
We conclude this section with a   definition which will be of paramount importance in   the study of characteristics of $2\Nd$ order PDEs (see Section \ref{secAntipatetica}).
\begin{definition}\label{def.rank.vector}
A tangent vector $U+\delta U$ is called \emph{rank--$k$} if and only if  $\rank\delta U=k$.
\end{definition}
Observe that the definition is well--posed in view of the canonical isomorphism \eqref{eqCanBundIso},
whereas $\rank U$ is not an intrinsic concept, as $U$ can have maximal rank in one coordinate chart, and be zero into another.

Rank--$1$ tensors in $S^2L^*$ are, up to sign, of type
\begin{equation}\label{eq.rank.1.vector}
\xi\odot\xi\,,\quad\xi\in L^*\,.
\end{equation}
So, rank--$1$ vectors, via identification \eqref{eqCanBundIso}, are $\sum_{i\leq j}\xi_i\xi_j\partial_{u_{ij}}$, where $\xi_k$ are the component of the covector $\xi$ in the basis $\{dx^i\}$ of $L^*$.
The annihilator $\Ann(\xi)$ of the covector $\xi$ is of course a hyperplane of $L$. The set of all Lagrangian subspaces containing $\Ann(\xi)$, i.e., $(\Ann(\xi))^{(1)}$ (see \eqref{eq.prolongation.of.V}), is a curve $L_t$, $t\in\R$, in the Lagrangian Grassmannian starting at $L_0=L$: the tangent vector at $L$ to this curve is exactly the aforementioned vector of rank $1$. We have thus established a correspondence
\begin{equation}\label{eq.correspondence.gianni}
\text{hyperplanes of $L$ (which correspond  to  elements of
$\mathbb{P}L^*$)$\,\,\,\Longleftrightarrow\,\,\,$ directions of
$T_L\LL(n,2n)$ of rank $1$}\,.
\end{equation}

\subsection{Geometry of $\LL(2,4)$}\label{sec.L24}
We specialise now to the case $n=2$.
Let us fix the dual to each other bases
$
\R^2=\Span{e_1,e_2}
$
and
$
\R^{2\,\ast}=\Span{\epsilon^1,\epsilon^2}
$.
Recall (see Proposition \ref{prop.ApAffLGR} and Corollary \ref{cor.AltDefLGR}) that $\LL(2,4)$ contains $S^2\R^{2\ast}$ as a preferred affine neighborhood. More precisely, this neighborhood is the image of the   map \eqref{eq.ApAffLGRn}, which in the case $n=2$ looks like
\begin{eqnarray}\label{eq.L.P.as.a.matrix}
S^2\R^{2\ast}  &\longrightarrow & \LL(2,4)\,,\nonumber \\
U=\left(\begin{array}{cc}u_{11} & u_{12} \\u_{12} & u_{22}\end{array}\right)&\longmapsto &L(U):=\Span{e_1+u_{11}\epsilon^1+u_{12}\epsilon^2,e_2+u_{12}\epsilon^1+u_{22}\epsilon^2}\, ,\label{eqDefElleU}
\end{eqnarray}
where $L(U)$ is the graph of the matrix $U$, understood as a homomorphism from $\R^2$ to $\R^{2\ast}$.\par
The Pl\"ucker embedding defined in Definition \ref{defPluck}, in the case $n=2$, reads
\begin{eqnarray}
 \LL(2,4) &\longrightarrow & \p\Lambda^2(\R^2\oplus \R^{2\,\ast} )\, \label{eqPluckL24}\\
 L(U)=\Span{e_1+u_{11}\epsilon^1+u_{12}\epsilon^2,e_2+u_{12}\epsilon^1+u_{22}\epsilon^2} &\longmapsto & [e_1\wedge e_2 +u_{11} \epsilon^1\wedge e_2  +u_{12}( \epsilon^2\wedge e_2+ e_1\wedge\epsilon^1)+te_1\wedge\epsilon^2+\det U\,\epsilon^1\wedge\epsilon^2]\, .\nonumber
\end{eqnarray}
Recall that
$
\Lambda^2(\R^2\oplus \R^{2\,\ast} )=\Lambda^2(\R^2)\oplus \End(\R^2)\oplus \Lambda^2  (\R^{2\,\ast})
$,
where
$
\Lambda^2(\R^2)=\Span{e_1\wedge e_2}
$,
and
$
\Lambda^2  (\R^{2\,\ast})=\Span{\epsilon^1\wedge\epsilon^2}
$,
whereas
\begin{equation*}
\End(\R^2)=\gl_2=\Span{\epsilon^1\wedge e_2 ,  \epsilon^2\wedge e_2, e_1\wedge\epsilon^1,e_1\wedge\epsilon^2}\, .
\end{equation*}
In terms on matrices,
\begin{equation*}
\epsilon^1\wedge e_2  \longleftrightarrow
\left(\begin{array}{cc}
0 & 0 \\
1 & 0
\end{array}\right)\, ,\quad
\epsilon^2\wedge e_2  \longleftrightarrow  \left(\begin{array}{cc}0 & 0 \\0 & 1\end{array}\right)\, ,\quad
e_1\wedge \epsilon^1  \longleftrightarrow  \left(\begin{array}{cc}-1 & 0 \\0 & 0\end{array}\right)\, ,\quad
e_1\wedge \epsilon^2  \longleftrightarrow  \left(\begin{array}{cc}0 & 1 \\0 & 0\end{array}\right)\, .
\end{equation*}
As we observed before (cf. Definition \ref{defPluck}), however, \eqref{eqPluckL24} takes values in the subspace
\begin{equation*}
\Lambda_0^2(\R^2\oplus \R^{2\,\ast} ):=\Lambda^2(\R^2)\oplus \sll_2\oplus \Lambda^2  (\R^{2\,\ast})\, .
\end{equation*}
Indeed,
\begin{equation*}
\sll_2=\Span{ \left(\begin{array}{cc}0 & 0 \\1 & 0\end{array}\right)\, ,    \left(\begin{array}{cc}-1 & 0 \\0 & 1\end{array}\right)     \, ,   \left(\begin{array}{cc}0 & 1 \\0 & 0\end{array}\right)   }\, ,
\end{equation*}
and
\begin{equation*}
 \epsilon^2\wedge e_2+ e_1\wedge\epsilon^1  \longleftrightarrow \left(\begin{array}{cc}-1 & 0 \\0 & 1\end{array}\right) \, .
\end{equation*}
Let us introduce coordinates
$
(z_0,z_1,z_2,z_3,z_4)
$
on $\Lambda_0^2(\R^2\oplus \R^{2\,\ast} )$, by
\begin{equation*}
(z_0,z_1,z_2,z_3,z_4)\longleftrightarrow z_0  e_1\wedge e_2 + \left(\begin{array}{cc}-z_2 & z_3 \\z_1 & z_2\end{array}\right)+z_4 \epsilon^1\wedge\epsilon^2\, .
\end{equation*}
These induce projective coordinates
$
[z_0:z_1:z_2:z_3:z_4]
$
on $\p^4=\p\Lambda_0^2(\R^2\oplus \R^{2\,\ast} )$, whence \eqref{eqPluckL24}  reads
\begin{equation*}
U=\left(\begin{array}{cc}u_{11} & u_{12} \\u_{12} & u_{22}\end{array}\right)\longmapsto L(U)\equiv [1:u_{11}:u_{12}:u_{22}:u_{11}u_{22}-u_{12}^2 ]\, .
\end{equation*}
Observe that we  have reduced the projective space surrounding $\LL(2,4)$ from $\p^5$ to $\p^4$ (recall Definition \ref{defPluck}). Now we can immediately obtain the equation describing the hypersurface $\LL(2,4)$ in $\p^4$.
\begin{proposition}
$\LL(2,4)$ is the smooth quadric in $\p^4$ given by the equation
\begin{equation}\label{eqLieQuadric}
z_0z_4=z_1z_3-z_2^2\, .
\end{equation}
\end{proposition}
\begin{figure}
{\epsfig{file=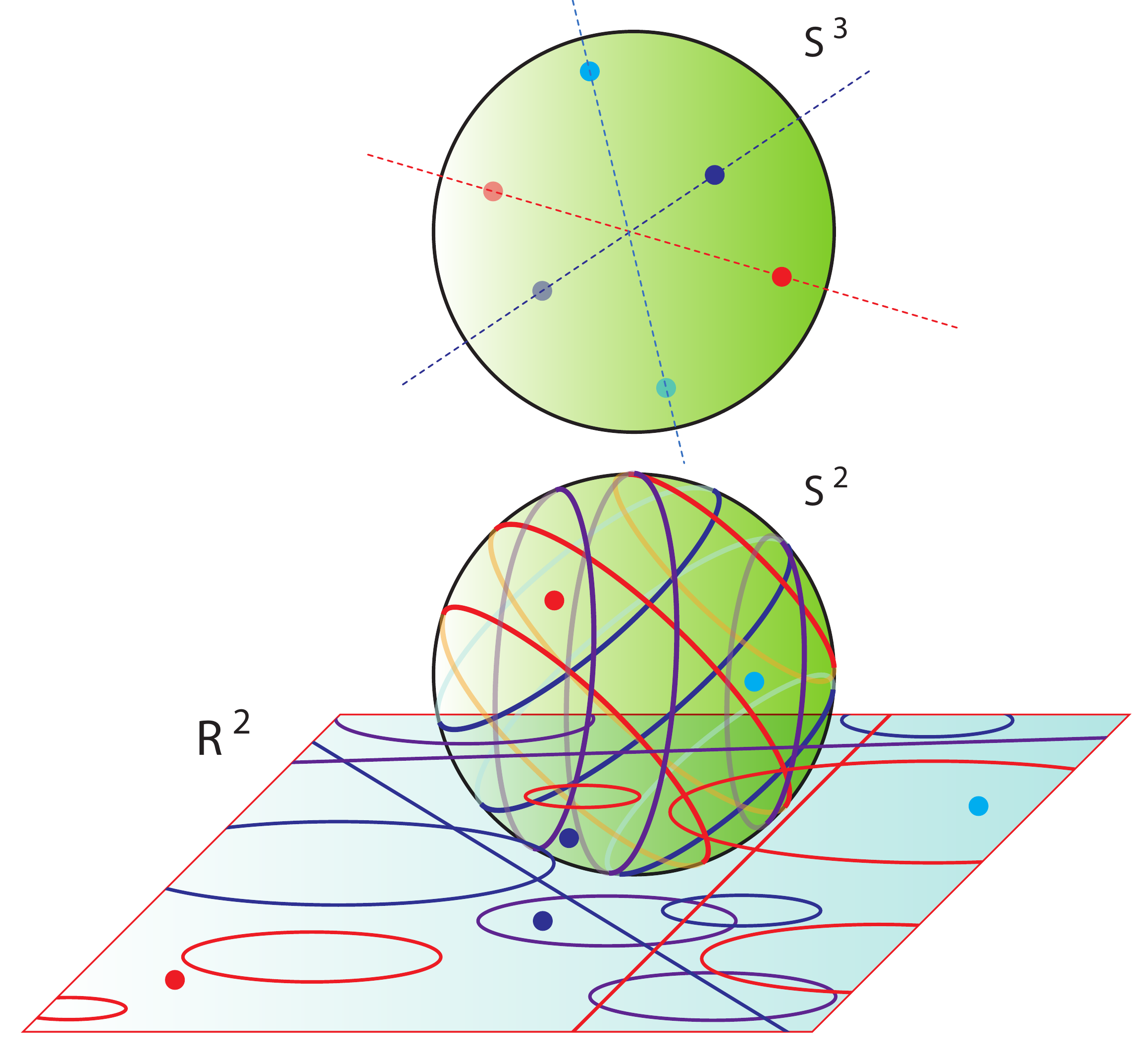,width=0.5\textwidth}}
\caption{The space of all circles in $\R^2$ (comprising the degenerate cases with zero/infinite radius) is surprisingly isomorphic to the Lagrangian Grassmannian in dimension 4.\label{Fig.Lie.Quad}}
\end{figure}
\begin{definition}
 The quadric $Q^3\subset\p^4$, cut out by the equation \eqref{eqLieQuadric}, is known as the \emph{Lie quadric}.
\end{definition}
The name \emph{Lie quadric} is due to the fact that Sophus Lie introduced it as the space of circles on the plane (see \cite{MR2876965,2014arXiv1405.5198J}). Indeed, one can take the $3$--sphere $S^3\subset \R^2\oplus \R^{2\,\ast}$ and observe that the intersection $L\cap S^3$, where $L\in\LL(2,4)$, is a circle in $S^3$. Through the Hopf fibration, such a circle becomes either a circle, either  a point in $S^2$. Finally, by stereographic projection on $\R^2$, one gets points, circles, and lines in $\R^2$ (i.e., circles of radius 0, greater than 0, and infinite, respectively). See also Fig. \ref{Fig.Lie.Quad}.\par

\section{Scalar PDEs of $2\Nd$ order}\label{sec.SecondOrderPDEs}
With the extended Darboux coordinates $(x^i,u,u_i,u_{ij})$ (see Remark \ref{rem.coordinates.M1}) on the prolongation $M^{(1)}$ of $M$, we are now ready to define $2\Nd$ order PDEs in $n$ independent variables, in strict analogy with the $1\St$ order PDEs defined in Section \ref{sec1stOrderPDEs}. Indeed, a relation amongst those coordinates \emph{formally} looks like a $2\Nd$ order  PDE. In order to recast the notion of a solution  in such a geometric framework it is indispensable to recall the  Definition  \ref{eq.def.prolog.lag.submanifold} of the prolongation of a Lagrangian submanifold $N\subset M$.

\subsection{General definition}
\begin{definition}\label{def2ndOrdPDE}
Let $(M,\mathcal{C})$ be a $(2n+1)$--dimensional contact manifold and $M^{(1)}$ its prolongation.
A hypersurface $\mathcal{E}$ of $M^{(1)}$ is called a \emph{scalar $2\Nd$ order partial differential equation} ($2^{nd}$ order PDE) with one unknown function and $n$ independent variables. A \emph{solution} of $\mathcal{E}$ is a Lagrangian submanifold $N\subset M$
whose prolongation $N^{(1)}$ is contained in $\mathcal{E}$.
\end{definition}
As in the $1\St$ order case, if
$\mathcal{E}=\{F(x^{i},u,u_{i},u_{ij})=0\}$ then a solution
$N$  parametrized by $x^{1},\dots, x^{n}$, can be written as
\begin{equation*}
N\equiv
\begin{cases}
\displaystyle{u=\varphi(x^{1},\dots,x^{n})}\,,\\
\\
\displaystyle{u_{i}=\frac{\partial\varphi}{\partial x^{i}}(x^{1},\dots,x^{n}%
)}\,,\\
\\
\displaystyle{u_{ij}=\frac{\partial^{2}\varphi}{\partial x^{i}\partial x^{j}%
}(x^{1},\dots,x^{n})\, ,}%
\end{cases}
\end{equation*}
where  the  function $\varphi$ satisfies the equation
\[
F\left( x^{i},\varphi,\frac{\partial\varphi}{\partial x^{i}},\frac
{\partial^{2}\varphi}{\partial x^{i}\partial x^{j}}\right) =0,
\]
which coincides with the classical notion of a solution.

\begin{remark}
According to Definition \ref{def2ndOrdPDE} above, a PDE of $2\Nd$ order in two independent variables can be regarded as a (family) of hypersurfaces in the Lagrangian Grassmannian $\LL(2,4)$. This is very important, in view of the identification between $\LL(2,4)$ and the Lie quadric $Q^3$ (see Section \ref{sec.L24}). Indeed, the theory of surfaces in $Q^3$ is a well--developed subject    (see \cite{MR2876965}).
\end{remark}

\subsection{Scalar ODEs of $2\Nd$ order: a working example}

The classical way to solve the ODE
\begin{equation}\label{eqEs2ndODE1}
F\left(  u(x), u'(x), u''(x)\right)=0\,
\end{equation}
is to introduce a new variable
\begin{equation}\label{eqEs2ndODEStarStarStar}
w=w(u):=u'
\end{equation}
in order to reduce the order of the considered equation (see \cite[Example 2.46]{MR1240056} and \cite[Example 2.1]{MR1670044}).
Thus, we have that
$
u'' =\frac{\dd w}{\dd u}w\,,
$
as
\begin{equation*}
\frac{\dd w}{\dd u} =\frac{\dd w}{\dd x}\frac{\dd x}{\dd u}\Rightarrow \frac{\dd w}{\dd x} =\frac{\dd w}{\dd u} \frac{\dd u}{\dd x} =\frac{\dd w}{\dd u} w\, .
\end{equation*}
So, \eqref{eqEs2ndODE1} becomes
\begin{equation}\label{eqEs2ndODE2}
F\left(  u, w(u),w'(u)w(u)\right)=0\, ,
\end{equation}
i.e., a $1\St$ order ODE with dependent variable $u$ and independent variable $w$.\par

Then, finding solutions of \eqref{eqEs2ndODE1}  means finding a solution $w=w(u)$ of \eqref{eqEs2ndODE2}, and then plug it into \eqref{eqEs2ndODEStarStarStar} and get a $1\St$ order ODE
$
u'(x)=w(u(x))
$.
Let us give this method a geometric interpretation.\par
Condition \eqref{eqEs2ndODE1} determines the subset
$
\E:=\{F=0\}\subset J^2(1,1)
$.
Observe that the $2$--dimensional distribution given as the kernel of $1$--forms \eqref{eqContForm2ndorder1} is spanned by
\begin{equation}\label{eq.distr.cont.on.J2}
D_x:=\partial_x+u_1\partial_u+u_{11}\partial_{u_1}\, ,\quad
V:=\partial_{u_{11}}\, .
\end{equation}
Consider the diffeomorphism
\begin{equation*}
t_c:(x,u,u_1,u_{11})\to (x+c,u,u_1,u_{11})
\end{equation*}
and observe that
${t_c}_*(\CC)=\CC$.
This can be shown in two equivalent ways. Either we observe that $t_c^*$ preserves the two forms \eqref{eqContForm2ndorder1}, or we observe that $L_{\partial_x}(D_x)=0$ and $L_{\partial_x}(V)=0$.
Moreover, $t_c(\E)=\E$ (cf. also Definition \ref{defPrelimSymm}).\par
\begin{figure}
  {\epsfig{file=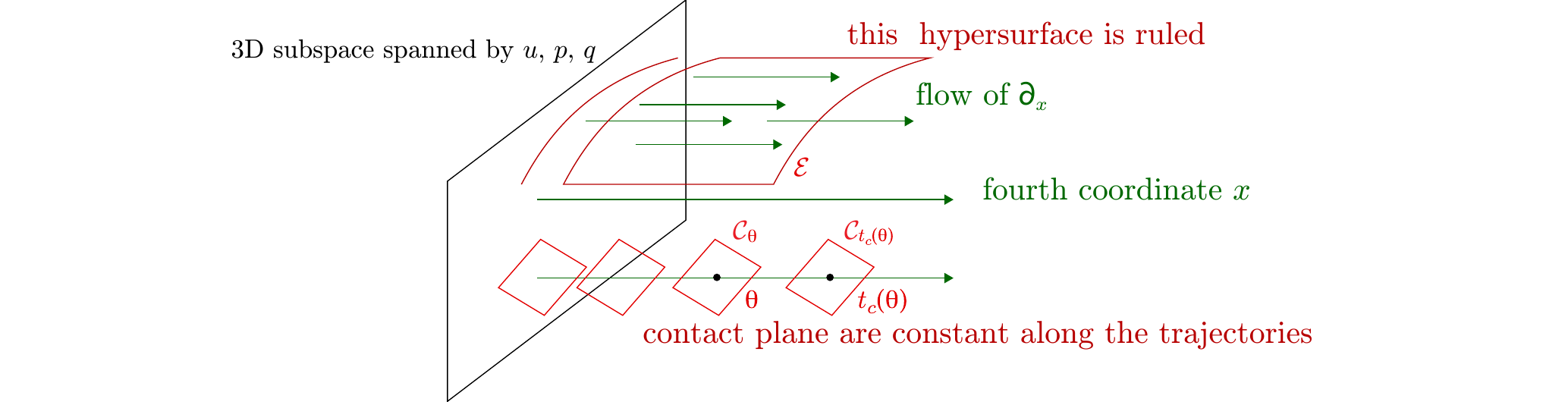,width=\textwidth}}\caption{If there is no explicit dependence on $x$, the equation can be projected to a lower--dimensional space.\label{Fig.Proiettante}}
\end{figure}
It is convenient now to introduce the smooth map
\begin{eqnarray*}
T:J^2(1,1) &\longrightarrow & \R^3\, ,\\
(x,u,u_1,u_{11}) &\longmapsto & (u,u_1,u_{11})\, ,
\end{eqnarray*}
which identifies the quotient space of $J^2(1,1)$ with respect to the $1$--parameter group of transformations $\{t_c\}$, with the more convenient space $\R^3$. The natural question is now whether the distribution \eqref{eq.distr.cont.on.J2} on $J^2(1,1)$ passes to the quotient $\R^3$, in such a way that the latter becomes (locally) isomorphic to the three--dimensional contact manifold $J^1(1,1)$. \par

Observe that the image $T(\E)$ is given, as a hypersurface in $\R^3$, by the same equation \eqref{eqEs2ndODE1}.
Let us compute the ``image distribution'' $\widetilde{\CC}:=T_*(\CC)$. %
By definition, for any
$
p_0=(u,u_1,u_{11})
$,
the plane $\widetilde{\CC}_{p_0}$ can be obtained by applying $d_p T$ to $\CC_p$, where $p$ is any point such that $T(p)=p_0$. See also Fig. \ref{Fig.Proiettante}.
So, choose an arbitrary $x\in\R$ and put $p=(x,u,u_1,u_{11})$.
Then
$
\CC_p=\Span{\partial_x|_p+u_1\partial_u|_p+u_{11}\partial_{u_1}|_p, \partial_{u_{11}}|_p}
$
and
$
d_p T\CC_p=\Span{ u_1\partial_u|_{p_0}+u_{11}\partial_{u_1}|_{p_0}, \partial_{u_{11}}|_{p_0}}
$.
Hence, $\widetilde{\CC}$ is spanned by the vector fields
$$
\widetilde{D}:=u_1\partial_u+u_{11}\partial_{u_1}\, ,\quad
\widetilde{V}:=
\partial_{u_{11}}\, .
$$
Being $\widetilde{D}$ and $\widetilde{V}$ independent, there must be a unique differential $1$--form $\widetilde{\omega}\in\Lambda^1(\R^3)$, such that
$
\widetilde{\CC}=\ker \widetilde{\omega}
$.
Let us write
$
\widetilde{\omega}=adu+bdu_1+cdu_{11}
$
and impose the two conditions:
\begin{equation*}
\widetilde{\omega}(\widetilde{D})=au_1+bu_{11}=0\, ,\quad
\widetilde{\omega}(\widetilde{V})=c=0\, .
\end{equation*}
It follows that
$ au_1=-bu_{11}$ and $c=0$,
i.e., any differential form $\widetilde{\omega}=adu-bdu_1$ with
\begin{equation}\label{eqEsSecOrdDelta}
 au_1=-bu_{11}
\end{equation}
defines the distribution $\widetilde{\CC}$.
We wonder now if there is a change of coordinates
$
\Phi:(t,y,y_1)\to (u,u_1,u_{11})
$
which brings $\widetilde{\CC}$ to the standard Contact distribution
$
\Span{\partial_t+y_1\partial_y,\partial_{y_1}}
$
on $J^1(1,1)$.\par
The ``global'' answer to this question is negative, since $\widetilde{D}$ vanishes when $u_1=u_{11}=0$ (this should not surprise, since in the algebraic manipulations we divided by $u'$). So, we rule out the zero--measure subset $u_1=u_{11}=0$, and we observe that, on the remaining part,
\eqref{eqEsSecOrdDelta} reads
$
a=-b\frac{u_{11}}{u_1}
$,
i.e.,
\begin{equation*}
\widetilde{\omega}=bdu_1-b\frac{u_{11}}{u_1}du=b\left(  du_1 -\frac{u_{11}}{u_1}du\right)\, ,
\end{equation*}
hence, being $\widetilde{\omega}$ defined up to proportionality,
$
\widetilde{\omega}=du_1-\frac{u_{11}}{u_1}du
$.
Hence, it is sufficient to set
\begin{equation*}
u := t\, ,\quad
u_1 := y\, ,\quad
u_{11} := \overline{u_{11}}(t,y,y_1)\, ,
\end{equation*}
and obtain
$
F^*(\widetilde{\omega})=dy-\frac{\overline{u_{11}}(t,y,y_1)}{y}dt\equiv dy-y_1dt\Leftrightarrow \overline{u_{11}}(t,y,y_1)=yy_1
$.
We conclude that, in the new variables $t,y,y_1$, $T(\E)$ is given by  the relations
$
F(t,y,yy_1)=0
$,
i.e., precisely \eqref{eqEs2ndODE2}.

\section{Symmetries of PDEs and the Lie--B\"acklund theorem}\label{secSymPDEBack}

In this section we review the general definition of symmetries and infinitesimal symmetries for arbitrary systems of PDEs in several unknown a known variables, that is, submanifolds of the jet spaces
\begin{equation*}
J^k(n,m):=\{[f]^k_{\x}\,\,|\,\,f\in C^{k}(\R^n,\R^m)\}\, ,
\end{equation*}
whose definition is an obvious generalization of that given in \eqref{eq.J1.n.m} for $k=1$.
The map
$$
j^kf:\R^n\to J^k(n,m)\,,\quad \x\mapsto [f]^k_{\x}\,,
$$
in analogy with \ref{eq.j1.F}, is also defined. A local chart on $J^k(n,m)$ is denoted by
$$
(x^i,u^a,u^a_J)=(\x,\u,\u_J)\,,
$$
and coordinates $\u_J$ are defined by
\begin{equation}\label{eq.uJ}
\u_{J}\circ j^kf=\frac{\partial^{|J|}f}{\partial x^J}:=\frac{\partial^{h}f}{\partial x^{j_1}\cdots\partial x^{j_h}}\,,
\end{equation}
where $J= j_1\cdots j_h$ is a multi--index and $|J|:=h\leq k$. Similarly as we did in \eqref{seconddef}--\eqref{eq.distr.C.general.local}, one can define a canonical distribution $\CC$ on $J^k(n,m)$ whose local expression
\begin{equation}\label{eq.Cartan.distribution.general}
\CC=\Span{D_{x^i}=\partial_{x^i}+\sum_{|J|<k} u^a_{J,i}\partial_{u^a_J}\,,\,\partial_{u^a_K}}\,,\quad |K|=k\,,
\end{equation}
where $(J,i):=(j_1\cdots j_h,i)$, is analog to \eqref{eq.distr.C.general.local}.
\begin{definition}\label{def.Cartan.distribution}
The distribution \eqref{eq.Cartan.distribution.general} is called the \emph{Cartan} distribution on $J^k(n,m)$.
\end{definition}
Even if this goes far beyond the context of contact manifolds and their prolongation, it is only in such an extended framework that one can state the celebrated Lie--B\"acklund theorem, whose proof is omitted due to its complexity. Moreover, the notions introduced here will be needed in Section \ref{secVeryClear}.
 We begin with the definition of a symmetry of $J^k(n,m)$, which is a particular case of Definition \ref{defPrelimSymm}.
\begin{definition}\label{defSymmJey}
Let $\phi$ be a diffeomorphism of $J^k(n,m)$ (resp., a vector field $X\in\X(J^k(n,m))$). The diffeomorphism $\phi$  (resp., $X$) is called a \emph{finite (resp., infinitesimal) symmetry} of $J^k(n,m)$ if and only if $\phi_*(\CC)=\CC$ (resp., $[X,\CC]\subset\CC$).
\end{definition}
By definition, any diffeomorphism of $J^0(n,m)=\R^n\times\R^m$ (resp., a vector field $X\in\X(J^0(n,m))$) is a finite (resp., infinitesimal) symmetry of $J^0(n,m)$. Let $\E\subset J^k(n,m)$ be a (system of) PDEs.
\begin{definition}\label{defSymmEQ}
$\phi\in\mathrm{Diffeo}(J^k(n,m))$ (resp., $X\in\X(J^k(n,m))$) is a \emph{finite} (resp., \emph{infinitesimal}) symmetry of $\E$ if and only if $\phi$ is a symmetry of $J^k(n,m)$ and $\phi(\E)=\E$ (resp., $X$ is tangent to $\E$).
\end{definition}
Recall (see Conclusion \ref{conConlcusioneSollevamento}) that we have managed to lift a vector field from $J^0$ to a symmetry of $J^1$.\par
This is a rather general fact, involving also finite symmetries. In fact, the prolongation of a symmetry $\phi$ of $J^{k-1}(n,m)$ to a symmetry of $J^k(n,m)$ is defined as follows: $\phi([F]^k_{\x}):=[\phi^*(F)]^k_{\phi(\x)}$. One can obtain the prolongation of an infinitesimal symmetry by considering the prolongation of its local flow.
\begin{remark}\label{remProlong}
One could define the above prolongation also by observing that  $J^{k}(n,m)$, similarly as in Section \ref{sec.HolEqsJ111}, can be seen as an ``equation'' inside $J^1(\R^n,J^{k-1}(n,m))$, and that
there is an obvious way to lift a diffeomorphism of $J^0=J^{k-1}(n,m)$ to $J^1$.  Less obviously, if the original diffeomorphism is a symmetry (see Definition \ref{defSymmJey}) of $J^{k-1}(n,m)$, then its lift is a symmetry (see Definition \ref{defSymmEQ}) of the equation $J^{k}(n,m)$ inside $J^1(\R^n,J^{k-1}(n,m))$, that is, a symmetry of $J^{k}(n,m)$.
\end{remark}

\begin{definition}
Let $\phi$ (resp., $X$) be a finite (resp., infinitesimal) symmetry of $J^{k}(n,m)$. Then the recursively defined diffeomorphism (resp., vector field) $\phi^{(l)}:=(\phi^{(l-1)})^{(1)}$ (resp., $X^{(l)}:=(X^{(l-1)})^{(1)}$)  of $J^{k+l}(n,m)$ is called the $l^\textrm{th}$ lift of $\phi$ (resp., $X$).
\end{definition}
The following theorem, whose proof is not trivial, classifies all the symmetries of $J^k(n,m)$.
\begin{theorem}[Lie--B\"acklund]\label{thLieBakc}
Let $\phi$ (resp., $X$) be a finite (resp., infinitesimal) symmetry of $J^{k}(n,m)$. Then, if $m=1$, there exists a finite (resp., infinitesimal) symmetry $\widetilde{\phi}$ (resp., $\widetilde{X}$) of $J^{1}(n,m)$, such that $\phi=\widetilde{\phi}^{(k-1)}$ (resp., $X=\widetilde{X}^{(k-1)}$); if  $m> 1$,  there exists a finite (resp., infinitesimal) symmetry $\widetilde{\phi}$ (resp., $\widetilde{X}$) of $J^{0}(n,m)$, such that $\phi=\widetilde{\phi}^{(k)}$  (resp., $X=\widetilde{X}^{(k)}$).
\end{theorem}
\begin{proof}
 See, e.g., \cite[Theorem 4.32]{book:27732}.
\end{proof}
\begin{definition}\label{def.point.trans}
Diffeomorphisms of $J^0(n,m)$ are called \emph{point transformations}. Symmetries of $J^k(n,m)$ coming from point transformations are called \emph{point symmetries}.
\end{definition}
The meaning of Theorem \ref{thLieBakc} is   that, basically, the only symmetries of $J^{k}(n,m)$, no matter how large $k$ is, are either point symmetries, or contact symmetries. So, all of the interestingness of the theory is concentrated at low orders $k=0,1,2$, and this is the reason why we are focusing on these cases only. As a final remark, we observe that the Lie--B\"acklund theorem is valid also for the Lagrangian Grassmannian bundle $M^{(1)}\to M$: any symmetry of $M^{(1)}$ (where a distribution $\CC$ on $M^{(1)}$  can be  defined by using similar reasonings that led to the definition \eqref{seconddef}--\eqref{eq.C.general}, see also \cite{ArticoloAntipatetico}
for a more conceptual viewpoint) is the prolongation of a contact symmetry of the underlying contact manifold $M$.

\section{Characteristics}\label{sec.characteristics}
So far we have built a geometric environment for $1\St$ and $2\Nd$ order PDEs based on contact manifolds and their prolongations. In  Section \ref{secSymPDEBack} above  we have pointed out that an analogous framework for higher--order PDEs $\E\subset J^k$ can be obtained as well. We do not insist on it here, for it does not show any really new feature.\par
Rather, we focus on the  problem which anyone expects to be central in a reasonable theory of PDEs:  \emph{how to construct solutions?}
More precisely, if a PDE $\E$ is complemented by a Cauchy datum $\Sigma$, we wish to establish an existence and uniqueness result for the initial values problem $(\E,\Sigma)$. This is the subject of the well--known Cauchy--Kowalewskaya theorem (see Theorem \ref{thCK} later on), which admits a transparent formulation in the present geometric framework (see also Section \ref{secAntipatetica}). \par
The key idea is that a Cauchy datum $\Sigma$ is an integral submanifold of the contact distribution of dimension one less than the maximal, and one needs to ``enlarge'' it, by using an appropriate vector field (characteristic symmetry), to a full--dimensional submanifold that turns out to be a solution (in the sense of Definition \ref{def2ndOrdPDE} in the case of $2\Nd$ order PDES). The possibility of uniquely extending the given Cauchy datum $\Sigma$ to a solution of a given PDE can be tested infinitesimally, i.e., by inspecting the tangent space to $\Sigma$. It turns out that these tangent hyperplanes can be distinguished into two kinds: characteristics and non--characteristics. Only the latter correspond to well--defined initial values problems.\par
On the other hand, the characteristic (infinitesimal) Cauchy data (i.e., those giving rise to ill--defined initial values problems) can be conveniently collected into a variety, traditionally dubbed \emph{characteristic variety}, whose description is given in Section \ref{sezioneZarinista} and, in more contact-geometrical terms, in Section \ref{secAntipatetica}.
Its striking feature is that it is a geometric structure naturally associated to a jet space of order less by one. In fact, for a rich and interesting class of PDEs (see Section \ref{secMAE} below and Remark \ref{rem.reconstruct}), the characteristic variety \emph{allows to reconstruct the equation itself}.

\subsection{The method of the characteristic field for $1\St$ order PDE}\label{sec.method.characteristic}

Let $(M,\CC)$ be a $(2n+1)$--dimensional contact manifold.
\begin{proposition}\label{prop.char.Yf}
Let $\F=\{f=0\}\subset M$ be a $1\St$ order  PDE.
The $(n-1)$--dimensional restricted distribution $\CC_\F=\CC\cap T\F$ possesses the Hamiltonian vector field $Y_{df}=Y_f$ as a characteristic symmetry (see Definition \ref{def.involution} and subsequent properties), called also the \emph{characteristic vector field} of $\F$.
\end{proposition}
\begin{proof}
The proposition has been, essentially, already proved after Definition \ref{def.involution}. It is enough to observe that $Y_f\in\CC$ and also that $Y_f(f)=0$, which implies $Y_f\in\ker df$, and that $\CC_\F=\{\theta=0\,,\,df=0\}$.
\end{proof}
The characteristic vector field $Y_f$ is always tangent to every solution of $\F=\{f=0\}$ as it is tangent to $\F$  (indeed $Y_f(f)=0$) and it
preserves $\CC_\F$ in view of Proposition \ref{prop.char.Yf}.
For $n=1$, one gets precisely the vector field generating $\CC_\F$, which is $1$--dimensional (see Section \ref{secMethOfChar} below and also Section  \ref{sec.an.example.solve.ode} above for a working example). In other words, in this case, finding solutions to $\F$ amounts to integrating $Y_f$, i.e., finding its trajectories. For $n\geq 2$, solutions are no longer curves (i.e., points evolving in time), but rather $n$--folds (i.e., $(n-1)$--folds evolving in time). Interestingly enough, the techniques valid for $n=1$ can be adapted to the other cases.
\begin{definition}\label{def.cauchy.datum.first}
A \emph{Cauchy datum} for a $1\St$ order PDE $\F=\{f=0\}$, $f\in C^\infty(M)$, is an
$(n-1)$--dimensional integral submanifold of $\CC$ included in $\F$. It is called \emph{non--characteristic} if it is
transversal to the Hamiltonian vector field $Y_f$.
\end{definition}
Our discussions lead to the following theorem.
\begin{theorem}[Method of characteristics for $1\St$ order PDEs]\label{th.meth.char}
Let $\F$ be a $1\St$ order scalar PDE. Let $\Sigma\subset\F$ be a non--characteristic Cauchy datum. Then, there is some $\epsilon>0$, such that
\begin{equation*}
\widetilde{\Sigma}:=\bigcup_{-\epsilon<t<+\epsilon}\phi_t(\Sigma)\, ,
\end{equation*}
with $\{\phi_t\}$ the flow of $Y_f$, is a solution of $\F$.
\end{theorem}
In other words, the characteristic field $Y_f$ allows to ``enlarge'' the Cauchy datum $\Sigma$ to a honest, fully--dimensional solution $\widetilde{\Sigma}$, provided that the datum $\Sigma$ is non--characteristic.

\begin{example}
Let us consider the $1\St$ order scalar PDE in $2$ independent variables $(x^1,x^2)$
\begin{equation}\label{eq.pde.da.integrare}
f:=u-u_1u_2=0\,.
\end{equation}
Below we find the solution $u=u(x^1,x^2)$ of \eqref{eq.pde.da.integrare} such that $u(x^1,0)=(x^1)^2$. This condition, together with equation \eqref{eq.pde.da.integrare}, determines the Cauchy datum
\begin{equation}\label{eq.cauchy.datum.first}
x^1=s\,,\quad x^2=0\,,\quad u=s^2\,,\quad u_1=2s\,,\quad u_2\overset{\eqref{eq.pde.da.integrare}}{=}\frac{u}{u_1}=\frac{s}{2}\,.
\end{equation}
The Hamiltonian vector field $Y_f$ associated to the function $f=u-u_1u_2$ is (see \eqref{eq.hamiltonian.local})
$$
Y_f=-u_2D_{x^1}-u_1D_{x^2}-u_1\partial_{u_1}-u_2\partial_{u_2} =-u_2\partial_{x^1}-u_1\partial_{x^2}-2u_1u_2\partial_u-u_1\partial_{u_1}-u_2\partial_{u_2}\,.
$$
The Cauchy datum \eqref{eq.cauchy.datum.first} is transverse to the characteristic field $Y_f$, so that \eqref{eq.cauchy.datum.first} is not characteristic. By integrating $Y_f$ we obtain
\begin{equation*}
x^1 = C_1e^{-t}+C_4\,,\quad x^2 = C_2e^{-t}+C_3\,,\quad u = C_1C_2e^{-2t}+C_5\,,\quad u_1 = C_2e^{-t}\,,\quad u_2 = C_1e^{-t}\,,
\end{equation*}
and by imposing that $(x^1(0),\dots,u_2(0))$ belongs to the Cauchy datum \eqref{eq.cauchy.datum.first}, we obtain
$$
C_1 = \frac12\tau\,,\quad C_2 = 2\tau\,,\quad  C_3 = -2\tau\,, \quad C_4 = \frac12\tau\,,\quad    C_5 = 0\,,
$$
i.e.,
$$
x^1 = \frac12\tau e^{-t}+\frac12\tau \,,\quad x^2 = 2\tau e^{-t}-2\tau \,,\quad u = \tau^2 e^{-2t} \,,\quad u_1 = 2\tau e^{-t} \,,\quad u_2 = \frac12 \tau e^{-t}\,.
$$
Now, by solving the first $2$ equations with respect to $t$ and $\tau$ and by substituting them in the third one we obtain $u=\frac{1}{16}(4x^1+x^2)^2$, that is the solution we were looking for.
\end{example}

\subsection{The method of the characteristic field for higher--order ODEs}\label{secMethOfChar}
Let us now pass to $r^\textrm{th}$ order ODEs. In analogy with Definition \ref{def.1stOrderPDEs},
such an equation can be written as
\begin{equation*}
\{F(x,u,p_1,p_2,\ldots,p_r)=0\}=:\E\subseteq J^r(1,1)\, , \quad\text{where}\quad p_k:=u_{\underset{\textrm{k-times}}{\underbrace{1\cdots 1}}}\,,
\end{equation*}
i.e., a hypersurface (and hence $(r+1)$--dimensional) in the $(r+2)$--dimensional jet space $J^r(1,1)$.\par
The Cartan distribution $\CC$ (see Definition \eqref{def.Cartan.distribution}) is the common kernel of the $r$ differential 1--forms
\begin{equation*}
\omega_i=dp_i-p_{i+1}dx\, ,\quad i=0,1,\ldots, r-1\, ,\quad p_0:=u\, ,
\end{equation*}
and as such is 2--dimensional. Recall that (see \eqref{eq.Cartan.distribution.general})
$$
\CC=\Span{D:=D_x=\partial_x+p_1\partial_u+p_2\partial_{p_1}+\cdots+p_r\partial_{p_{r-1}}\,,\,V:=\partial_{p_r}}\,.
$$
This confirms the fact that, in the geometric theory of differential equations, the main tool is the Cartan distribution, not the jet space, which is a mere container---to be discarded once the equation one wants to study has been defined. Indeed, no matter how high is the order $r$ of the ODE, the Cartan distribution is always 2--dimensional, and, accordingly, the Cartan distribution restricted to $\E$, denoted by $\CC_\E$, is $1$--dimensional. This means that computations go along the same lines as before. Additional difficulties are apparent.\par
Namely, let $X\in\CC$, i.e.,
$$
X=\alpha D+ \beta V\,, \quad \alpha,\beta\in C^\infty(J^r(1,1))\,,
$$
and compute
$$
X(F)=dF(X)=\alpha(\partial_xF+p_1\partial_uF+p_2\partial_{p_1}F+\cdots+p_r\partial_{p_{r-1}}F)+\beta\partial_{p_r}F\,.
$$
Hence $X_p\in T_p\E$ $\forall\,p\in\E$ if and only if $(dF)_p(X_p)=0$, that is,  $X(F)(p)=0$ $\forall\,p\in\E$. This is equivalent to require
\begin{equation}\label{eqCampCaharConPallino}
\alpha \left(  F_x+\sum_{i=0}^{r-1}p_{i+1} F_{p_i} \right)+  \beta F_{p_r} \equiv F\, ,
\end{equation}
since any function which vanishes on $\E$ is proportional to $F$. An obvious solution of \eqref{eqCampCaharConPallino} is obtained with
$\alpha=-F_{p_r}=-V(F)$ and $\beta=D(F)$.
In other words,
\begin{equation*}
X=-V(F)D+D(F)V=:Y_F
\end{equation*}
is a vector field such that
$
\CC_\E=\Span{Y_F}
$,
for the points where $\CC_\E$ is $1$--dimensional.\par
In the other points,
$
(\CC_\E)_p=\CC_p
$,
i.e., both $D_p$ and $V_p$ are tangent to $\E$. This means that  $D(F)$ and $V(F)$ both vanish in $p$ and, thence,
\begin{equation*}
(Y_F)_p=-V(F)(p)D_p+D(F)(p)V_p=0\,.
\end{equation*}
In other words,
\begin{equation*}
(\CC_\E)_p=\left\{\begin{array}{lll}\Span{(Y_F)_p} & \textrm{when} & \dim(\CC_\E)_p=1\,, \\\CC_p & \textrm{when} & \dim(\CC_\E)_p=2\Leftrightarrow (Y_F)_p=0\, ,\end{array}\right.
\end{equation*}
i.e., the vector field $Y_F$, called \emph{characteristic vector field} in analogy to the characteristic vector field introduced in Section \ref{sec.method.characteristic}, generates $\CC_\E$ when it is not zero. See also Fig. \ref{Fig.Sfera}.
\begin{remark}
 Being $\dim\E=r+1$, we incidentally showed that an ODE of order $r$ is equivalent to a system of $r+1$ ODEs of order one (which is usually reduced to $r$ equations, since we look for ``horizontal trajectories", i.e., with $x(t)=t$), namely the \emph{equation of the trajectories of $Y_F$}.
\end{remark}
\begin{remark}
 When we face the simple multi--variable calculus problem of finding the directions along which a function $f=f(x,y)$ is constant, we may think as follows: \emph{the gradient $\nabla f=(f_x,f_y)$ is the direction along which the variation of $f$ attains it maximum}. It is a theorem of the theory of surfaces
 that the direction along which $f$ is constant is precisely the orthogonal one to $\nabla f$. For example, a counterclockwise rotation of $90^\circ$ produces the vector $(-f_y,f_x)$, the rotation being obtained from the standard symplectic structure of $\R^2$. Similarly, the vector $(D_p (F), V_p (F))$ can be thought of as the direction (in $\CC_p$) along which the variation of $F$ is maximal (and, hence, $\E$ intersect transversally $\CC_p$), so $(-V_p(F), D_p(F))$ is the direction along which $F$ is constant (i.e., $\E$ is tangent to $\CC_p$). This analogy should motivate the definition of $Y_F$.
\end{remark}
\begin{figure}
{\epsfig{file=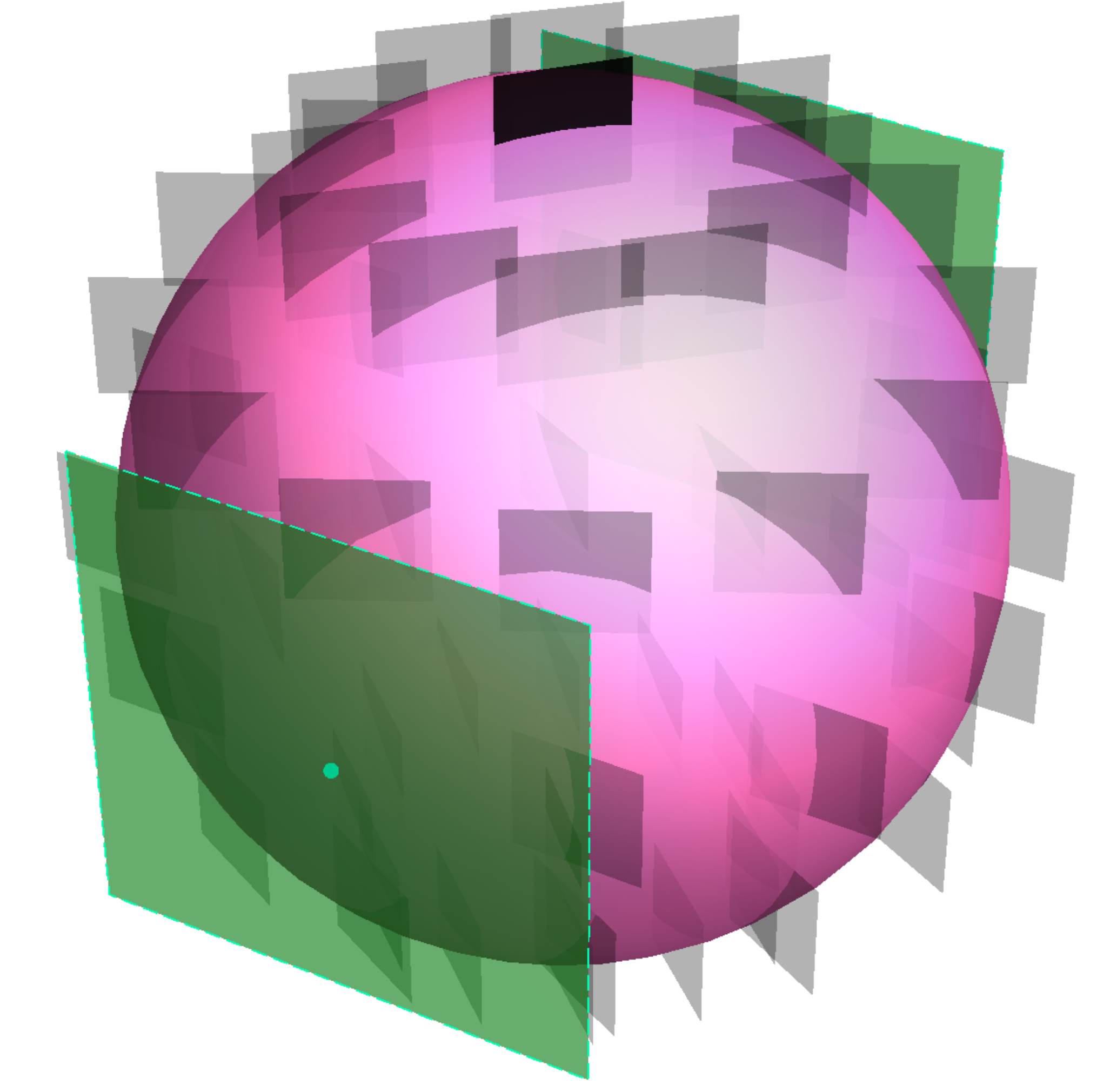,width=0.5\textwidth}}
 \caption{The tangent planes in green coincide with the contact hyperplanes. Hence, any characteristic vector field must vanish when passing through them.\label{Fig.Sfera}}
\end{figure}
\subsubsection{An example of how to solve a $1\St$ order ODE by integrating the characteristic vector field}
Consider the equation
\begin{equation}\label{eqEQEsempioChar}
u'(x)^2+u(x)^2=1\,
\end{equation}
(see \cite[Example 1.3]{MR1670044}) and the corresponding
submanifold
$
\{F:=u^2+u_1^2-1=0\}=:\F\subseteq J^1(1,1)
$.
The contact form is
$
\omega=du-u_1dx
$.
We wish to compute the restriction $\CC_\F$. To this end, observe that
$
dF=2udu+2u_1du_1
$.
Since
$\CC=\Span{D,V}$, $D=\partial_x+u_1\partial_u$, $V=\partial_{u_1}$,
let us take a linear combination
$
X=\alpha D+ \beta V\, ,\alpha,\beta\in C^\infty(J^1(1,1))
$,
and let us establish for which coefficients $\alpha$ and $\beta$,  $X$ is annihilated by $dF$, i.e., it is tangent to $\F$.\par
We compute
$
dF(X)=2u\alpha u_1 + 2u_1\beta =2u_1(u\alpha+\beta)
$.
Hence, for any ${\vec{\ell}}\in J^1(1,1)$ (cf. \eqref{eqCorrdSbaglSuJ121}), we see that
$
X_{{\vec{\ell}}}\in T_{\vec{\ell}}\F$ if and only if  $2u_1({\vec{\ell}})(u({\vec{\ell}})\alpha({\vec{\ell}})+\beta({\vec{\ell}}))=0$.
Away from the plane $u_1=0$, this means $\beta=-u\alpha$, i.e.,
$
(\CC_\F)_{{\vec{\ell}}}=\Span{(\partial_x+u_1\partial_u-u\partial_{u_1})_{\vec{\ell}}}$, for all ${\vec{\ell}}$  with $u_1({\vec{\ell}})\neq 0
$.
On the other hand, on the plane $u_1=0$, all values $\alpha, \beta$ are admissible, so that
$
(\CC_\F)_{{\vec{\ell}}}=\CC_{{\vec{\ell}}}\, ,\forall{\vec{\ell}}$ with $u_1({\vec{\ell}})= 0
$.
Summing up, as shown in Fig. \ref{fig.Cilindro},
\begin{equation*}
\dim (\CC_\F)_{{\vec{\ell}}}=\left\{\begin{array}{ccc}1 & \textrm{if} & u_1(\vec{\ell})\neq 0 \,, \\2 & \textrm{if} & u_1(\vec{\ell})= 0\end{array}\right.\,.
\end{equation*}
\begin{figure}
{\epsfig{file=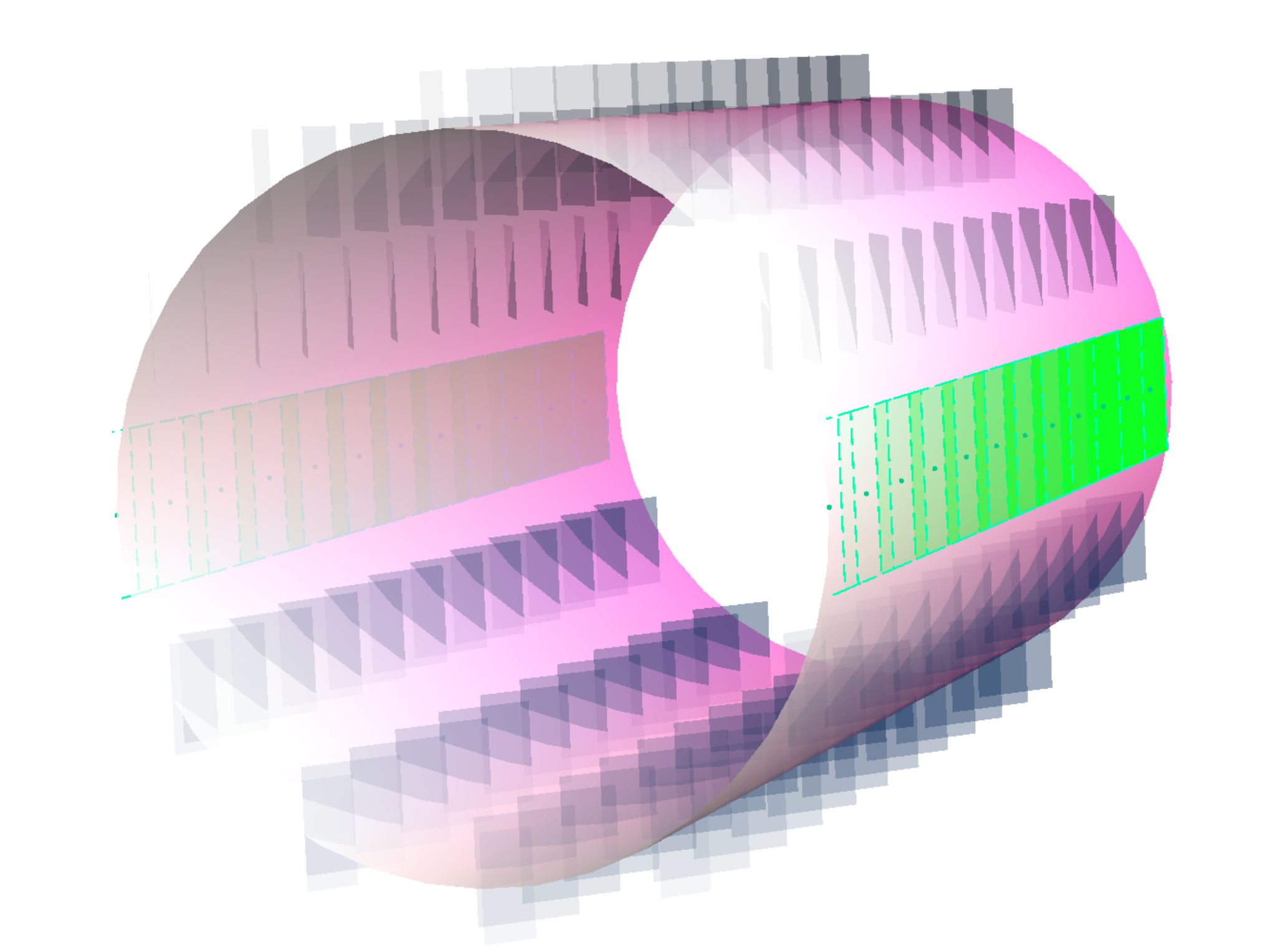,width=0.5\textwidth}}
\caption{Intersection of the contact distribution in $\R^3$ with the equation $(u')^2+u^2=1$.\label{fig.Cilindro}}
\end{figure}
In other words, the vector field
$
Z:=\partial_x+u_1\partial_u-u\partial_{u_1}
$
spans $\CC_\F$ away from the plane $u_1=0$, while on this plane, two vector fields, namely $D$ and $V$, are necessary to span $\CC_\F$.\par
We compute now the trajectories of $Z$. Let $\gamma(t)=(\overline{x}(t),\overline{u}(t),\overline{u_1}(t))$ be a curve in $J^1(1,1)$ and suppose that it is $Z$--integral, i.e., $\dot{\gamma}(t)=Z_{\gamma(t)}$ for all $t$. This means that
\begin{equation*}
 \left\{\begin{array}{l}\overline{x}'(t)=1 \,, \\\overline{u}'(t)=\overline{u_1}(t) \,, \\\overline{u_1}'(t)=-\overline{u}(t)\, ,\end{array}\right.
\end{equation*}
from which it follows that $\overline{x}(t)=t$ and then
\begin{equation*}
\left\{\begin{array}{l}\overline{u}'(x)=u_1(x) \,, \\\overline{u_1}'(x)=-u(x)\, .\end{array}\right.
\end{equation*}
This implies that $\overline{u}''(x)=-\overline{u}(x)$, i.e.,
\begin{equation*}
\left\{\begin{array}{l}u=c_1\cos x+c_2\sin x \,, \\u_1=-c_1\sin x+c_2\cos x\, .\end{array}\right.
\end{equation*}
Introduce the matrix
\begin{equation*}
A(x):=\left|\begin{array}{cc}\cos x & \sin x \\-\sin x & \cos x\end{array}\right|
\end{equation*}
and regard the two integration constants $(c_1,c_2)$ as a vector of $\R^2$. Then
\begin{equation}\label{eqEquazioneSegnataConDelta}
x\longmapsto \left|\begin{array}{cc}1 & 0 \\0 & A(x)\end{array}\right|\cdot (x,c_1,c_2)^t
\end{equation}
describes the trajectory $\gamma$ of $Z$ which emanates from the point $(x,c_1,c_2)$. Since $Z$ is tangent to $\F$, if $(x,c_1,c_2)\in\F$ then
$
\gamma(\R)\subseteq\F
$.
Hence, all the solutions of the form $\gamma$ can be obtained by choosing a point
$
(0,c_1,c_2)\in\F\cap\{ x=0\}
$,
which is a circle, i.e., an angle $\phi\in[0,2\pi]$ (and then setting $c_1=\cos\phi$, $c_2=\sin\phi$).\par
However, also the two curves
\begin{equation}\label{eqEquazioneSegnataConBOX}
\left\{\begin{array}{ccc}\gamma_+(x) & := & (x,+1,0) \\\gamma_-(x) & := & (x,-1,0)\end{array}\right.
\end{equation}
are integral for the distribution $\CC_\F$, since
\begin{equation*}
\dot{\gamma}_{\pm}(x)=\left.\frac{\partial}{\partial x}\right|_{(x,\pm 1,0)}\in (\CC_\F)_{(x,\pm 1,0)}\, ,\quad \forall x\in\R\, .
\end{equation*}
Hence, \eqref{eqEquazioneSegnataConDelta} and \eqref{eqEquazioneSegnataConBOX} are all the solutions of \eqref{eqEQEsempioChar}, rendered in Fig. \ref{Figura.sborona}.
\begin{figure}
{\epsfig{file=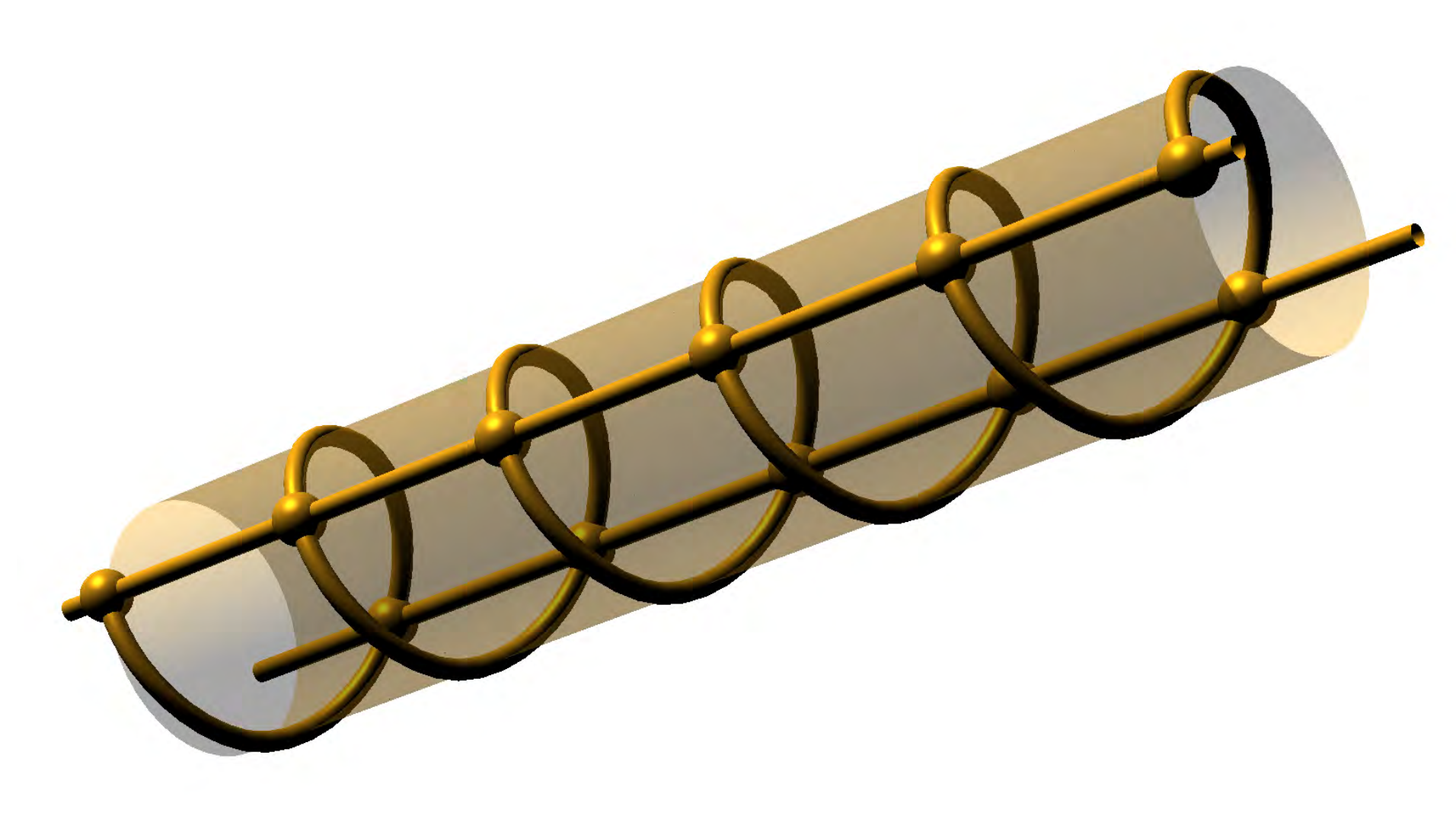,width=0.8\textwidth}}
 \caption{Two solutions of $(u')^2+u^2=1$ artistically rendered to show the difference between the regular and singular ones.\label{Figura.sborona}}
\end{figure}
Observe that the curves of the form \eqref{eqEquazioneSegnataConDelta}, being trajectories of a vector field ($Z$, in this case), form a family of non--intersecting curves. But this property fails if to such a family we add also the curves \eqref{eqEquazioneSegnataConBOX}, which in fact intersect all the curves  \eqref{eqEquazioneSegnataConDelta}. The points of intersection are precisely the points where $\CC_\F$ fails to be $1$--dimensional.\par
Let now solve the same equation  \eqref{eqEQEsempioChar} with a wise change of variables. Namely, pass from the implicit form $F=0$ to the parametric form
\begin{equation*}
\left\{\begin{array}{l}u=\sin\phi \,,  \\u_1=\cos\phi \,.  \end{array}\right.
\end{equation*}
In other words, $\Phi:(x,\phi)\in\R\times[0,2\pi]\longrightarrow (x,\sin\phi,\cos\phi)$ is (on an open and dense coordinate chart) a diffeomorphism. So, we identify
$
\omega\equiv\Phi^*(\omega)
$.
Let us compute
\begin{eqnarray}\label{eqCosFiDeFiMenoX}
\Phi^*(\omega)=\Phi^*(du-u_1dx) &=& d\Phi^*(u)-\Phi^*(u_1)\,d\Phi^*(x)=d(\sin\phi)-\cos\phi \,dx=\cos\phi \,d\phi -\cos\phi \,dx\nonumber\\
&=& \cos\phi\, d(\phi-x)\,.
\end{eqnarray}
Hence, in the points where $\cos\phi\neq 0$, i.e., $\phi\not\in\{\frac{\pi}{2},\frac{3\pi}{2}\}$, the kernel
of $\Phi^*(\omega)$ is the same as the kernel of $d(\phi-x)$, which is an exact form with potential $\phi-x$.
In other words, the function $\phi-x$ is constant along the solutions which do not pass through $u_1=0$, i.e.,
such solutions are given by the family of curves $\phi=x+c$, $c\in\R$,
which is precisely the family \eqref{eqEquazioneSegnataConDelta} found before.
The advantage of this second approach is that we did not have to integrate explicitly the vector field(s) belonging to $\CC_\F$. On the other hand,  \eqref{eqCosFiDeFiMenoX} vanishes on the points $(x,\frac{\pi}{2})$ and $(x,\frac{3\pi}{2})$, which correspond to the two curves \eqref{eqEquazioneSegnataConBOX}. Represented on the $(x,u)$--plane, the family \eqref{eqEquazioneSegnataConDelta} appears to be enveloping the two singular solutions \eqref{eqEquazioneSegnataConBOX}. See also Fig. \ref{Fig.Enveloping}.
\begin{figure}
\epsfig{file=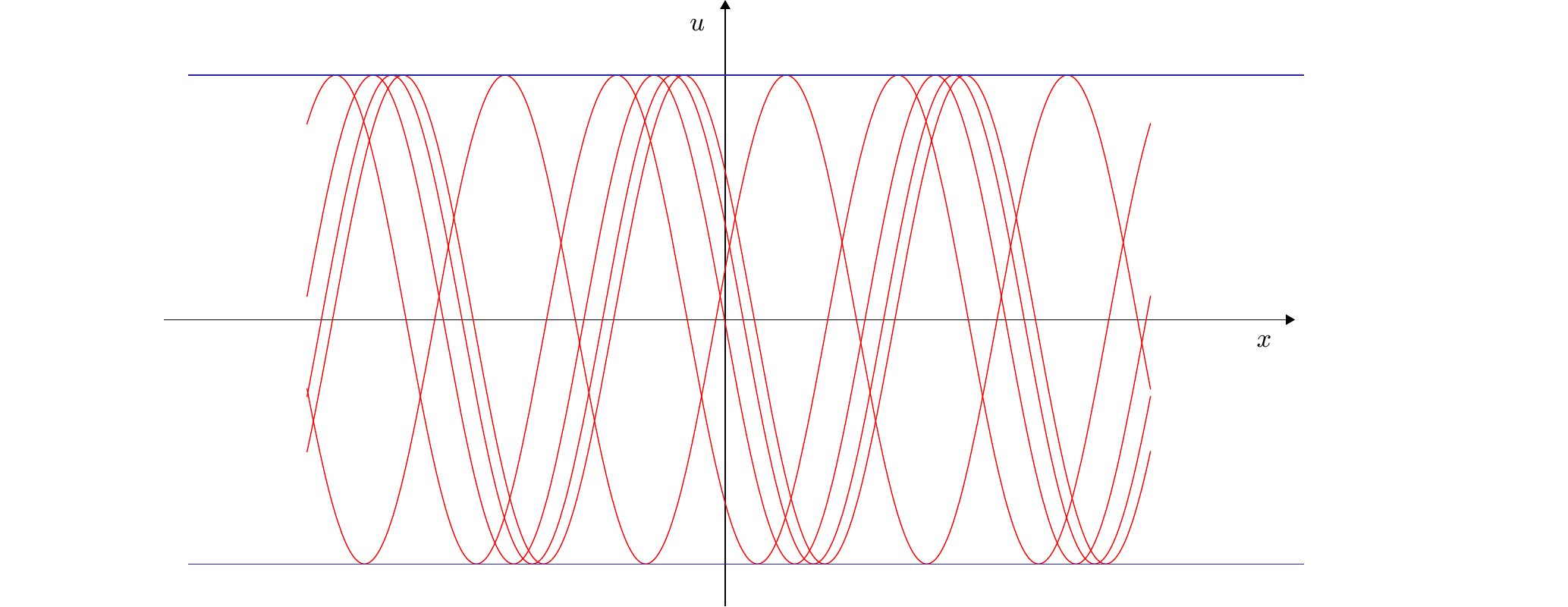,width=\textwidth}
 \caption{Singular and regular solutions to $(u')^2+u^2=1$.\label{Fig.Enveloping}}
\end{figure}

\subsection{Characteristic Cauchy data for $k\Th$ order PDEs in $n$ independent and $m$ dependent variables}\label{sezioneZarinista} In this section  we  consider the most general case of a system of (non--linear) PDEs imposed on $m$ functions of $n$ variables. Hopefully, this will help the reader to gain a geometric intuition on the notion of characteristic Cauchy data. More details can be found in \cite{Vitagliano2013}, where it is given also overview of characteristics in physically relevant contexts. The reader is also invited to have a look at \cite{MorenoCauchy} for a geometric study of Cauchy data.
\par
Denote by $\u:=(u^1,\ldots, u^m)$ a point of $\R^m$, and by
\begin{equation}\label{eqSplitCoord}
(\x=(x^1,\ldots, x^{n-1}),t)\in\R^{n-1}\times\R
\end{equation}
a point of the space--time.
Then, on $J^k(n,m)$ we have coordinates
$
(\x,t,\u,\u_{l,J})
$,
where $u_{j,J}$, similarly to \eqref{eq.uJ}, are defined by
$$
\u_{l,J}\circ j^k\phi=\frac{\partial^{l+|J|}\phi}{\partial t^l\partial x^J}:=\frac{\partial^{l+h}\phi}{\partial t^l\partial x^{j_1}\cdots\partial x^{j_h}}
$$
where $\phi:\R^n=\R^{n-1}\times\R\to\R^m$ is in $C^k(\R^n,\R^m)$, $J= j_1\cdots j_h$ is a multi--index and $|J|:=h\leq k-l$.

Similarly as we did in Section \ref{sec.Scalar.ODEs.order1}, one can see that horizontal integral $n$--dimensional submanifolds of the Cartan distribution \eqref{eq.Cartan.distribution.general} are the image of $j^k\phi$ for some $\phi$.
\begin{definition}
 A (system of $m$) $k^\textrm{th}$ order PDEs $\E\subset J^r(n,m)$ is in \emph{normal form} if and only if
 \begin{equation}\label{eqSistNormForm}
\E=\left\{   \u_{k,0}=f(\x,t,\u,\u_{l,J})    \right\}\, ,\quad l<k\, .
\end{equation}
\end{definition}
\begin{theorem}[Cauchy--Kowalewskaya \cite{Petrovsky}]\label{thCK}
If the equation $\E$ is in normal form \eqref{eqSistNormForm} with $f$ analytic, then there is a unique solution to the Cauchy problem
\begin{eqnarray}
\u_{k,0}&=&f(\x,t,\u,\u_{l,J})\, ,\label{eqNormF1}\\
 \left.\u_{l,0}\right|_{t=t_0}&=&\boldsymbol{h}_l(\x)\,,\,\,l<k\, ,\label{eqNormF2}
\end{eqnarray}
where the initial data $\boldsymbol{h}_l(\x)$ are analytic on the Cauchy surface $\Sigma:=\{t=t_0\}$.
\end{theorem}
Coordinates \eqref{eqSplitCoord} are known to physicists as ``split space--time coordinates'', i.e., reflecting an (unnatural) subdivision between space variables  and time variables. The commonly adopted viewpoint is the \emph{covariant} one, i.e., with space--time coordinates
$\x=(x^1,\ldots, x^{n-1},t)=(x^1,\ldots, x^n)$. Accordingly, on $J^k(n,m)$ we have coordinates
$(\x,\u,\u_{I})$,
where now $I$ is a multi--index of length $|I|\leq k$ defined similarly as above. Below we shall adopt such a notation. For pedagogical purposes, we shall explain the Cauchy problem, together with its characteristicity, in the case of quasi--linear systems. This will be useful for understanding the ``general'' Cauchy problem that will be treated from a contact--geometrical viewpoint in Section \ref{secAntipatetica} in the case of scalar $2\Nd$ order PDEs.
\begin{definition}
A system $\E\subset J^k(n,m)$ of $m$ equations of order $k$ of the form
$$
{A}^{i_1\cdots i_k,h}_a(x^1,\ldots,x^n,\ldots,\u_I,\ldots)\cdot u^a_{i_1\cdots i_k}=g^h(x^1,\ldots,x^n,\ldots,\u_I,\ldots)
$$
or, in short,
\begin{equation}\label{eqSisQLdetermined}
\boldsymbol{A}^{i_1\cdots i_k}(x^1,\ldots,x^n,\ldots,\u_I,\ldots)\cdot \u_{i_1\cdots i_k}=\boldsymbol{g}(x^1,\ldots,x^n,\ldots,\u_I,\ldots)\,,
\end{equation}
where $\boldsymbol{A}^{i_1\cdots i_k}$ is an $m\times m$ matrix, is called \emph{determined} and \emph{quasi--linear}.
\end{definition}
By comparing \eqref{eqSistNormForm} with \eqref{eqSisQLdetermined}, one notice that both are determined (i.e., the number of equations coincides with the number of unknown functions, which is $m$), but only the first one is solved with respect to the top space derivatives.
We wish to understand what are the obstruction to the possibility of bringing an equation of the form \eqref{eqSisQLdetermined} to the normal form \eqref{eqSistNormForm}.\par
\begin{proposition}
If the space--time coordinates change as
$\x\longmapsto\overline{\x}=\overline{\x}(\x)
$,
then
\begin{equation}\label{eqCoordchange}
\boldsymbol{A}^{i_1\cdots i_k}\longmapsto \frac{\partial\overline{x}^{i_1}}{\partial {x}^{j_1}}\cdots \frac{\partial\overline{x}^{i_k}}{\partial {x}^{j_k}}
\boldsymbol{A}^{j_1\cdots j_k}\,.
\end{equation}
\end{proposition}
In other words, $\boldsymbol{A}^{i_1\cdots i_k}$ is a symmetric contravariant tensor of order $k$, taking values in the space of $m\times m$ matrices. In particular, for scalar PDEs (i.e., $m=1$ is both the number of equations and of dependent variables), it is just a honest scalar--valued symmetric tensor.
Let now assume that the system \eqref{eqSisQLdetermined} is complemented by the Cauchy datum
\begin{equation}\label{eqCauchyDatumKZ}
 \left.\frac{\partial^{l}\u}{\partial z^l}\right|_{z= 0}=\boldsymbol{h}_l\, ,\quad l<k\, ,
\end{equation}
where $\boldsymbol{h}_l$ are functions on $\Sigma:=\{z=0\}$, an arbitrary Cauchy surface. Then, assuming $z$ to be regular, we can think of $z$ as the ``time coordinate'' of a new split coordinate system $(\boldsymbol{y},z)$, and then perform the change
$(x^1,\ldots,x^n)\longmapsto(\boldsymbol{y},z)$.
According to \eqref{eqCoordchange}, the equations become
\begin{equation}\label{eqQuasiQuasiCiSiamo}
\frac{\partial {z}}{\partial {x}^{j_1}}\cdots \frac{\partial {z}}{\partial {x}^{j_k}}\boldsymbol{A}^{j_1\cdots j_k}\cdot \frac{\partial^k\u}{\partial z^k}=\overline{\boldsymbol{f}}(\boldsymbol{y},z,\ldots,\overline{\u}_{l,J},\ldots)\, ,\quad l<k\, ,
\end{equation}
i.e., we have collected on the left the terms with $(i_1,\ldots, i_k)=(1,\ldots,1)$, while all the others are absorbed in $\boldsymbol{f}$. Now, to bring \eqref{eqQuasiQuasiCiSiamo} to the normal form \eqref{eqSistNormForm}, it remain to insulate the vector  $ \frac{\partial^k\u}{\partial z^k}$ at the left--hand side, which, in turn, is possible if and only if the matrix
\begin{equation}\label{eqMatriceCheDovrebbeInvertirsi}
\frac{\partial {z}}{\partial {x}^{j_1}}\cdots \frac{\partial {z}}{\partial {x}^{j_k}}\boldsymbol{A}^{j_1\cdots j_k}
\end{equation}
in front of it is invertible. So, the invertibility of \eqref{eqMatriceCheDovrebbeInvertirsi} is the expected obstruction.
\begin{definition}\label{defDefinizioneCAuchyDATUM}
 A Cauchy datum \eqref{eqCauchyDatumKZ} on the Cauchy surface $\Sigma=\{z=0\}$ is \emph{characteristic} if   the $m\times m$ matrix \eqref{eqMatriceCheDovrebbeInvertirsi} is not invertible.
\end{definition}
By putting together all the reasonings and the results of this subsection, we arrive to the following theorem.
\begin{theorem}\label{thExsUniqness}
 If the Cauchy datum \eqref{eqCauchyDatumKZ}  is not characteristic, then the Cauchy problem  \eqref{eqSisQLdetermined}--\eqref{eqCauchyDatumKZ} can be brought to the normal form \eqref{eqNormF1}--\eqref{eqNormF2}.
\end{theorem}
Then one apply the Cauchy--Kowalewskaya Theorem to prove uniqueness of solution. The locus where the Cauchy--Kowalewskaya theorem fails in uniqueness corresponds to the so--called ``wave--fronts'' in the equations of hydrodynamics. A wave--front is precisely the place where an ``unperturbed state'' bifurcates and becomes a ``perturbed state'', i.e., a wave. See also Fig. \ref{Fig.Planes}.

By looking at \eqref{eqMatriceCheDovrebbeInvertirsi} and Definition \ref{defDefinizioneCAuchyDATUM}, it is  important to consider the equation
\begin{equation}\label{eqCharVAr}
\det \large( p_{i_1}\cdots p_{i_k}\boldsymbol{A}^{i_1\cdots i_k} \large)=0\, ,
\end{equation}
imposed on covectors $p_idx^i$ of the space--time. Observe that, for $m=1$ the determinant can be omitted, and the equation \eqref{eqCharVAr} is homogeneous of degree $k$. It then defines a degree--$k$ hypersurface in the projectivised cotangent space to the space--time.
\begin{definition}\label{def.char.luca}
The covector $p_idx^i$ is called \emph{characteristic} and the variety cut out by \eqref{eqCharVAr} is called the \emph{characteristic variety} (at a given point of $\E$). The form $p_{i_1}\cdots p_{i_k}\boldsymbol{A}^{i_1\cdots i_k}$ is called the \emph{symbol} of system \eqref{eqSisQLdetermined}.
\end{definition}
As we said, we shall explain from a contact--geometrical viewpoint the object described in Definition \ref{def.char.luca} in the context od scalar $2\Nd$ order PDE.

\begin{figure}
  {\epsfig{file=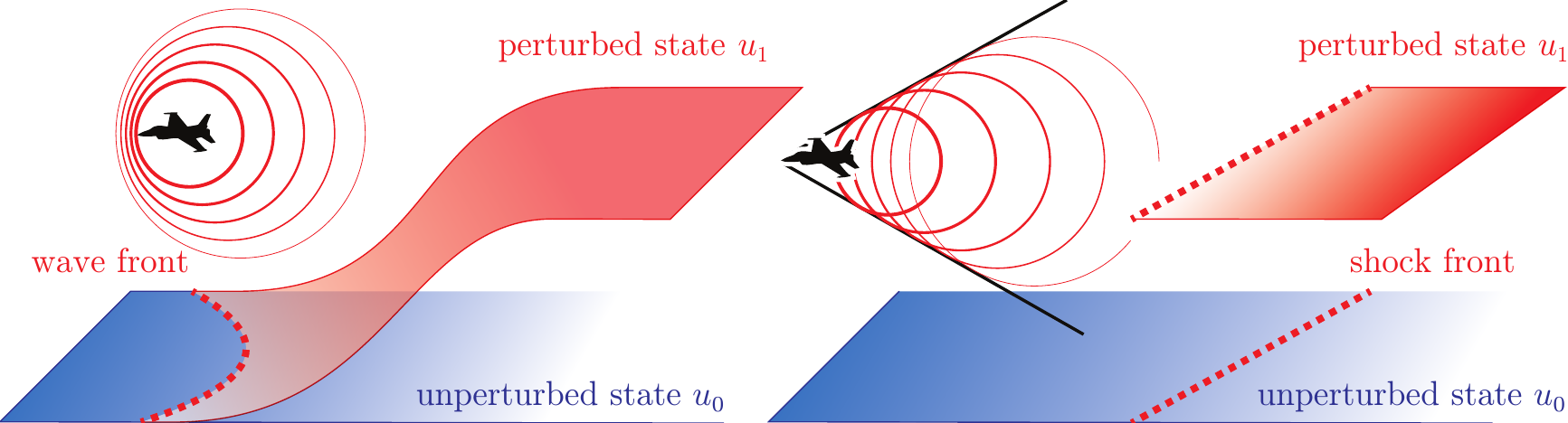,width=0.5\textwidth}}
  \caption{A characteristic initial surface admits a tangible physical interpretation, when describing a flying body going supersonic. It is exactly the place where people on the ground hear the ``boom''. \label{Fig.Planes}}
\end{figure}

\subsection{Characteristics  for $2\Nd$ order scalar PDEs in $n$ independent  variables}\label{secAntipatetica} Now we go back to the setting based on a $(2n+1)$--dimensional contact manifold $M$ (with Darboux coordinates $x^i,u,u_i$) and its prolongation $M^{(1)}$. This allows us to interpret geometrically the reasonings and the results of Section \ref{sezioneZarinista} in the case of scalar $2\Nd$ order PDEs. We shall see how the equation \eqref{eqCharVAr} becomes particularly simple and clarify why (in the case of $2\Nd$ order PDEs) vectors of rank $1$ in the sense of Definition \ref{def.rank.vector} are actually characteristic.
This will pave the way to the geometric description of Monge--Amp\`ere equations carried out in Sections \ref{secMAE}, \ref{sec.multidim.MAE} and \ref{secComplessa} below. Let us consider the scalar
$2\Nd$ order PDE with one unknown function $u$ (see Definition \ref{def2ndOrdPDE})
\begin{equation}\label{equation}
\mathcal{E}:\,F(x^{1},\dots,x^{n},u,u_{1},\dots,u_{n},u_{11},u_{12},\dots u_{nn}
)=0.
\end{equation}
The Cauchy problem
consists in finding a solution $u=f(x^{1},\dots,x^{n})$ of \eqref{equation}
such that
\begin{equation*}
f\big(X^{1}(\mathbf{t}),\dots,X^{n}(\mathbf{t})\big)=U(\mathbf{t})\,,\,\,\,
\frac{\partial f}{\partial z} {\big(X^{1}(\mathbf{t}),\dots,X^{n}(\mathbf{t})\big)}=Q(\mathbf{t}), \,\,\ell\leq 1,
\end{equation*}
where $X^i(\mathbf{t}), U(\mathbf{t}),Q(\mathbf{t})$ are given functions of $\mathbf{t}=(t_1,\dots,t_{n-1})$ and $\frac{\partial}{\partial z}$ is the derivative along a direction transversal to the $(n-1)$--dimensional submanifold parametrically described by $\{\big(X^{1}(\mathbf{t}),\dots,X^{n}(\mathbf{t})\big)\}$. In order to prescribe such a directional derivative, it is sufficient to know the gradient of $f$ along the aforementioned $(n-1)$--dimensional submanifold, i.e.,
\begin{equation}\label{Cauchy_problem}
f\big(X^{1}(\mathbf{t}),\dots,X^{n}(\mathbf{t})\big)=U(\mathbf{t})\,,\,\,\,\left.
\frac{\partial f}{\partial x^{i}}\right\vert _{(X^{1}(\mathbf{t}),\dots
,X^{n}(\mathbf{t}))}=U_{i}(\mathbf{t})\,.
\end{equation}
Equations \eqref{Cauchy_problem} describe parametrically the $(n-1)$--dimensional submanifold
\begin{equation}\label{Cauchy_datum}
\Sigma=\{(X^{1}(\mathbf{t}),\dots,X^{n}(\mathbf{t}),Z(\mathbf{t}),U_{1}(\mathbf{t}),\dots,U_{n}(\mathbf{t}))\}\,,
\quad
\mathbf{t}=(t_{1},\dots,t_{n-1})\,,
\end{equation}
of $M$. The submanifold \eqref{Cauchy_datum} is called a \emph{Cauchy datum} of $\mathcal{E}$
(obviously, the particular choice of the parametrization is irrelevant). Note that the tangent space to the above Cauchy datum is contained in the contact distribution $\CC$ of $M$. This allows us to specialise the definition of Cauhy datum already given in Section \ref{sezioneZarinista}.
\begin{definition}\label{def.Cauchy.datum.2.order}
A Cauchy datum for a $2\Nd$ order PDE is an $(n-1)$--dimensional integral submanifold of the contact distribution $\CC$.
\end{definition}
Now we see what are, in this context, characteristic hyperplanes (that can be identified with characteristic covectors) and characteristic vectors. Tangent vectors of rank $1$ (recall Definition \ref{def.rank.vector}) at a point $L\in\mathcal{E}$, are, up to sign,
\begin{equation}\label{eq.vector.rank.1.vero}
\sum_{i\leq j}\xi_i\xi_j\partial_{u_{ij}}\,,
\end{equation}
where $\xi_k$ are the components, with respect to the basis $\{dx^i\}$ of $L^*$, of the covector $\xi$ defined by \eqref{eq.rank.1.vector}. Vector \eqref{eq.vector.rank.1.vero} is tangent to the equation \eqref{equation} if and only if
\begin{equation}\label{eq.symb.Gianni}
\sum_{i\leq j}F_{u_{ij}}\xi_i\xi_j=0\,.
\end{equation}
Equation \eqref{eq.symb.Gianni} coincides with equation \eqref{eqCharVAr}, and the left--hand side term of \eqref{eq.symb.Gianni} coincides with the symbol of the PDE \eqref{equation}
if we assume $F$ to be quasi--linear (recall that here we are treating only the second order case). This motivates the following definition.
\begin{definition}\label{def.symbol.gianni}
Let $\E:=\{F=0\}\subset M^{(1)}$ be a $2\Nd$ order scalar PDE. Then its symbol at a point $L\in\E$ is the differential $d(F_{\pi(L)})$ (computed at $L$) where $F_{\pi(L)}$ means the function $F$ with $\pi(L)$ fixed (thus, $F_{\pi(L)}$ is a function depending only on $u_{ij}$).
\end{definition}
Locally, $d(F_{\pi(L)})$ reads
\begin{equation}\label{eq.symb.Gianni.2}
\sum_{i\leq j} F_{u_{ij}}du_{ij}\overset{\eqref{eqCanBundIso}}{\simeq}\sum_{i\leq j}F_{u_{ij}}D_i\odot D_j = \sum_{i\leq j} F_{u_{ij}}(\partial_{x^i}+u_i\partial_{u}+u_{ik}\partial_{u_k})\odot(\partial_{x^j}+u_j\partial_{u}+u_{kj}\partial_{u_k})
\end{equation}
where $u_{ij}$ of the above expression are computed in the point $L\in\E$. By denoting $D_i$ with $\xi_i$, we obtain expression \eqref{eq.symb.Gianni}.

\smallskip
Now  the correspondence \eqref{eq.correspondence.gianni} allows us to decide whether
or not a Cauchy datum is characteristic for a PDE \eqref{equation}. The definitions given in Section \ref{sezioneZarinista}, in the case of $2\Nd$ order PDEs, in the context outlined in the present section, become as follows.
\begin{definition}\label{def.char.tutto}
Let $\Sigma$ be a Cauchy datum. Then $T_p\Sigma$ (that we can interpret as an ``infinitesimal Cauchy datum'') is characteristic for the equation $\E$ at a point $L\in\E$ if the rank--$1$ direction corresponding to $T_p\Sigma$ through \eqref{eq.correspondence.gianni} is tangent to $\E$ at $L$. Such rank--$1$ direction is also called a \emph{characteristic line}. In this case, $T_p\Sigma$ is said also \emph{characteristic hyperplane} and its annihilator (unique in $\p L^*$, see again correspondence \eqref{eq.correspondence.gianni}) a \emph{characteristic covector}. A Cauchy datum $\Sigma$ is not characteristic for the PDE $\E$ if there are no points where it is characteristic for $\E$. The \emph{characteristic variety} at $L$ is the union of all characteristic hyperplanes at $L$.
\end{definition}

\subsubsection{An example of computation} Let $\E$ be a $2\Nd$ order PDE in $2$ independent variables.
A Cauchy datum is an integral curve $\gamma$ of $\CC=\Span{D_{x^1},D_{x^2},\partial_{u_1},\partial_{u_2}}$. Let us fix a point $p\in\gamma$ and denote by $\ell$ the tangent space to $\gamma$ at $p$, $\ell=T_p\gamma$, i.e., the ``infinitesimal Cauchy datum'' (according to Definition \ref{def.char.tutto}). A ``candidate infinitesimal solution'' to the Cauchy problem is any tangent space to a solution containing $\ell$, i.e., a Lagrangian plane $L$ such that $L\supset\ell$ and $L$ belongs to $\E$, i.e., in other words, the prolongation $\ell^{(1)}$ of $\ell$ (see Definition \ref{eq.def.prolog.lag.submanifold}) intersects $\E$:
$$
\ell^{(1)}\cap\E=\{   L\in\E \mid L\supset\ell \}\,.
$$
Let us fix a Lagrangian plane $L_0=L(U_0)$ containing $\ell$. Here $U_0$ is a $2\times 2$ matrix
$$
U_0=\left(
\begin{array}{cc}
u_{11}^0 & u_{12}^0
\\
u_{12}^0 & u_{22}^0
\end{array}
\right)
$$
that gives the Lagrangian plane (see also \eqref{eq.L.P.as.a.matrix})
\begin{equation*}
L(U_0)=\Span{D_{x^1}+u_{11}^0\partial_{u_1}+u_{12}^0\partial_{u_2}, D_{x^2}+u_{12}^0\partial_{u_1}+u_{22}^0\partial_{u_2}}\, .
\end{equation*}
If the line $\ell$ is given as the kernel of a covector $\xi=\xi_1dx^1+\xi_2dx^2$, i.e.,
\begin{equation*}
\ell=\ker\xi=-\xi_2( D_{x^1}+u_{11}^0\partial_{u_1}+u_{12}^0\partial_{u_2} )+\xi_1 ( D_{x^2}+u_{12}^0\partial_{u_1}+u_{22}^0\partial_{u_2}  )\, ,
\end{equation*}
then, a generic Lagrangian plane
$L(U)$ contains $\ell$ if and only if the matrix
 \begin{equation*}
\left(\begin{array}{cccc}
1 & 0 & u_{11} & u_{12} \\
0 & 1 & u_{12} & u_{22} \\
-\xi_2 & \xi_1 & -\xi_2 u_{11}^0+\xi_1 u_{12}^0 &  -\xi_2 u_{12}^0+\xi_1 u_{22}^0 \end{array}\right)
\end{equation*}
has rank $2$. This is true if and only if, by fixing $u_{12}$ as a parameter,
\begin{equation*}
u_{11} =  u_{11}^0+\frac{\xi_1}{\xi_2}(u_{12}-u_{12}^0)\,, \quad
u_{22} =  u_{22}^0 +\frac{\xi_2}{\xi_1}(u_{12}-u_{12}^0)\, .
\end{equation*}
In other words,
\begin{equation*}
\ell^{(1)}=\left\{
\left(
\begin{array}{cc}
u_{11}^0+\frac{\xi_1}{\xi_2}(u_{12}-u_{12}^0)& u_{12} \\
u_{12} & u_{22}^0 +\frac{\xi_2}{\xi_1}(u_{12}-u_{12}^0)
\end{array}
\right)\mid u_{12}\in \R
\right\}\,,
\end{equation*}
which is $1$--dimensional, as expected.
Its tangent line at $L_0$, i.e., at $u_{12}=u_{12}^0$, is
\begin{equation*}
T_{L_0}\ell^{(1)}=\Span{ \left(\begin{array}{cc}\frac{\xi_1}{\xi_2} & 1 \\1 & \frac{\xi_2}{\xi_1}\end{array}\right)   }
=\Span{ \frac{\xi_1}{\xi_2}\partial_{u_{11}}+\partial_{u_{12}}+ \frac{\xi_2}{\xi_1}\partial_{u_{22}} } = \Span{ \xi_1^2\partial_{u_{11}}+\xi_1\xi_2\partial_{u_{12}}+\xi_2^2\partial_{u_{22}}}\, .
\end{equation*}
Observe that
$$
\rank \left(\begin{array}{cc}\xi_1^2 & \xi_1\xi_2 \\\xi_1\xi_2 & \xi_2^2\end{array}\right)=1\,,
$$
i.e., $T_{L_0}\ell^{(1)}$ is a rank--$1$ direction (recall also the correspondence \eqref{eq.correspondence.gianni}). Now, accordingly to Definition \ref{def.char.tutto}, if $T_{L_0}\ell^{(1)}\subset T_{L_0}\E$, then $\ell$ is a hyperplane (of $L_0$) which is characteristic for the PDE $\E$ at the point $L_0$, and its annihilator is a characteristic covector. Moreover, the rank--$1$ line $T_{L_0}\ell^{(1)}$ is a characteristic line for $\E$ at $L_0$. We observe that (cf. equation \eqref{eq.symb.Gianni})
\begin{equation*}
T_{L_0}\ell^{(1)}\subset T_{L_0}\E\Leftrightarrow dF(\xi_1^2\partial_{u_{11}}+\xi_1\xi_2\partial_{u_{12}}+\xi_2^2\partial_{u_{22}})= F_{u_{11}}\xi_1^2 + F_{u_{12}}\xi_1\xi_2 + F_{u_{22}}\xi_2^2  =0\, .
\end{equation*}
There is another condition for $\ell$, ``stronger'' than $T_{L_0}\ell^{(1)}\subset T_{L_0}\E$, namely
\begin{equation}\label{eqCharStromg}
\ell^{(1)}\subset\E\, .
\end{equation}
\begin{definition}
If \eqref{eqCharStromg} holds, then $\ell\subset\CC_p$ is called a \emph{strong characteristic} (hyperplane) for $\E$.
\end{definition}
One can ask which kind of $2\Nd$ order PDEs are such that the characteristics  are also strong characteristics. We shall see that Monge--Amp\`ere equations satisfy this requirement. On the top of that, the characteristics show
a particularly simple geometrical meaning for these PDEs.

\section{Monge--Amp\`ere equations in two independent variables}\label{secMAE}

The importance of the Monge--Amp\`ere equation is underlined by the nearly one thousands citations,
gathered by the review book \cite{MR1964483} during the past dozen of years. This equation was always linked,
since its inception in 1781, to the problem of optimal mass transportation, whose most famous modern
reincarnation is the problem of optimal allocation of resources, tightly connected with the
Nobel prize in Economy awarded to Leonid Kantorovich in 1975.
This shows the enormous interest triggered by a \emph{particular}
Monge--Amp\`ere equation and its solutions, also in view of
its numerous  applications to natural sciences and economy.
On the purely theoretical side, there has been just a few
(though excellent) studies of the \emph{family} of all
Monge--Amp\`ere equations, which can be described as
\emph{hyperplane sections} of the Lagrangian Grassmannian
$\LL(n,2n)$ inside its Pl\"ucker embedding space
(see Definition \ref{defPluck}).\par
We explain below how to achieve this algebraic--geometric interpretation
of such second--order non--linear PDEs, by stressing the importance of their characteristics. In this section we treat only the case of $2$ independent variables. The multidimensional case, at least for those Monge--Amp\`ere equations which we call of \emph{Goursat} type (see Definition \ref{def.E.D.n.var} for the multidimensional case) can be dealt analogously. Some geometrical aspects of multidimensional Monge--Amp\`ere equations are discussed in Section \ref{sec.multidim.MAE}, where a concrete example is given to better illustrate some theoretical issues concerning the characteristics and the construction of solutions.

\subsection{Monge--Amp\`ere equations with $2$ independent variables}
 The main advantage of having $\LL(2,4)$ embedded
 into $\p^4$ (see Section \ref{sec.L24}), is that one can use hypersurfaces in $\p^4$ to ``cut out''
 hypersurfaces in $\LL(2,4)$. The latter are precisely $2\Nd$ order PDEs in two independent variables  (according to Section \ref{sec.SecondOrderPDEs}).\par
The simplest case of a hypersurface in $\p^4$ is that of a hyperplane
$
\Pi:=\{ Dz_0+Az_1+Bz_2+Cz_3+Nz_4=0\}\subset\p^4
$.
In this case,
\begin{equation}\label{eqHypSect}
\Pi\cap\LL(2,4)=\left\{       U=\left(\begin{array}{cc}u_{11} & u_{12} \\u_{12} & u_{22}\end{array}\right) \mid  D+Au_{11}+Bu_{12}+Cu_{22}+N(u_{11}u_{22}-u_{12}^2)=0 \right\}\,.
\end{equation}
\begin{definition}\label{defHypSec}
 The hypersurface \eqref{eqHypSect} is called a \emph{hyperplane section} of $\LL(2,4)$.
\end{definition}
Let now regard the 8--dimensional jet space
$
J^2(2,1)=\{  (x^1,x^2,u,u_1,u_2,u_{11},u_{12},u_{22})\}
$
as the Lagrangian Grassmannian bundle of the 5--dimensional contact manifold
$
J^1(2,1)= \{  (x^1,x^2,u,u_1,u_2)\}
$
 (see Section \ref{sec.LagGrassPrologM}).
\begin{definition}\label{defMAE}
 The sub--bundle
 \begin{equation}\label{eqDefMAE}
\E:=\{  D+Au_{11}+Bu_{12}+Cu_{22}+N(u_{11}u_{22}-u_{12}^2)=0\}\subset J^2(2,1)\, ,
\end{equation}
where $D,A,B,C,N\in C^\infty(J^1(2,1))$, is a \emph{Monge--Amp\`ere equation} (in two independent variables).
\end{definition}
Observe that, for each point $p\in J^1(2,1)$, the fibre $\E_p$ is a hyperplane section of $\LL(\CC_p)=\LL(2,4)$.\par
Let us study now the behavior of the Cauchy problem (see Sections \ref{sezioneZarinista} and \ref{secAntipatetica}) for the Monge--Amp\`ere equations \eqref{eqDefMAE}.
The symbol (see the left hand side of \eqref{eq.symb.Gianni} and \eqref{eq.symb.Gianni.2}) of \eqref{eqDefMAE} is
\begin{equation}\label{eqSmblMAE}
(A+Nu_{22})\xi_1^2+(B-2Nu_{12})\xi_1\xi_2+(C+Nu_{11}(u_{12},u_{22}))\xi_2^2\, ,
\end{equation}
where
\begin{equation}\label{eq.r}
u_{11}(u_{12},u_{22})=-\frac{D+Bu_{12}+Cu_{22}-Nu_{12}^2}{A+Nu_{22}}
\end{equation}
has been computed from \eqref{eqDefMAE}. The discriminant of \eqref{eqSmblMAE} is
\begin{equation*}
(B-2Nu_{12})^2-4(A+Nu_{22})(C+Nu_{11})
\end{equation*}
which, on the points of PDE \eqref{eqDefMAE}, reads
\begin{equation}\label{eqDEltaGiusta}
B^2-4AC+4ND=:\Delta\, .
\end{equation}
So, according to $\Delta$ being negative, zero, or positive, the Monge--Amp\`ere equation \eqref{eqDefMAE} will be elliptic, parabolic or hyperbolic, respectively. This means that the zero locus of the symbol \eqref{eqSmblMAE} will be empty, a line, or the union of two distinct lines, respectively. If the annihilator of the Cauchy datum lies on such a zero locus, then the Cauchy problem will be ill--defined (cf. Definition \ref{def.char.tutto}). In particular, for elliptic Monge--Amp\`ere equations, the Cauchy problem is \emph{always} well--defined.\par
Passing now to non--elliptic Monge--Amp\`ere equations, that is, equations of the form
\begin{equation}\label{eq.MAE.hyp.par}
\det( {U}- {F})=  \det(F) -f_{22}u_{11} + (f_{12}+f_{21})u_{12} -f_{11}u_{22} + \det(U)=0\,,
\end{equation}
where $U=(u_{ij})$ and $F=(f_{ij})$, with $f_{ij}=f_{ij}(x^1,x^2,u,u_1,u_2)$. We see that the discriminant of \eqref{eq.MAE.hyp.par} is equal to
\begin{equation}\label{eqNonEllitticita}
\Delta= {(f_{12}-f_{21})}^2\geq 0\, .
\end{equation}
Now $F$ can be identified with a $2$--dimensional sub--distribution of $\CC$, viz.
\begin{equation*}
F\longleftrightarrow \Span{D_{x^1}+f_{11}\partial_{u_1}+f_{12}\partial_{u_2}, D_{x^2}+f_{21}\partial_{u_1}+f_{22}\partial_{u_2}}\,.
\end{equation*}
In this perspective, $\det( {U}- {F})=0$ is the equation of the Lagrangian  planes, i.e., the elements of $J^2(2,1)$, which are non--transversal to $F$. Indeed,
\begin{equation}\label{eqDeterminanteGasante}
L(U)\cap F\neq 0\Leftrightarrow \det \left(\begin{array}{cccc}1 & 0 & f_{11} & f_{12} \\0 & 1 & f_{21} & f_{22} \\1 & 0 & u_{11} & u_{12} \\0 & 1 & u_{12} & u_{22}\end{array}\right)=\det(U-F)=0\, .
\end{equation}
We will show that these are the \emph{only} examples of non--elliptic Monge--Amp\`ere equations.

\smallskip
Let us begin with a $5$--dimensional contact manifold
$ (M,\CC)$. Let $\D$ be a $2$-dimensional sub--distribution of $\CC$. It is always possible to choose a contact chart
$(x^1,x^2,u,u_1,u_2)$ on $M$ such that
\begin{equation}\label{eq.D.2.dim}
\D=\Span{D_{x^1}+f_{11}\partial_{u_1}+f_{12}\partial_{u_2}, D_{x^2}+f_{21}\partial_{u_1}+f_{22}\partial_{u_2}}
\end{equation}
with $f_{ij}\in C^\infty(M)$, where we recall that $D_{x^i}$ are defined by \eqref{eq.contact.local}.
\begin{definition}\label{def.E.D.2.var}
Let $\D\subset\CC$ be a $2$--dimensional sub--distribution of $\CC$. Then the sub--bundle
 \begin{equation}\label{eq.ED.prima}
\E_\D:=\{  L\in M^{(1)}_p\mid L\cap\D_p\neq 0\}\subset M^{(1)}
\end{equation}
is called the \emph{Goursat--type Monge--Amp\`ere equation} associated with $\D$ (see Fig. \ref{Fig.E.con.D}). The distribution $\D$ is called the \emph{characteristic distribution} of the Monge--Amp\`ere equation $\E_\D$.
\end{definition}
The substitution $f_{12}\leftrightarrow f_{21}$ does not change the PDE \eqref{eq.MAE.hyp.par}. This shows that, if we consider the transpose of the matrix $F$, we obtain the same Monge--Amp\`ere obtained from $F$. Essentially, this means that $\mathcal{D}$ and $\mathcal{D}^\perp$ in definition \eqref{def.E.D.2.var} give the same Monge--Amp\`ere equation.
\begin{proposition}\label{prop.D.D.ort}
Let $\D\subset\CC$ be a $2$--dimensional sub--distribution of $\CC$. Then $\E_D=\E_{D^\perp}$.
\end{proposition}
\begin{figure}
\centerline{\epsfig{file=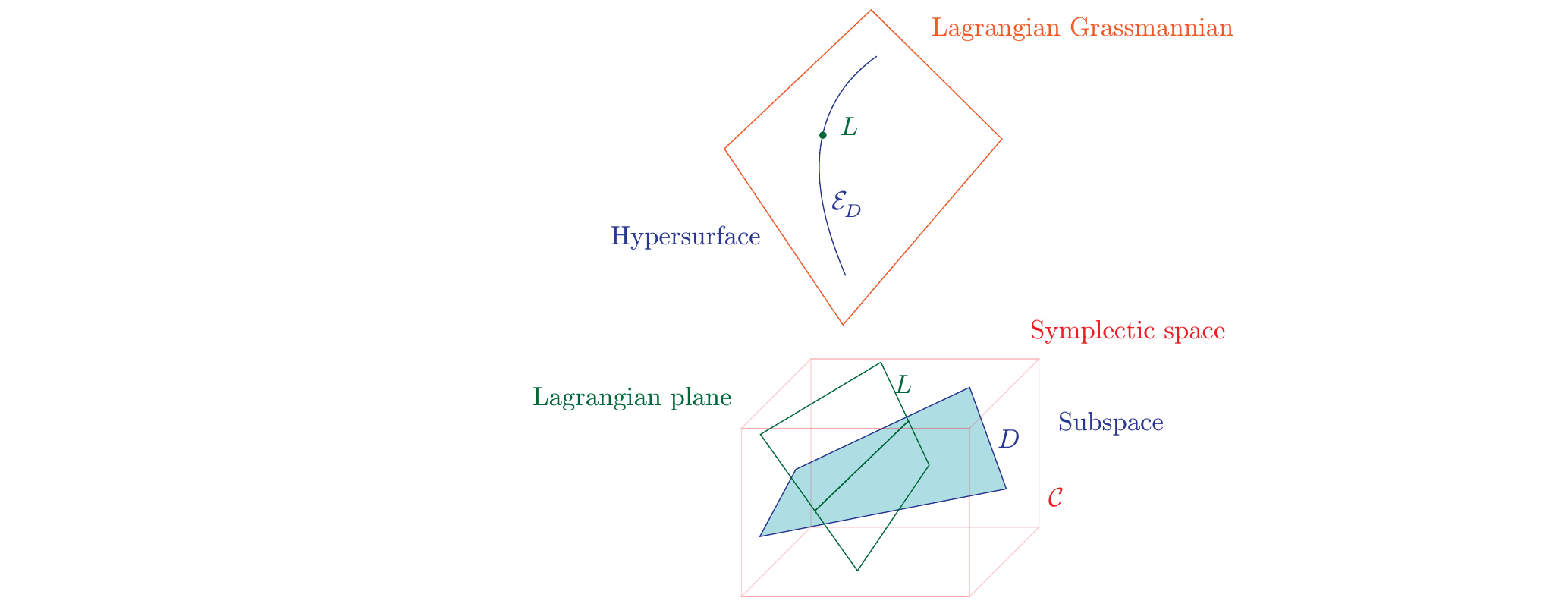,width=\textwidth}}
 \caption{Selecting only those Lagrangian planes that are non--transversal to $D$, one gets a hypersurface in $\E_D$ in the Lagrangian Grassmannian.\label{Fig.E.con.D}}
\end{figure}
We have already discovered (see formula \eqref{eqDeterminanteGasante}) that, locally, $\E_\D$ can be described as
\begin{equation*}
\det\left(\begin{array}{cc}u_{11}-f_{11} & u_{12}-f_{12} \\u_{12}-f_{21} & u_{22}-f_{22}\end{array}\right)=0\, ,
\end{equation*}
where $(u_{11},u_{12},u_{22})$ are the fibre coordinates of $M^{(1)}$. We also have proved (see \eqref{eqNonEllitticita}) that $\E_\D$ is non--elliptic, i.e., it possesses characteristics.  Now we show the converse by using geometrical properties of the characteristics of the Monge--Amp\`ere equation \eqref{eqDefMAE}.
\begin{theorem}\label{Th.non.elliptic.MAEs}
Any non--elliptic Monge--Amp\`ere equation is of Goursat type.
\end{theorem}
\begin{proof}
The proof could be carried out by means of a straightforward computation: given the equation \eqref{eqDefMAE} one can find the $f_{ij}$'s from \eqref{eq.MAE.hyp.par} by solving a system and then substitute them in \eqref{eq.D.2.dim} in order to obtain the distribution $\D$. Instead, we choose a different method, relying on the geometry of characteristics and symbol of PDE \eqref{eqDefMAE} in order to construct the distribution $\D$.

\smallskip
The symbol of the equation \eqref{eqDefMAE} is \eqref{eqSmblMAE} with $u_{11}(u_{12},u_{22})$ given by \eqref{eq.r}. Up to a non--zero factor, this symbol is equal to
\begin{equation}\label{eq.symb.MAE.2.dim}
\left((A+Nu_{22})\xi_1-\left(-\frac12 B+Nu_{12}+\frac12\sqrt{\Delta}\right)\xi_2\right)
\cdot\left((A+Nu_{22})\xi_1-\left(-\frac12 B+Nu_{12}-\frac12\sqrt{\Delta}\right)\xi_2\right)\,.
\end{equation}
Let us consider the first factor of \eqref{eq.symb.MAE.2.dim}. To this it is associated a line in $\CC$ depending on the point of the PDE, i.e., on $u_{12}$ and $u_{22}$:
\begin{equation}\label{eq.nonso.proprio}
\ell(u_{12},u_{22})=(A+Nu_{22})(\partial_{x^1}+u_1\partial_u+u_{11}(u_{12},u_{22})\partial_{u_1}+u_{12}\partial_{u_2})-\left(-\frac12 B+Nu_{12}+\frac12\sqrt{\Delta}\right)(\partial_{x^2}+u_2\partial_u+u_{12}\partial_{u_1}+u_{22}\partial_{u_2})\,,
\end{equation}
where $u_{11}(u_{12},u_{22})$ is given by \eqref{eq.r}. A computation shows that \eqref{eq.nonso.proprio} is equal to
\begin{equation}\label{eq.nonso.proprio.2}
\ell(u_{12},u_{22})=\left(-N D_2 -\frac12 (B+\sqrt{\Delta})\partial_{u_1}+A\partial_{u_2}\right)u_{12} + \left(N D_1 -C\partial_{u_1} +\frac12 (B-\sqrt{\Delta})\partial_{u_2}\right)u_{22}\,.
\end{equation}
If we let the point $(u_{11}(u_{12},u_{22}),u_{12},u_{22})$ vary on the equation \eqref{eqDefMAE} (i.e., we let $u_{12},u_{22}$ vary arbitrarily in $\R$), the line \eqref{eq.nonso.proprio.2} describes a $2$-dimensional distribution $\D$, i.e.,
$$
\D=\bigcup_{u_{12},u_{22}\in\R}\ell(u_{12},u_{22}) = \Span{-N D_2 -\frac12 (B+\sqrt{\Delta})\partial_{u_1}+A\partial_{u_2} \,,\, N D_1 -C\partial_{u_1} +\frac12 (B-\sqrt{\Delta})\partial_{u_2} }\,.
$$
Now it is easy to realize that the PDE \eqref{eqDefMAE} is equal to $\E_\D$ with $\D$ given above.

\smallskip\noindent
If we had chosen the second factor of \eqref{eq.symb.MAE.2.dim}, by doing the same reasoning as above, we would have come to the orthogonal distribution $\D^\perp$ of $\D$, as predicted by Proposition \ref{prop.D.D.ort}.
\end{proof}

\begin{remark}\label{rem.reconstruct}
We can say that the characteristic variety of a non--elliptic Monge--Amp\`ere equation is the union of a $2$--dimensional sub--distribution of $\CC$ and of its symplectic--orthogonal. The results we obtained so far show that we can re--construct a Monge--Amp\`ere equation starting from its characteristics.
\end{remark}

\subsection{Normal form of parabolic Monge--Amp\`ere equations }\label{secNormFormMAEs}

We have seen (see for instance Theorem \ref{Th.non.elliptic.MAEs}) that to a non--elliptic Monge--Amp\`ere equation it is associated a pair of distributions (possibly coinciding): a $2$--dimensional distribution $\D\subset\CC$ and its symplectic orthogonal $\D^\perp$. If $\D=\D^\perp$, then $\D$ is Legendrian and its associated Monge--Amp\`ere equation turns out to be parabolic as its $\Delta$ (defined by \eqref{eqDEltaGiusta}) is equal to zero (cf. also \eqref{eqNonEllitticita}): in this case we call $\D$ the \emph{characteristic} distribution of the considered parabolic Monge--Amp\`ere equation. It is easy to realize that two parabolic Monge--Amp\`ere equations are contactomorphic if and only if  so are their characteristic distributions.
Indeed, given a Legendrian distribution $\D$, one can canonically construct a parabolic Monge--Amp\`ere equation $\E_\D$ via formula \eqref{eq.ED.prima}. Conversely, given a parabolic Monge--Amp\`ere equation $\E$, one can construct a Legendrian distribution $\D$ such that $\E=\E_\D$ by using the reasonings contained in the proof of Theorem \eqref{Th.non.elliptic.MAEs}.
Thus, the problem of classifying parabolic Monge--Amp\`ere equations (with $2$ independent variables) is the same as classifying Legendrian distributions on $5$--dimensional contact manifolds. In this section we shall see how to find normal forms of Legendrian distributions on a  $5$--dimensional contact manifold and, as a byproduct, we obtain normal forms of (the corresponding) parabolic Monge--Amp\`ere equations. Our approach is based on the normal form of vector fields lying in the contact distribution $\CC$.

To this end, we need the following technical propositions (more details and ideas can be found in \cite{MR2503974}). Recall the notion of type and of Hamiltonian vector field given in Section \ref{sec.contact.manifolds.2}.
\begin{proposition}
\label{Ortogonale_Tipo_2} Let $Y\in\mathcal{C}$. Then $Y$ is of type $2$ if
and only if the derived distribution $(Y^{\bot})^{\prime}$ has dimension $4$.
Furthermore, $Y$ is a multiple of a hamiltonian field $Y_{f}$ if and only if
$(Y^{\perp})^{\prime}$ is $4$-dimensional and integrable.
\end{proposition}
\begin{proof}
Let $\dim(Y^{\bot})^{\prime}=4$. Since $Y^\bot$ is $3$-dimensional and it is contained in the contact distribution $\CC=\ker\,\theta$, in view of Lemma \ref{Lemma_Calcolo_Derivato}, the annihilator $\sigma$ of $(Y^{\bot})^{\prime}$, being such distribution $4$-dimensional by hypothesis,   is a linear combination of $\theta$ and $L_Y(\theta)$. Also, since there are no $3$-dimensional sub--distributions of $\CC$ whose derived distribution is still contained in $\CC$, we have that
\begin{equation}\label{forma_sigma}
\sigma=\alpha \theta+L_Y(\theta)\,,
\end{equation}
i.e., we can assume that the component of $\sigma$ with respect to $L_Y(\theta)$ is nowhere vanishing. Without loss of generality, we can set it equal to $1$. By Lemma \ref{Lemma_Calcolo_Derivato}, we can choose $\alpha$ from \eqref{forma_sigma} in such a way that $L_Y^2(\theta)$ depends on $\theta$ and $L_Y(\theta)$, i.e., $Y$ is of type $2$.

\smallskip\noindent
Conversely, let $Y$ be of type $2$. To prove our statement we have to find an
$\alpha$ in (\ref{forma_sigma}) such that $(Y^{\bot})'$ is described by a single equation
$\sigma=0$. To this end, let $\{X,Y,Z\}$ be a basis of $Y^{\bot}$. Then
$L_X(\sigma)$, $L_Y(\sigma)$ and $L_Z(\sigma)$ must vanish on $Y^{\bot}$. By
assumption it holds
\[
L_Y(\sigma)=L_Y^2(\theta)+L_Y(\alpha)\theta+\alpha L_Y(\theta)\equiv0\,\,\text{mod}\Span{\theta\,,\,L_Y(\theta)}
\]
and, therefore, $L_Y(\sigma)$ vanishes on $Y^{\bot}$ for any choice of $\alpha$.
As to $L_X(\sigma)$, relations
\[
L_X(\sigma)(Y)=-L_Y(\sigma)(X)=0\,,\quad L_X(\sigma)(X)=d\sigma(X,X)=0
\]
hold true for any $\alpha$, whereas the equation
\[
0=L_X(\sigma)(Z)=dY(\theta)(X,Z)+\alpha\, d\theta(X,Z)
\]
determines $\alpha$. Therefore, by choosing $\alpha$ in this way, one has that
$L_X(\sigma)$ vanishes on $Y^{\bot}$, and the same holds for $L_Z(\sigma)$.

Now we prove the second part of the statement. Assume $Y=Y_{f}$ (or
a multiple of it). Then
\[
Y^{\perp}=Y_{f}^{\perp}=\Span{Y_{f}\,,\,Y_{g}\,,\,Y_{h}}\,,
\]
with $g$ e $h$ being independent first integrals of $Y$, obviously in
involution with $f$. On the other hand,
\[
Y_{f}^{\perp}=\{\theta=L_{Y_{f}}(\theta)=0\}=\{\theta=df=0\}.
\]
Furthermore,
\[
L_{Y_{f}}(df)=L_{Y_{g}}(df)=L_{Y_{h}}(df)=0\,,
\]
i.e., by Lemma \ref{Lemma_Calcolo_Derivato},
\[
(Y_{f}^{\perp})^{\prime}=\{df=0\}\,,
\]
which is $4$-dimensional and integrable.

\smallskip Viceversa, let ($Y^{\perp})^{\prime}$ be $4$-dimensional and
integrable. Then there exists a function $f$ such that ($Y^{\perp})^{\prime
}=\{df=0\}$. Therefore
\[
Y^{\perp}=(Y^{\perp})^{\prime}\cap\mathcal{C}=\{\theta=df=0\}=\{\theta=L_{Y_{f}}
(\theta)=0\}=Y_{f}^{\perp}\,,
\]
which entails the parallelism between $Y$ and $Y_{f}$.
\end{proof}

\begin{proposition}\label{Campi_Tipo2}
Let $Y\in\mathcal{C}$. Then
\begin{enumerate}
\item[1)] $Y$ is a Hamiltonian vector field if and only if, in a suitable contact chart $(x^1,x^2,u,u_1,u_2)$,
$$
Y=\partial_{u_1}\,;
$$
\item[2)] $Y$ is of type $2$ but not Hamiltonian if and only if, in a suitable contact chart $(x^1,x^2,u,u_1,u_2)$,
$$
Y=\partial_{u_1}+u \partial_{u_2}\,;
$$
\item[3)] $Y$ is of type $3$ if and only if, in a suitable contact chart $(x^1,x^2,u,u_1,u_2)$.
\begin{equation*}
Y=a\partial_{u_1}+b\partial_{u_2}\,,\,\, \text{with}\,\, Y(b/a)\neq0;
\end{equation*}
\end{enumerate}
\end{proposition}
\begin{proof}
Let us prove statements $1)$ and $2)$. Let $(Y^{\bot})^{\prime}$ be locally described by the equation
$\sigma=0$ (see also Proposition \ref{Ortogonale_Tipo_2}). Then (see Section \ref{sec.non.int} and \cite[Section 3.2]{MR1300019}),
one can choose independent functions $f,g,h,k,l$ in such a way that, up to a non--zero
factor, one of the following three expressions holds:
\begin{equation}
\sigma=df \label{(4,5)_zera}
\end{equation}
or
\begin{equation}
\sigma=df-gdh \label{(4,5)_prima}
\end{equation}
or
\begin{equation}
\sigma=df-gdh-kdl. \label{(4,5)_seconda}
\end{equation}
Expression \eqref{(4,5)_seconda} can be ruled out because, otherwise,
$\{\sigma=0\}$ would be a contact structure containing a $3$-dimensional
distribution, $Y^{\bot}$, such that $(Y^{\bot})^{\prime}=\{\sigma=0\}$, which
is impossible. If (\ref{(4,5)_zera})
holds, $Y$ is a multiple of $Y_{f}$ (Proposition \ref{Ortogonale_Tipo_2}); on
the other hand, there exists a contact
transformation sending $f$ into the coordinate function $x^1$, so that, modulo a non--zero factor,
\[
Y=\partial_{u_1}.
\]
Finally, in the case \eqref{(4,5)_prima}, one has
\begin{equation}
\label{X_di_appoggio}Y=Y_{\sigma}=Y_{f}-gY_{h}.
\end{equation}
Hence,
\[
Y(\theta)=df-gdh-(f_{u}-gh_{u})\theta\,, \quad Y^{2}(\theta)=-Y_{h}(f)dg+Y_{g}(f-gh)dh.
\]
But, being $Y$ of type $2$, one gets
\begin{equation}
-Y_{h}(f)dg+Y_{g}(f-gh)dh=\lambda \theta+\mu(df-gdh) \label{bingo3}
\end{equation}
for some $\lambda,\mu\in C^{\infty}(M)$. As the contact form $\theta$ is determined
up to a non--zero factor, one may assume that $\lambda$ does not vanish. Hence, it
follows from (\ref{bingo3}) that
\[
\theta=-\frac{Y_{h}(f)}{\lambda} \left(  dg +\frac{\mu}{Y_{h}(f)}df + \frac
{Y_{g}(gh-f)-\mu g}{Y_{h}(f)}dh \right)  .
\]
So, the functions
\[
x^1=f,\quad x^2=h,\quad u=-g,\quad u_1 = \frac{\mu}{Y_{h}(f)},\quad u_2=\frac
{Y_{g}(gh-f)-\mu g}{Y_{h}(f)}\,,
\]
form a contact chart. Consequently, $Y$ from (\ref{X_di_appoggio}) takes the
form
\[
Y=Y_{x^1}+uY_{x^2}=\partial_{u_1}+u\partial_{u_2}.
\]

\medskip\noindent
Let us prove statement $3)$. If $Y$ is of type $3$, then it is characteristic for the distribution $\mathcal{D}_{Y}:=\{\theta=L_Y(\theta)=L_Y^{2}(\theta)=0\}=\Span{Y\,,\,X}$, for some $X\in Y^{\perp}$. Hence,
$\mathcal{D}_{Y}$ is integrable (because it contains $[X,Y]$) and,
consequently, it is spanned by two Hamiltonian vector fields $Y_f$ and $Y_g$ in involution, i.e., $Y=aY_f+bY_g$ for some functions $a,b\in C^\infty(M)$. In view of the involutivity, we can choose $f$ and $g$ as the coordinates $u_1$ and $u_2$ in a contact system. A straightforward computation shows that $L_Y^2(\theta)$ depends on $L_Y(\theta)$ and $Y$ (i.e., $Y$ is of type $2$) if and only if $b/a$ is a first integral of $Y$.
\end{proof}

We can split parabolic Monge--Amp\`ere equations into four mutually non--contactomorphic classes, according to the growth of the derived flag of their characteristic distribution.
More precisely, the derived flag of a Lagrangian distribution $\D$ can have only the following behavior (the subscripts below indicate the dimension of the distribution):
\begin{itemize}
\item (2,2): $\D_2=\mathcal{D}_2'$, i.e. $\D$ is completely integrable;
\item (2,3,4,4): $\D_2\subset\D'_3\subset\D''_4=\D'''_4$;
\item (2,3,4,5): $\D_2\subset\D'_3\subset\D''_4\subset\D'''_5= TM$;
\item (2,3,5): $\D_2\subset\D'_3\subset\D''_5=TM$.
\end{itemize}
The last case of the above list is called \emph{generic}, whereas the first three ones are called \emph{non--generic}.

\subsubsection{The case when the characteristic distribution $\D$ is of type $(2,3,5)$, i.e. $\D$ is generic}

A natural question is if any parabolic Monge--Amp\`ere equation is contactomorphic to a linear one. This is true if the equation is non--generic, i.e., if its characteristic distribution is not of type $(2,3,5)$ (see Sections \ref{sec.2.2}, \ref{sec.2.3.4.4}, \ref{sec.2.3.4.5} below). In the generic case, the fact that a parabolic Monge--Amp\`ere equation with $C^\infty$ coefficients is contactomorphic to a linear one remains conjectural  \cite{MR1334205,MR2503974}, whereas in the analytic case the conjecture has been proved in \cite[Section 3, Theorem 1]{MR1334205} and in \cite[Theorem 47]{MR2503974} by showing the existence of a vector field of type less than $4$ (see Theorem \ref{normal_form_generic_D} below).

\begin{theorem}\label{normal_form_generic_D}
If a Lagrangian distribution $\mathcal{D}$
contains a vector field of type less than $4$, then there exist contact local coordinates on $M$ in which $\mathcal{D}$ takes the form
\begin{equation}\label{eq_normal_form_generic_D}
\D=\langle\partial_{u_1}+a\partial_{u_2}\,,\,
D_{x^2}-aD_{x^1}+b\partial_{u_2}\rangle\,, \quad a,b\in C^\infty(M)\,.
\end{equation}
\end{theorem}
\begin{proof}
In view of Proposition \ref{Campi_Tipo2}, a vector field $X\in\mathcal{C}$ of type less than $4$ can be put, up to a non--vanishing factor, in the form $\partial_{u_1}+a\partial_{u_2}$. Then it is enough to choose a vector field that is orthogonal to $X$ to obtain a basis for $\D$.
\end{proof}
Taking into account Definition \ref{def.E.D.2.var}, we realize that the parabolic Monge--Amp\`ere equation admitting \eqref{eq_normal_form_generic_D} as its characteristic distribution is
\begin{equation}\label{eq.PMAE.generic}
u_{22}-2au_{12}+a^{2}u_{11}=b\,,\quad a,b\in C^{\infty}(M)\,.
\end{equation}

\subsubsection{The case when the characteristic distribution $\D$ is of type $(2,2)$, i.e., $\D$ is completely integrable}\label{sec.2.2}

Since $\D$ is completely integrable, in view of Proposition \ref{prop.int.one.way} there exists a system of coordinates $(x^1,x^2,u,u_1,u_2)$ such that $\D=\Span{\partial_{u_1}\,,\partial_{u_2}}$. This choice of coordinates leads to no Monge--Amp\`ere equation as $\D$, in this specific coordinate system, turns out to be completely vertical. We overcome this problem by considering partial (and total) Legendre transformations (see \eqref{eq.partial.legendre}). Indeed, if we apply one of these transformations, we realize that, in the new coordinates that for simplicity we keep denoting by $(x^1,x^2,u,u_1,u_2)$, $\D$ can assume one of the following forms:
\begin{equation}\label{eq.3.distr}
\Span{\partial_{u_1}\,,D_{x^2}}\,,\quad \Span{D_{x^1}\,,\partial_{u_2}}\,,\quad \Span{D_{x^1}\,,D_{x^2}}\,.
\end{equation}
In other words, the distribution $\D=\Span{\partial_{u_1}\,,\partial_{u_2}}$ is contactomorphic to one of the distributions \eqref{eq.3.distr}.  In view of Definition \ref{def.E.D.2.var}, the differential equation associated to the first (resp., second and third) distribution from \eqref{eq.3.distr} is $u_{22}=0$ (resp., $u_{11}=0$ and $u_{11}u_{22}-u_{12}^2=0$). Since $\D$ is Lagrangian the corresponding Monge--Amp\`ere equation is parabolic. Recall that two parabolic Monge--Amp\`ere equations are contactomorphic if and only if so are their characteristic distributions (see the beginning of Section \ref{secNormFormMAEs}).
\begin{corollary}
The parabolic Monge--Amp\`ere equations $u_{11}u_{22}-u_{12}^2=0$ and $u_{11}=0$ are contactomorphic.
\end{corollary}

\subsubsection{The case when the characteristic distribution $\D$ is of type $(2,3,4,4)$}\label{sec.2.3.4.4}

Now, in view of Proposition \ref{Ortogonale_Tipo_2}, $\D'^\perp$ is $1$-dimensional and Hamiltonian. Let $X\in\mathcal{C}$ such that $\D'^\perp=\langle X\rangle$. Then
$$
\D=\langle X\,,\,Y\rangle \,, \quad \D'=\langle X\,,\,Y\,,\,[X,Y]\rangle\,, \quad Y\in X^\perp\,.
$$
Indeed,
$$
d\theta(X,[X,Y])=X(\theta)([X,Y])=X(X(\theta)(Y))-X^{2}(\theta)(Y)=0\,,
$$
in view of the fact $Y\in X^\perp$ implies $X(\theta)(Y)=0$ and that, if $X$ is a Hamiltonian vector field (in particular of type $2$), then $X^{2}(\theta)(Y)=0$.

Now, there exists a contact system of coordinates $(x^1,x^2,u,u_1,u_2)$ where the Hamiltonian vector field $X$ takes the form $\partial_{u_1}$, see Proposition \ref{Campi_Tipo2}. Thus, $\D$ is spanned by $\partial_{u_1}$ and a vector field orthogonal to it. Since
$
\langle \partial_{u_1}\rangle^\perp = \langle D_{x^2}\,,\, \partial_{u_1}\,,\, \partial_{u_2} \rangle\,,
$
without loss of generality,
\begin{equation}\label{eq.D.2.3.4.4}
\D=\langle \partial_{u_1}\,,\,  D_{x^2} + b\partial_{u_2} \rangle\,, \quad b\in C^\infty(M)\,,\,\, b_{u_1}\neq 0\,.
\end{equation}
The last condition in \eqref{eq.D.2.3.4.4} assures that $\D$ is not integrable.
The parabolic Monge--Amp\`ere equation having \eqref{eq.D.2.3.4.4} as its characteristic distribution is
\eqref{eq.PMAE.generic} with $a=0$.

\subsubsection{The case when the characteristic distribution $\D$ is of type $(2,3,4,5)$}\label{sec.2.3.4.5}

We can follow the same reasonings as in Section \ref{sec.2.3.4.4}, taking into account that now $\D'=\langle X \rangle$, where $X\in\mathcal{C}$ is of type $2$ but not Hamiltonian. So (see again Proposition \ref{Campi_Tipo2}), $X$ can be put in the form $\partial_{u_1}+u\partial_{u_2}$. By following the lines of Section \ref{sec.2.3.4.4}, we arrive to the normal form
\begin{equation}\label{eq.D.2.3.4.5}
\D=\langle\partial_{u_1}+u\partial_{u_2},\,D_{x^2}
-uD_{x^1}+b\partial_{u_2}\rangle\,, \,\, b\in C^{\infty}(M)\,, \,\,
b_{u_1}+ub_{u_2}+uu_1-u_2\neq0\,.
\end{equation}
The last condition in \eqref{eq.D.2.3.4.5} assures that the distribution $\D$ is of type $(2,3,4,5)$. Note that, if $b_{u_1}+ub_{u_2}+uu_1-u_2=0$, then the distribution $\D$ would be integrable.
The parabolic Monge--Amp\`ere equation having \eqref{eq.D.2.3.4.5} as its characteristic distribution is
\eqref{eq.PMAE.generic} with $a=u$.

\section{Multidimensional Monge--Amp\`ere equations of Goursat type and construction of a complete integral}\label{sec.multidim.MAE}

The general form of a Monge--Amp\`ere equation in $n$ independent variables is
\begin{equation}\label{eq.mult.MAE.minors}
M_n+M_{n-1}+\dots + M_0=0\,,
\end{equation}
where $M_k$ is a linear combination (with functions of $x^i,u,u_i$ as coefficients) of all $k\times k$ minors of the $n\times n$ Hessian matrix of $u=u(x^1,\dots,x^n)$. Observe that for $n=2$ we recover the definition \eqref{eqDefMAE}.
According to our general definition of a $2\Nd$ order PDE (see Definition \ref{def2ndOrdPDE}), equation \eqref{eq.mult.MAE.minors} must be interpreted as a hypersurface in $M^{(1)}$, where $M$ is a $(2n+1)$--dimensional contact manifold. It turns out that the simplest way to make \eqref{eq.mult.MAE.minors} intrinsic is by means of an $n$--form $\Omega$ on $M$. More precisely, given an $n$--form $\Omega$ on M, the subset
\begin{equation}\label{eq.MAE.intr.Omega}
\mathcal{E}_{\Omega}:=\{L\in M^{(1)}
\,\,\big|\,\,\Omega|_{L}=0\}\,,\quad\Omega\in\Lambda^n(M)
\end{equation}
of $M^{(1)}$ is a hypersurface whose local coordinate expression is given by \eqref{eq.mult.MAE.minors}.

The correspondence between Monge--Amp\`ere equations and $n$--forms on $M$ is not one--to--one. Indeed, two $n$--forms $\Omega, \Omega'$ define the same equation
$\mathcal{E}_{\Omega}= \mathcal{E}_{\Omega'}$ if and only if, up to a non--vanishing factor,
\begin{equation*}
\Omega'=\Omega+\alpha\wedge d\theta+\beta\wedge \theta
 \text{\quad for some $\alpha\in\Lambda^{n-2}(M),\,\,\beta\in\Lambda^{n-1}(M)$}.
\end{equation*}
Thus, the aforementioned correspondence is one--to--one (up to a non--vanishing factor) if we factor our the differential forms contained in the differential ideal generated by an arbitrarily chosen contact form. The so--obtained forms are also called \emph{effective} \cite{MR2352610}. It turns out that there is a special class of Monge--Amp\`ere equations, namely, those associated to decomposable $n$--forms (up to the contact differential ideal). These Monge--Amp\`ere equations represent precisely the multidimensional analogues of Definition \ref{def.E.D.2.var}.
\begin{definition}\label{def.E.D.n.var}
Let $\D\subset\CC$ be a $n$--dimensional sub--distribution of $\CC$. Then the sub--bundle
\begin{equation*}
\E_\D:=\{  L\in M^{(1)}_p\mid L\cap\D_p\neq 0\}\subset M^{(1)}
\end{equation*}
is called the \emph{Goursat--type Monge--Amp\`ere equation} (in $n$ independent variables) associated to $\D$.
\end{definition}
Locally, in a system of contact coordinates $(x^i,u,u_i,u_{ij})$, if the distribution $\D$ is given by
$$
\D=\Span{\partial_{x^i}+u_i\partial_u+f_{ij}\partial_{u_j}}_{i=1,\dots, n}\,,\quad f_{ij}\in C^\infty(M^{(1)})\,,
$$
then the corresponding a Goursat--type Monge--Amp\`ere equation $\E_\D$ is given by
$$
\det\big(u_{ij}-f_{ij}\big)=0\,.
$$
The PDE $\E_D$ will be called a \emph{parabolic Goursat--type Monge--Amp\`ere equation} if $\D$ is a Lagrangian distribution. This, locally, implies that $f_{ij}=f_{ji}$.

Now we link the intrinsic Definition \ref{def.E.D.n.var} with the traditional coordinate--wise definition \eqref{eq.mult.MAE.minors}.
\begin{proposition}\label{cor.decomposable}
The equation $\E_\D$ (see Definition \ref{def.E.D.n.var}) is the Monge--Amp\`ere equation associated, according to \eqref{eq.MAE.intr.Omega}, with  the $n$--form
\[
\Omega=L_{Y_{1}}(\theta)\wedge\dots\wedge
L_{Y_{n}}(\theta),
\]
where $Y_{i}$ are vector fields generating the orthogonal
distribution $\D^{\perp}$. The converse is also true.
\end{proposition}
\begin{proof}
Note that
$\D\subset \mathcal{C}$ is defined  by the equations
$\{\theta=0\,,\,\,L_{Y_i}(\theta)=0\}$
where the vector fields $Y_{i}$ generate $\mathcal{D}^{\perp}$. The remaining part of the claim follows by simple considerations of linear algebra.
\end{proof}

Recall Definition \ref{def.char.tutto}.

\begin{proposition}
Let $\E$ be a $2\Nd$ order PDE on a $(2n+1)$--dimensional contact manifold. Then $\E$ admits an integrable $n$--dimensional distribution as a (field of) characteristic varieties if and only if it is contactomorphic to $u_{x^1x^1}=0$.
\end{proposition}
\begin{proof}
This follows from the fact that integrable distributions are all contactomorphic each other (see Proposition \ref{prop.int.one.way}). In particular, the distribution \eqref{eq.int.dist.vert} is contactomorphic to
\begin{equation}\label{eq.int.dist.1.hor}
\mathcal{D}=\langle D_{x^1}\,,\partial_{u_2} \dots\,,\partial_{u_n}\rangle\,,
\end{equation}
and the equation $\E_\D$, with $\D$ as in \eqref{eq.int.dist.1.hor}, is precisely $u_{11}=0$.
\end{proof}
\begin{corollary}
Let $U=(u_{x^ix^j})$ be the Hessian matrix of $u=u(x^1,\dots,x^n)$. Then all the equations $\det(U_k)=0$, $1\leq k\leq n$, where $U_k$ is the $k\times k$ matrix formed by the first $k$ rows and the first $k$ columns of $U$, are contactomorphic each other.
\end{corollary}
\begin{proof}
It is enough to observe that the distributions
$$
\D_1=\langle D_{x^1}\,,\partial_{u_2}\,, \dots\,,\partial_{u_n}\rangle\,,\quad \D_2=\langle D_{x^1}\,,D_{x^2}\,,\partial_{u_3} \dots\,,\partial_{u_n}\rangle\,, \,\,\cdots\,,\,\, \D_n=\langle D_{x^1}\,,D_{x^2}\,, \dots\,,D_{x^n}\rangle
$$
are contactomorphic each other being integrable. Note that the local expression of $\D_k$ can be obtained by considering partial Legendre transformations \eqref{eq.partial.legendre}. Equations $\E_{\D_k}$ are exactly $\det(U_k)=0$.
\end{proof}

\begin{definition}
Let $\E\subset M^{(1)}$ be a $2\Nd$ order PDE. A function $f\in C^{\infty}(M)$ is called an \emph{intermediate integral} of $\E$ if all the solutions of the family of $1\St$ order PDEs $f=c$, $c\in \mathbb{R}$, are also solutions of $\E$.
\end{definition}

\begin{theorem}[Theorem 6.7 of \cite{MR2985508}]\label{th.int.int.goursat}
A function $f\in C^{\infty}(M)$ is an intermediate integral of $\E_\D$  if and only if it is a first integral either of $\D$ or $\D^\perp$.
\end{theorem}
We have seen in Theorem \ref{th.meth.char} how to find the solution of a $1\St$ order PDE starting from a non--characteristic Cauchy datum. Basically, one can ``enlarge'' the given Cauchy datum by means of the (local) flow of the characteristic vector field $Y_f$ (Hamiltonian vector field associated to $df$, recall Definition \ref{def.involution}). A similar method, called the \emph{generalized Monge method}, can apply also to Monge--Amp\`ere equations of Goursat type. This method employs vector fields of type $2$ rather than Hamiltonian ones.
\begin{theorem}[Theorem 6.12 of \cite{MR2985508}]\label{th.Monge.method}
Let $\E_\D\subset M^{(1)}$ be a Goursat--type Monge--Amp\`ere equation, where $(M,\CC)$ is a  $(2n+1)$--dimensional contact manifold. Let $\theta$ be a contact form. Let $X\in\D$ be a vector field of type $2$. Let $\Sigma\subset M$ be a Cauchy datum for $\E_\D$ (see Definition \ref{def.Cauchy.datum.2.order})
transversal to $X$. Then
$N=\bigcup_{t}\,\varphi_{t}(\Sigma)\subset M$, where $\{\varphi_{t}\}$
is the local flow of $X$, is a solution of the equation
$\mathcal{E}_{\mathcal{D}}$ if and only if
$d\theta(T_{p}\Sigma,X_{p})=0\,\,\forall\,\,p\in \Sigma$.
\end{theorem}
It is not difficult to realize that any first integral $f$ of the distribution $\D^\perp$ defines a Hamiltonian vector
field $Y_f$ belonging to $\D$. So, in view of Theorem \ref{th.Monge.method},
the problem of constructing solutions of $\E_\D$ reduces to the problem of constructing $(n-1)$--dimensional submanifolds
$\Sigma$ of $M$ such that
\begin{equation}\label{eq.Y.N}
d\theta(T_p\Sigma\,,\,({Y_f})_p) = 0\,,\quad\forall p\in \Sigma.
\end{equation}
\begin{proposition}\label{prop.Monge.method.2}
Let $f$ be an intermediate integral of $\E_\D$
and $\Sigma$ a Cauchy datum for the $1\St$ order PDE $\{f=0\}$. Then the submanifold
$N=\bigcup_t\varphi_t(\Sigma)$,
where $\{\varphi_t\}$ is the local flow of the Hamiltonian vector field
$Y_f$, is a solution of $\E_\D$. If $\Sigma$ is
non--characteristic, then such a solution is unique.
\end{proposition}
\begin{proof}
Let $X\in T\Sigma$. Then
$
d\theta(Y_f,X)=df(X)=X(f)=0.
$
Therefore, $N$ is a solution of $\E_\D$,
since $\Sigma$ satisfies the condition \eqref{eq.Y.N}. The uniqueness of $N$ follows from the fact that $N$ is also a solution of the $1\St$ order
PDE $f=0$.
\end{proof}
Now we have all the necessary tools to provide a sufficient condition for a Monge--Amp\`ere equation of Goursat type $\E_\D$ to have any Cauchy datum ``enlargeable'' to a solution.
\begin{theorem}\label{th.Monge.Morimoto}
Let us suppose that $\mathcal{D}$ (or $\mathcal{D}^\perp$) possesses $n$ independent first integrals.
Then any Cauchy datum $\Sigma$ can be extended to a solution of $\mathcal{E}_{\mathcal D}$.
\end{theorem}
\begin{proof}
Let $f^1,\dots,f^n$ be independent first integrals of $\mathcal{D}$ (so that any function of them is an
intermediate integral of $\mathcal{E}_{\mathcal D}$). Let denote by $g^i$ the restriction of $f^i$ to $\Sigma$.
 Of course, the functions $g^i$ are dependent. So, there exists a non--trivial functional relation
$
\psi(g^1,\dots,g^n)=0.
$
The function $f=\psi(f^1,\dots,f^n)$ turns out to be an intermediate integral which vanishes on $\Sigma$ and it also satisfies the hypothesis of Proposition \ref{prop.Monge.method.2}. Then the flow of $Y_f$ extends $\Sigma$ to a solution of $\mathcal{E}_{\mathcal D}$.
\end{proof}
Above Theorems actually allow to solve a Cauchy problem for a given $2\Nd$ order PDE.

\subsection{An illustrative example}
Let us consider the $2\Nd$ order PDE
\begin{equation}\label{eq.example.bello}
\mathcal{E}:\,\,u_{11}+(x^1+u_2)u_{12}+x^1u_2u_{22}-u_3u_{13}-x^1u_3u_{23}=0
\end{equation}
in $3$ independent variables $x^1,x^2,x^3$ and a dependent variable $u$.
A point $L\in\mathcal{E}$ ha coordinates
$$
\big(x^1,x^2,x^3,u,u_1,u_2,u_3,-(x^1+u_2)u_{12}-x^1u_2u_{22} +u_3u_{13}+x^1u_3u_{23},u_{12},\dots,u_{33}\big)\,.
$$

\subsubsection{Computation of characteristics} A straightforward computation shows that the symbol (see \eqref{eq.symb.Gianni.2}) of the equation \eqref{eq.example.bello} is decomposable as follows
\begin{equation}\label{eq.ex.mult}
(x^1w_2+w_1)\odot (w_1+u_2w_2-u_3w_3)\,,
\end{equation}
where the $w_i$'s are the total derivatives restricted to  \eqref{eq.example.bello}, i.e.,
$$
w_1=D_{x^1}
-\big((x^1+u_2)u_{12}+x^1u_2u_{22}-u_3u_{13}-x^1u_3u_{23}\big)\,\partial_{u_1}
+u_{12}\partial_{u_2}+u_{13}\partial_{u_3}
$$
$$
w_2=D_{x^2} + u_{12}\partial_{u_1}
+u_{22}\partial_{u_2} + u_{23}\partial_{u_3}$$
$$
w_3=D_{x^3} + u_{13}\partial_{u_1} +
u_{23}\partial_{u_2} + u_{33}\partial_{u_3}.
$$

\subsubsection{Proof that $\mathcal{E}$ is of Goursat
type} By
substituting the above values of $w_1$, $w_2$ and $w_3$ in
the first factor of \eqref{eq.ex.mult} we obtain
$$
D_{x^1}+x^1D_{x^2} + (u_{13}+x^1u_{23})(\partial_{u_3}+u_3\partial_{u_1}) + (u_{12}+x^1u_{22})(\partial_{u_2}-u_2\partial_{u_1})\,,
$$
where we recall that the $D_{x^i}$'s are defined by \eqref{eq.contact.local}.
If we let  vary the point $m^1$ along the fibre $\mathcal{E}_m=\pi^{-1}(m)\cap\mathcal{E}$ of the equation $\mathcal{E}$, we
obtain the distribution
\begin{equation*}
\mathcal{D}=\langle X_1, X_2, X_3\rangle ,\,\,
X_1=D_{x^1}+x^1D_{x^2},\,\,
X_2=\partial_{u_3}+u_3\partial_{u_1},\,\,
X_3=\partial_{u_2}-u_2\partial_{u_1}.
\end{equation*}
Thus,
$\mathcal{E}=\mathcal{E_D}$ (recall Definition \ref{def.E.D.n.var}). On the other hand, by taking the
orthogonal complement of $\mathcal{D}$ or by performing the
analogous calculation for the other factor of \eqref{eq.ex.mult}, $w_1+u_2w_2-u_3w_3$, we get the distribution
$$
\mathcal{D}^\perp=\langle Y_1, Y_2, Y_3\rangle\,,\,\,\, Y_1=D_{x^1}+u_2D_{x^2}-u_3D_{x^3}\,,\,\,\,
Y_2=\partial_{u_2}-x^1\partial_{u_1}\,,\,\,\, Y_3=\partial_{u_3}.
$$

\subsubsection{Intermediate integrals} It is easy to check that
$\mathcal{D}'$ is of rank $4$ and integrable, and it is generated by $\partial_{x^1}+x^1\partial_{x^2}$, $\partial_{u_3}+u_3\partial_{u_1}$, $\partial_{u_2}-u_2\partial_{u_1}$ and $\partial_u$. In view of Theorem \ref{th.int.int.goursat}, a
standard computation gives the following $3$  independent intermediate
integrals of $\mathcal{E}_{\mathcal D}$ (they are first integrals of $\D$), namely
$$
\lambda_1:={(x^1)}^2-2x^2\,,\,\,\,
\lambda_2:=(u_3)^2-(u_2)^2-2u_1\,,\,\,\,  \lambda_3:=x^3.
$$

\subsubsection{A Cauchy problem} Any Cauchy datum can be extended to a solution as we found $3$ first integrals of $\D$ (see Theorem \ref{th.Monge.Morimoto}). In this case, a Cauchy datum
$\Sigma$, according to the Definition \ref{def.Cauchy.datum.2.order}, consists of a $2$-dimensional integral submanifold of $\mathcal C$. If we suppose that this datum can be parameterized by $x^1$ and
$x^2$, then we can arbitrarily fix $u$, $x^3$ and
$u_3$ as functions of $x^1$ and $x^2$ and then determine $u_1$
and $u_2$ by the contact condition. Let us choose, for instance,
$$
\Sigma:\,\,\,u={(x^1)}^2+x^2, \,\, x^3=x^1, \,\, u_3=0.
$$
Then, from the contact condition, one easily obtains that $u_1=2x^1$ and $u_2=1$, which completes
the parametrization of $\Sigma$.

The restrictions of $\lambda_1$, $\lambda_2$ and $\lambda_3$ to
$\Sigma$ are $ \overline\lambda_1:={(x^1)}^2-2x^2$,
$\overline\lambda_2:=-1-4x^1$, $\overline\lambda_3=x^1$. A first
integral vanishing on $\Sigma$ is $f:=\lambda_2+4\lambda_3+1$, whose
associated Hamiltonian field is
$$
Y_f=2u_3Y_{u_3}-2u_2Y_{u_2}-2Y_{u_1}+4Y_{x^3}=
2\left(u_3\widehat\partial_{x^3}-u_2\widehat\partial_{x^2}-\widehat\partial_{x^1}-
2\partial_{u_3}\right)\,.
$$
$Y_f$ has $6$ independent first integrals, namely
\begin{equation*}
\mu_1:=u_1,\,\,\, \mu_2:=u_2,\,\,\, \mu_3:=\frac 12 (u_3)^2+2x^3,\,\,\, \mu_4:=u_2x^1-x^2,
\mu_5:=x^1-\frac 12u_3,\,\,\,
\mu_6:=\frac 12((u_2)^2+u_1)u_3-\frac 16(u_3)^3-u.
\end{equation*}
\noindent In order to prolong the Cauchy datum $\Sigma$ along the orbits
of $Y_f$, we restrict the above $6$ first integrals on $\Sigma$ (the
bar denotes such a restriction): $ \overline\mu_1=2x^1$, $
\overline\mu_2=1$, $\overline\mu_3=2x^1$, $
\overline\mu_4=x^1-x^2$, $\overline\mu_5=x^1$, $
\overline\mu_6=-({(x^1)}^2+x^2)$. By eliminating the parameters $x^1$
and $x^2$, we obtain  $4$ independent relations, viz.
\begin{equation}\label{e:relaciones}\mu_2=1,\quad \mu_3-\mu_1=0,\quad
\mu_5-\frac 12\mu_1=0,\quad \mu_6+\frac 14(\mu_1)^2+\frac
12\mu_1-\mu_4=0\,.
\end{equation}
If we substitute the $\mu$'s in
\eqref{e:relaciones},
we get
\begin{equation}\label{e:otrasrelaciones}
\begin{cases}
&u_2=1\,,\\
&\frac 12 {(u_3)}^2+2x^3-u_1=0\,,\\
&x^1-\frac 12 u_3-\frac 12 u_1=0\,,\\
&\frac 12 ({(u_2)}^2+u_1)u_3-\frac 16{(u_3)}^3-u+\frac 14 {(u_1)}^2+\frac 12
u_1-u_2x^1+x^2=0\,.
\end{cases}
\end{equation}
Finally, from the first three equations of the system \eqref{e:otrasrelaciones} we can obtain the $u_i$'s in terms of the $x^i$'s and then the fourth equation allows to express $u$ as a function of $x^1,x^2,x^3$ which is the required solution:
$$
u=\frac{1}{6}+x^1+{(x^1)}^2+x^2-x^3\mp\sqrt{1-4x^3+4x^1}\left(\frac{1}{6} + \frac{2}{3}x^1-\frac{2}{3}x^3\right).
$$

\section{Monge--Amp\`ere equations on K\"ahler and para--K\"ahler manifolds and Special Lagrangian submanifolds}\label{secComplessa}

So far we interpreted scalar $2\Nd$ order PDEs, essentially, as hypersurfaces of the Lagrangian Grassmannian bundle $M^{(1)}\to M$, where $M$ is a contact manifold. One can also introduce PDEs on symplectic manifolds in complete analogy with the constructions we considered so far. More precisely, $1\St$ order scalar PDEs which do not depend on the unknown function are hypersurfaces of a $2n$--dimensional symplectic manifold $(W,\omega)$, where $n$ is the number of the independent variables. For this reason, following \cite{MR2352610}, we call such equations \emph{symplectic PDEs} of $1\St$ order. We recall that, by Darboux's theorem \ref{th.Darboux.sympl},
there always exists a system of coordinates $(x^i,u_i)$ on $W$ such that $\omega=dx^i \wedge du_i$, so that a $1\St$ order scalar symplectic PDE is locally described by $\{f(x^i,u_i)=0\}$. Solution of symplectic PDEs of the $1\St$ order are Lagrangian submanifolds of $W$ contained in the aforementioned hypersurface. We can introduce the prolongation $W^{(1)}$ of a symplectic manifold $W$ as the set of Lagrangian planes of $W$. More precisely, we have the bundle $W^{(1)}\to W$ whose fiber at $w\in W$ is the Lagrangian Grassmannian $\LL(T_wW)$ of $T_wW$.
Smooth coordinates on $W^{(1)}$ can be constructed similarly as we did in Remark \ref{rem.coordinates.M1}.

A $2\Nd$ order scalar PDE $\E$ with $n$ independent variables, which is independent of the unknown function, is a hypersurface of $W^{(1)}$, where $W$ is a $2n$-dimensional symplectic manifold. Thus, locally, in a system of Darboux coordinates, $\mathcal{E}=\{F(x^i,u_i,u_{ij})=0\}$. Such PDEs are called \emph{symplectic PDEs} of $2\Nd$ order. Solutions of such PDEs are Lagrangian submanifolds of $W$ whose prolongation is contained in the aforementioned hypersurface.

Similarly to the contact case (see \eqref{eq.MAE.intr.Omega}), an $n$-form $\Omega$ on a symplectic manifold $W$ defines  the set $\{L\in W^{(1)}\,\,|\,\, \Omega|_{L}=0\}$. Such a set is a hypersurface of $W^{(1)}$ which turns out to be a (symplectic) Monge--Amp\`ere equation. This will be the key observation for introducing and studying Monge--Amp\`ere equations on K\"ahler and para--K\"ahler manifolds.
In fact, an important natural class of  $n$--forms on symplectic manifolds  arises when we  assume that a symplectic manifold $(W, \omega)$ is a K\"ahler manifold (for more details and discussions see also \cite{MR2945626}). We recall that this means that $(W, \omega)$ is equipped with a complex structure $J$, i.e., a field of endomorphisms $J:TW\to TW$ such that $J^2=-Id$ (almost--complex structure), whose Nijenhuis tensor
\begin{equation}\label{eq.Nijenhuis}
N_J(X,Y)=[X,Y]+[JX,JY]-J[JX,Y]-J[X,JY]=0
\end{equation}
vanishes, such that  $g = \omega \circ J$ is a Riemannian metric.

Similar constructions can be performed also by starting from a field of endomorphisms $K:TW\to TW$, such that $K^2=Id$, called an \emph{almost--para--complex structure}. A \emph{para--complex structure} is an almost--para--complex structure whose Nijenhuis tensor $N_K$ (see \eqref{eq.Nijenhuis}) vanishes.
A  para--K\"ahler manifold is a symplectic manifold $(W, \omega)$ equipped with a para--complex structure $K$, such that $g = \omega \circ K$ is a pseudo--Riemannian metric. This metric has  neutral signature $(n,n)$  and  the associated Levi-Civita connection  preserves $\omega$  and $K$. In particular, the  eigendistributions $T^{\pm}W$ of $K$ are parallel (and  $\omega$--isotropic, i.e., Lagrangian). Since the distributions $T^{\pm}W$ turn out to be integrable as the Nijenhuis tensor $N_K$ of $K$ vanishes, locally, a para-K\"ahler manifold is the direct product of two Lagrangian submanifolds.

Below we discuss the Monge--Amp\`ere equations associated to the
imaginary (or real) part  of a (para-)holomorphic  $n$-form $\Omega$ on a $2n$-dimensional  (para-)K\"ahler  manifold. Solutions  are special Lagrangian submanifolds. Following Harvey and Lawson \cite{MR3220439}, we  point out an application of  such Monge--Amp\`ere equations to the Monge--Kantorovich mass transport problem.

\begin{problem}[Monge--Kantorovich Optimal Mass Transport problem in the smooth setting, with classical cost function]\label{prob.Monge_Kant}
Let $\xi(y)$ and $\eta(v)$ be smooth positive functions, respectively, on open domains $U$ and $V$ of $\mathbb{R}^n$ with compact closure. Here $y=(y^1,\dots,y^n)$ and $v=(v_1,\dots,v_n)$ are coordinates on $U$ and $V$, respectively.  Let $\xi(y)d^ny$ and $\eta(v)d^nv$ be the associated (positive) densities.  The Monge--Kantorovich problem is to find a map $F:U\to V$ which preserves such densities and that minimizes the functional $\int_Uc(y,F(y))\xi(y)d^ny$, where $c=\frac12|y-v|^2$ is the standard ``cost'' function.
\end{problem}
Denote by $\Lambda^{n,0}$ the line bundle of holomorphic $n$--forms on $W$. A local (holomorphic) section of $\Lambda^{n,0}$, in holomorphic coordinates $(z^1,\dots,z^n)$, has the form
$$
\Omega=\Omega_1 + i\Omega_2 = f(z)d^nz:=f(z^1,\dots,z^n)dz^1\wedge\dots\wedge dz^n\,.
$$

\noindent
The real $n$--forms $\Omega_1, \Omega_2$  defines Monge--Amp\`ere equations  $\mathcal E_{\Omega_1},\mathcal E_{\Omega_2}$, according to \eqref{eq.MAE.intr.Omega}. A solution of $\mathcal E_{\Omega_i}$, that is, a  Lagrangian submanifold of $W$ which annihilates $\Omega_i$, is called  {\it  a special
$\Omega_i$-Lagrangian manifold} (SLAG) \cite{harvey1982,Joyce01lectureson}.

From now on we shall consider the case when
\begin{equation}\label{eq.sympl.compl}
W= \mathbb{C}^n\,,\quad \omega =\frac{i}{2} \sum_k dz^k \wedge d\bar z^k = \sum_j dx^j \wedge du_j \,,\quad z^j = x^j + i u_j\,,
\end{equation}
where $(x^j,u_j)$ are Darboux coordinates on $\mathbb{R}^{2n}\simeq\mathbb{C}^n$ and the $z^k$'s are holomorphic coordinates on $\mathbb{C}^n$. Also, let
\begin{equation*}
\Omega  = d^nz = dz^1 \wedge \cdots \wedge dz^n = \Omega_1 + i \Omega_2
\end{equation*}
be the  standard holomorphic $n$--form.
\begin{proposition}\label{prop.compl.equiv}
All  equations $\mathcal E_{\Omega_{\phi}}$, where
  $ \Omega_{\phi}= Re (e^{i \phi}\Omega) = \cos (\phi)  \Omega_1 + \sin (\phi) \Omega_2$, are  equivalent. In particular, $\E_{\Omega_1}$ is equivalent to $\E_{\Omega_2}$.
\end{proposition}
\begin{proof}
Let us consider the transformation $z \to  e^{i \phi} z$. It is a symplectic transformation. In fact, we have also $\bar{z} \to  \overline{e^{i \phi}} \bar{z}= e^{-i \phi} \bar{z}$, so that the symplectic form $\omega$ of \eqref{eq.sympl.compl} transforms into itself. Such a transformation maps $\Omega_0$ into $\Omega_{-n \phi}$.
\end{proof}

\begin{example}
The Laplace equation $u_{11}+u_{22}=0$ and the standard elliptic Monge--Amp\`ere equation $u_{11}u_{22}-u_{12}^2=1$ are equivalent. Indeed, in the case $n=2$, we have that
$$
\Omega=\Omega_1+i\Omega_2=dx^1\wedge dx^2 - du_1\wedge du_2 + i(du_1\wedge dx^2 + dx^1\wedge du_2)\,,
$$
and by restricting $\Omega$ to the Lagrangian planes (see \eqref{eq.MAE.intr.Omega}) we get the $2$--form:
$$
\big((1-u_{11}u_{22}+u_{12}^2) + i (u_{11}+u_{22}) \big) dx^1\wedge dx^2.
$$
Thus, our assertion follows in view of Proposition \ref{prop.compl.equiv}.
\end{example}

Now let us pass to para--complex case, which allows to recast the Monge--Kantorovich problem in a pure geometric setting. To begin with, we underline that a possible approach to develop para--complex geometry by analogy with complex geometry, is to introduce \emph{para-complex} numbers as the elements of the $2$-dimensional real algebra generated by $1$ and $\tau$ with $\tau^2=0$  \cite{MR3220439}.
We will denote such algebra by $\mathbb{D}$. Any number $z\in\mathbb{D}$ can be written as $z=x+\tau p$, with $x,p\in\mathbb{R}$. Conjugation is defined as in the complex case: $\bar{z}=x-\tau p$. The essential difference with complex numbers is the existence of distinguished coordinates $(y,v)$, called \emph{null coordinates}, defined as follows:
$$
z=x+\tau p= ye+v\bar{e}\,,\quad e=\frac{1}{2}(1-\tau)\,,\,\,\, \bar{e}=\frac{1}{2}(1+\tau)\,.
$$
The vector space $\mathbb{D}^n$, together with its null coordinates, is defined in an obvious way.
The multiplication of vectors of $\mathbb{D}^n$ by $\tau$ turns out to be a para--complex structure on $\mathbb{D}^n$.
Analogously to the complex case,
a function $F:\mathbb{D}^n\to\mathbb{D}$ is called \emph{para--holomorphic} if $F=F(z^j)$. If we write also $F$ in null coordinates, i.e.,
$F=ef+\bar{e}g$, with $f$ and $g$ real-valued functions, then it is para--holomorphic if and only if $f=f(y^j)$ and $g=g(v_j)$. Also, we can
introduce the bundle $\Lambda^{n,0}$ of \emph{para--holomorphic} $n$--forms. By considering the vector notation $z=(z^1,\dots,z^n)$, $y=(y^1,\dots,y^n)$, $v=(v_1,\dots,v_n)$, a para-holomorphic $n$--form $\Omega$ can be written as
\begin{equation}\label{eq.omega.para}
\Omega=\Omega_1+\tau\Omega_2=\alpha(z)dz^1\wedge\cdots\wedge dz^n=e\xi(y)d^ny  + \bar{e}\eta(v)d^nv\,,
\end{equation}
where
\begin{equation*}
d^ny:=dy^1\wedge\cdots\wedge dy^n\,,\quad d^nv:=dv_1\wedge\cdots\wedge dv_n\,,\quad  \Omega_1=\xi(y)d^ny + \eta(v)d^nv\,, \quad \Omega_2=\xi(y)d^ny - \eta(v)d^nv\,.
\end{equation*}
The real $n$-forms $\Omega_1, \Omega_2$  of \eqref{eq.omega.para} defines Monge--Amp\`ere equations $\mathcal E_{\Omega_1},\mathcal E_{\Omega_2}$  via formula \eqref{eq.MAE.intr.Omega}. A solution of $\mathcal E_{\Omega_i}$, following \cite{MR3220439}, is called  {\it  a special
$\Omega_i$-Lagrangian manifold.}

A remarkable difference from the K\"ahler case is that we can associate some distinguished Monge--Amp\`ere equations with a para--holomorphic $n$--form $\Omega\in\Lambda^{n,0}$. Essentially, this follows from the fact that in the para--complex case we have two distinguished ``directions'', i.e., the null directions.
\begin{proposition}
Let $\Omega$ be as in \eqref{eq.omega.para}. Then the equations $\mathcal{E}_{\Omega_1\mp\Omega_2}$ do not depend on the conformal class of $\Omega$.
\end{proposition}
\begin{proof}
If we multiply $\Omega$ by a para-holomorphic function $f+\tau g$, we have that
$$
(f+\tau g)\Omega = f\Omega_1+g\Omega_2+\tau(g\Omega_1+f\Omega_2) = \widetilde{\Omega}_1+\tau\widetilde{\Omega}_2,
$$
so that $\widetilde{\Omega}_1+\widetilde{\Omega}_2=(f+g)(\Omega_1+\Omega_2)$ and $\widetilde{\Omega}_1-\widetilde{\Omega}_2=(f-g)(\Omega_1-\Omega_2)$. Therefore, the sum and the difference of real and imaginary part is well defined up to a multiplication by a para-holomorphic function, so that the associated Monge--Amp\`ere equations are the same.
\end{proof}
As an example, consider  the   simplest instance of a para-K\"ahler manifold, namely the para-Hermitian vector space $\mathbb{D}^n = \mathbb{R}^n + \tau \mathbb{R}^n$ with  the standard para-complex  structure (defined by the multiplication by $\tau$) and the  standard  symplectic  form
\begin{equation*}
\omega:=\frac{\tau}{2}\sum_{j=1}^n dz^j\wedge d\bar{z}^j = \sum_{j=1}^n dx^j\wedge du_j = \sum_{j=1}^n dy^j\wedge dv_j\,.
\end{equation*}
\begin{proposition}\label{prop.para.compl.equiv}
Let us consider the $n$--form
$$
\Omega= d^nz := dz^1 \wedge \cdots \wedge dz^n = \Omega_1 + \tau \Omega_2\,.
$$
Then all the equations $\mathcal E_{\Omega_{\phi}}$, where
  $\Omega_{\phi}=\cosh (\phi)  \Omega_1 + \sinh (\phi) \Omega_2$, are  equivalent.
\end{proposition}
\begin{proof}
Similar to the proof of Proposition \ref{prop.compl.equiv}.
\end{proof}
Let us now go back to general para-K\"ahler manifolds. We recall that a para-K\"ahler manifold is locally the product of two Lagrangian submanifold $U$ and $V$. We can choose on $U\times V$ a system of null coordinates $(y^i,v_i)$ such that we can locally describe a Lagrangian submanifold
$L_f=\{y^i,v_i=f_{y^i}\}$ as the graph of the gradient $\nabla f$ of a function $f$ on $U$.
\begin{proposition}\label{prop.ultimo.sforzo}
A Lagrangian submanifold $L_{f(y)}=\{y^i,v_i=f_{y^i}\}\subset U\times V$ of a para-K\"ahler manifold is $\Omega_2$--Lagrangian if and only if the function  $F(p):=(p,\nabla f(p))$ satisfies the Monge--Kantorovich problem (see Problem \ref{prob.Monge_Kant}).
\end{proposition}
\begin{proof}
Let $\Omega$ be as in \eqref{eq.omega.para}. The $n$-form $\Omega_2=\xi(y)d^ny - \eta(v)d^nv$ vanishes on $L_{f(y)}=\{y^i,v_i=f_{y^i}\}$ if and only if
\begin{multline*}
0=\xi(y)-\eta(f_{y^i})(f_{11}dy^1\wedge\dots\wedge f_{1n}dy^n)\wedge \dots \wedge (f_{n1}dy^1\wedge\dots\wedge f_{nn}dy^n)
= \big(\xi(y)-\eta(f_{y^i})\det(Hess(f)) \big)d^ny\,,
\end{multline*}
i.e., if
\begin{equation}\label{eq.Monge-Kant}
\det \mathrm{Hess}(f)=\frac{\xi(y)}{\eta(f_{y^i})}\,.
\end{equation}
In view of a result by \cite{MR3220439}, $f$ satisfies the PDE \eqref{eq.Monge-Kant} if and only if the function $F$ defined in Proposition \ref{prop.ultimo.sforzo} solves the Monge--Kantorovich problem described in Problem \ref{prob.Monge_Kant}.
\end{proof}

\section{Homogeneous $2\Nd$ order PDEs on the projectivized cotangent bundle of $\C\p^{n+1}$}\label{secVeryClear}
Recall (see Proposition \ref{propPTRP2homogeneous}) that the natural action of the real Lie group $\SL_3(\R)$ on $\R\p^2$ lifts to a transitive action on the 3--dimensional (real) contact manifold $\p T^*\R\p^2$. That is, the latter is a (real) homogenous contact manifold.\par
In this Section we switch to the complex--analytic setting and we work with the   complex projective space $ \C\p^{n+1}$. In analogy with the aforementioned Proposition, we show that the $(2n+1)$--dimensional \emph{complex}  contact manifold
\begin{equation*}
M:=\p T^*\C\p^{n+1}
\end{equation*}
is homogeneous with respect to the complex Lie group
 \begin{equation*}
G:=\PGL_{n+2}(\C)\,,
\end{equation*}
which acts naturally on $\C\p^{n+1}$ (see also Remark \ref{remPGL3}). Symbols $M$ and $G$ will denote $\p T^*\C\p^{n+1}$ and $\PGL_{n+2}$, respectively, throughout this entire Section.
\subsection{The adjoint contact manifold of $\PGL_{n+2}(\C)$}
In this Section we show that  there is essentially a unique $2\Nd$ order PDE $\E\subset M^{(1)}$, whose group of symmetries (see Section \ref{secSymPDEBack}) is precisely  $G$. This example fits into a very extensive research programme, that is, finding PDEs with a prescribed Lie group of symmetries, which originated in the work of Lie, Darboux, Cartan and others (see \cite{2016arXiv160602633A,2016arXiv160308251T}). \par
Observe that the Lie algebra of $G$ coincides with the Lie algebra of $\SL_{n+2}$, that is, $\sll_{n+2}$. Consider the adjoint action of $G$ on $\p\sll_{n+2}$. Lemma \ref{lemTypeA} below shows that such an action admits a unique closed orbit, which is precisely our homogeneous contact manifold $M$. In view of this property, $M$ is often referred to as the \emph{adjoint contact manifold} of $G$.  In \cite{2016arXiv160602633A} the authors study the adjoint contact manifold of any complex simple Lie group $G$ and obtain the corresponding $G$--invariant $2\Nd$ order PDE. Herewith we confine ourself to the easiest case, that is, when $G:=\PGL_{n+2}$ is of type $\Asf$.
\begin{lemma}\label{lemTypeA}
$M$ is  the unique closed orbit of $G$ in  $\p\sll_{n+2}$.
\end{lemma}
\begin{proof}
A point $L\in \p^{n+1}$ is a line in $\C^{n+2}$, and an element $H\in\p T^*_L\p^{n+1}$ is a tangent hyperplane to  $\p^{n+1}$ at $L$. As such, $H$ is an hyperplane in $\C^{n+2}$ containing the line $L$. In other words,   $M$ can be identified with the space
\begin{equation*}
 \{  (L,H)\mid L\subset H\}\subset \p^{n+1}\times\p^{n+1\,\ast}\,
\end{equation*}
of $(1,n+1)$--flags in $\C^{n+2}$.\par
Fix standard coordinates $e_{1},\ldots, e_{n+2}$ on $\C^{n+2}$, together with their duals $\epsilon^i$, and set the point
\begin{equation*}
o:=(L_0=\Span{e_1}, H_0=\Span{e_1,\ldots,e_{n+1}}=\ker\epsilon^{n+2})
\end{equation*}
as the origin of  $M$.  The Lie algebra of the stabiliser   of $o$ is
\begin{equation*}
\gp= \left\{  \left(\begin{array}{ccc}\lambda & v & \mu \\0 &A & w \\0 & 0 & \nu\end{array}\right) \mid A\in  \gl_n , v,w\in\C^n,  \lambda+\nu+\tr A=0 \right\}\, .
\end{equation*}
Since $\dim\sll_{n+2}-\dim\gp=((n+2)^2-1)-(n^2+2n+2)=2n+1$, the  $G$--orbit through $o$ is open in $M$.\par
Now we show that the   map
\begin{eqnarray*}
M&\stackrel{F}{\longrightarrow} & \p(\sll_{n+2})\subset\p(\C^{n+2}\otimes\C^{n+2\ast})\, ,\\
(L=\Span{v},H=\ker\varphi) &\longmapsto & [v\otimes\varphi]\,,
\end{eqnarray*}
is well--defined, injective, $G$--equivariant, and it defines a contactomorphism on its image, which is precisely the adjoint contact manifold.
The class $[v\otimes\varphi]$  is well--defined because both $v$ and $\varphi$ are defined up to a  non--zero factor, and  $v\otimes\varphi$ indeed belongs to $\sll_{n+2}$, since from $L\subset H$ it follows that   $\tr (v\otimes\varphi) =\varphi(v)= 0$. By construction, $F(o)=[e_1\otimes\epsilon^{n+2}]$,
where $e_1\otimes\epsilon^{n+2}$ is the highest weight vector of $\sll_{n+2}$, whence $X:=G\cdot [e_1\otimes\epsilon^{n+2}] $ is, by definition, the adjoint variety (the unique closed orbit is the one passing through the highest weight vector).\par
The $G$--equivariancy of $F$ is obvious, since
\begin{equation*}
g\cdot (L,H)=(g(L),g(H))=(\Span{g(v)},\ker g^*(\varphi))\longmapsto [g(v)\otimes g^*(\varphi)]=g\cdot [v\otimes\varphi]\, ,\quad\forall g\in G\, .
\end{equation*}
Moreover, $F$ is injective being the restriction of the Segre embedding  $\p^{n+1}\times\p^{n+1\,\ast}\subset \p(\C^{n+2}\otimes\C^{n+2\ast})$, and its image coincides with $X$. Indeed, if $[h]\in \p(\sll_{n+2})$, where $h$ is a rank--$1$ homomorphism, then $(\textrm{im}\, h, \ker h)\in M$ and
$[h]=F((\textrm{im}\, h, \ker h))$.
It remains to prove that $F$ realises a contactomorphism between the contact structures on $M$ and $X$. By homogeneity, we can simply show that $T_oF$ sends the contact hyperplane (see also Proposition \ref{prop.HpContact})
\begin{equation*}
H_0\oplus T_{H_0}(\p T_{L_0}^*\p^{n+1})\subset T_{L_0}\p^{n+1}\oplus T_{H_0}(\p T_{L_0}^*\p^{n+1})=T_oM
\end{equation*}
to the contact hyperplane of $T_{F(o)}X$. The latter is better described in terms of the cone $\widehat{X}$ over $X$, inside $\sll_{n+2}$: it is the subspace
\begin{equation*}
 {\ker\ad_{e_1\otimes\epsilon^{n+2}}}{ }\subset {[\sll_{n+2},e_1\otimes\epsilon^{n+2}]}{ }=T_{e_1\otimes\epsilon^{n+2}}\widehat{X}\, .
\end{equation*}
A curve $\gamma(t):=(\Span{v_t},\ker\phi_t)$ belongs to the contact hyperplane at $o$ if and only if $\epsilon^{n+2}(v_0')=0$, that is, if the horizontal projection of $\gamma(t)$ keeps, to first order, the line $\Span{v_t}$ inside the hyperplane $\ker\phi_t$. Observe that, since $\phi'_0$ is the velocity of a curve of hyperplanes containing $L_0$, we have also  $ \phi'_0(e_1)=0$. Consider now the curve
\begin{equation}\label{eqSprojecCurve}
 t\longmapsto e_1\otimes\epsilon^{n+2} + (v'_0\otimes\epsilon^{n+2}+e_1\otimes\phi'_0)t+o(t^2)\,
\end{equation}
in $\sll_{n+2}$, whose projectivisation is precisely $F_*\gamma$.
Since
\begin{equation*}
\ad_{e_1\otimes\epsilon^{n+2}}(  v'_0\otimes\epsilon^{n+2}+e_1\otimes\phi'_0 ) =(\epsilon^{n+2}(v_0')+ \phi'_0(e_1))e_1\otimes\epsilon^{n+2}=0\, ,
\end{equation*}
the velocity at $0$ of \eqref{eqSprojecCurve}  belongs to the contact hyperplane of $\widehat{X}$ at $e_1\otimes\epsilon^{n+2}$, whence the velocity at $0$ of    $F_*\gamma$ at $F(o)$ belongs to the contact hyperplane of $X$ at $F(o)$.
\end{proof}
\begin{corollary}
 There is a gradation of $\sll_{n+2}$, namely
 \begin{equation}\label{eq:gradingTypeA}
\sll_{n+2}= \underset{ {  \g_{-2}  }}{\underbrace{ \C}} \oplus\underset{ {   \g_{-1} }} {\underbrace{ (\C^n\oplus \C^{n\ast})}} \oplus\underset{ {  \g_0 }}{\underbrace{  (\sll_n\oplus\C^2)}} \oplus\underset{ {   \g_{1} }}{\underbrace{  (\C^{n\ast}\oplus\C^n) }} \oplus\underset{ {   \g_{2} }}{\underbrace{  \C^*}}\, ,
\end{equation}
such that $\g_{<0}$ is the tangent space to $M$ and $\g_{-1}$ is the contact hyperplane within it.
\end{corollary}
\begin{proof}
 The stabiliser   of $o$ is
\begin{equation*}
\gp= \left\{  \left(\begin{array}{ccc}\lambda & v & \mu \\0 &A & w \\0 & 0 & \nu\end{array}\right) \mid A\in  \gl_n , v,w\in\C^n,  \lambda+\nu+\tr A=0 \right\}\, .
\end{equation*}
Then
\begin{equation*}
 \left(\begin{array}{ccc}\lambda & v & \mu \\0 &A & w \\0 & 0 & \nu\end{array}\right) \longmapsto \left( \underset{ {  \g_0 }}{\underbrace{  A-\frac{1}{n}\tr A\id_n, \tr A, \lambda+\nu}} ,\underset{ {   \g_{1} }}{\underbrace{  \left(\begin{array}{ccc}0 & v & 0 \\0 &0 & 0 \\0 & 0 & 0\end{array}\right),\left(\begin{array}{ccc}0 & 0 & 0 \\0 &0 & w \\0 & 0 & 0\end{array}\right) }} ,\underset{ {   \g_{2} }}{\underbrace{  \mu}} \right)
\end{equation*}
defines an isomorphisms $\gp\cong\g_{0}\oplus\g_1\oplus\g_2$.
\end{proof}
\begin{remark}
 The twisted symplectic form on  $\g_{-1}$ is the unique one extending the standard pairing between $\C^n$ and $\C^{n\ast}$.
Indeed, observe that
 \begin{equation*}
\left[ \left(\begin{array}{ccc}0 & 0 & 0 \\v &0 & 0 \\0 & 0 & 0\end{array}\right)    ,   \left(\begin{array}{ccc}0 & 0 & 0 \\0 &0 & 0 \\0 & w & 0\end{array}\right)      \right]=-v\cdot w
\end{equation*}
is the matrix counterpart of the pairing
\begin{equation*}
\C^{n}\times \C^{n\ast}\ni(x,\xi)\longmapsto \xi(x)\in\C \qedhere
\end{equation*}
(cf. also Proposition \ref{propCRnRnStar}).
\end{remark}
 \subsection{Finding $\PGL_{n+2}(\C)$--invariant $2\Nd$ order PDEs on $\p T^*\C\p^n$}\label{secHomMAEonPTstarCP}
 From \eqref{eq:gradingTypeA} it follows that all the $\g_i$'s are $\g_0$--modules and hence modules over the semi--simple part of $\g_0$, which is $\sll_n$. In particular, we can  decompose the dual of the Pl\"ucker embedding space $\Lambda_0^n(\g_{-1}^*)$
into $\sll_n$--irreducible components. According to Definition \ref{defPluck},
  $\Lambda_0^n(\g_{-1}^*)=\Lambda_0^n(\C^n\oplus\C^{n\ast})$ is the subspace of $\Lambda^n(\C^n\oplus\C^{n\ast})$  consisting of $n$ forms whose contraction with the standard symplectic form $\omega$ on $\g_{-1}=\C^n\oplus\C^{n\ast}$ vanishes. Since the Pl\"ucker embedding is going to play a key role in this Section,  the reluctant reader may find it useful to go back to  Section \ref{sec.L24} before continuing. There, the  case of $\LL(2,4)$ has been thoroughly examined and one of the simplest instances of a Pl\"ucker embedding has been provided (see formula formula \eqref{eqPluckL24}).\par
Since $G_0=\GL_n$ preserves the bi--Lagrangian decomposition
$\g_{-1} = \C^n \oplus \C^{n*}$ (see Definition \ref{DefDopo_propCRnRnStar}), we have:
\begin{align}
\Lambda^n_0 (\g_{-1}^*)
 &= \Lambda^n_0 (\C^{n\ast}\oplus \C^{n})\nonumber\\
 &= \bigoplus_{i=0}^nS^2_0(\Lambda^i(\C^{n\ast}))\subset \bigoplus_{i=0}^n (\Lambda^i(\C^{n\ast}))^{\otimes 2}= \bigoplus_{i=0}^n (\Lambda^i(\C^{n\ast}))\otimes (\Lambda^{n-1}(\C^{n }))\label{eqMysterEmbed},\\
\Lambda^n_0(\g_{-1}^*) &=\C\oplus S_0^2\C^{n\ast}\oplus\cdots \oplus S_0^2 (\Lambda^{n-1}(\C^{n\ast}))\oplus S_0^2(\Lambda^n(\C^{n\ast}))\, .\label{eqDecPluckLinearAnp1}
\end{align}

The step \eqref{eqMysterEmbed} may use some extra comment. First, we decomposed $n$--forms on $\C^{n\ast}\oplus \C^{n}$ as products of forms on each summand, and then we used Poincar\'e duality. Finally,   one checks that  a bilinear form on $\Lambda^i(\C^{n})$ belongs to the kernel of the contraction with $\omega$  if and only if it is symmetric and trace--free.
So,  \eqref{eqDecPluckLinearAnp1} represents the decomposition of the space of linear functions on the Pl\"ucker embedding space of $M^{(1)}_o$ into $\SL_n$--irreducible submodules.
Clearly, there are only two $1$--dimensional summands in \eqref{eqDecPluckLinearAnp1}:
the first and the last.\par

Since $\Lambda^n_0 (\g_{-1}^*)$ is the space of linear functions on the Pl\"ucker embedding space, the two hyperplanes corresponding, by line--to--hyperplane duality, to the two $1$--dimensional summands in \eqref{eqDecPluckLinearAnp1}, that is, the first and the last one, are by construction $\SL_n$--invariant. By the definition\footnote{The $n$--dimensional, i.e., general definition is Definition \ref{def.E.D.n.var}. Its 2--dimensional counterpart is Definition \ref{def.E.D.2.var}, provided earlier. Since the 2--dimensional case helps to visualise important notions, the reader may  look again at    Definition \ref{defHypSec} and Definition \ref{defMAE}.} of Monge--Amp\`ere equations, these produce in turn two $\SL_n$--invariant Monge--Amp\`ere equations. Their ``speciality'' is due to the fact that 1 is the minimal degree of an invariant hypersurface in $\p  \Lambda^n_0 (\g_{-1}^*)$ (see also \cite{2016arXiv160308251T}).

\subsection{Homogeneous Darboux coordinates on $\p T^*\C\p^n$}
In order to see these hyperplane sections as PDEs according to the common understanding (see Section \ref{sec.SecondOrderPDEs}), we need a sort of ``homogeneous version'' of the   Darboux's theorem (see Theorem \ref{th.Darboux.contact}).
\begin{proposition}\label{propDarboux}
 For any bi--Lagrangian decomposition
 $L\oplus L^* = \g_{-1}$
there exist complex Darboux coordinates
 \begin{equation}\label{eqDarbSuGiMenoUno}
  x^1,\ldots, x^n, u, u_1,\ldots, u_n,
\end{equation}
in a neighborhood of $o$, such that
$$
L=g^{-1}\cdot\Span{\left.D_{x^1}\right|_{gP},\ldots, \left.D_{x^n}\right|_{gP}} ,\,\,\,
\g_{-2} =g^{-1}\cdot\Span{\left.\frac{\partial}{\partial u}\right|_{gP}} ,\,\,\,
L^*=g^{-1}\cdot\Span{\left.\frac{\partial}{\partial u_1}\right|_{gP},\ldots, \left.\frac{\partial}{\partial u_n}\right|_{gP}}\, ,
$$
for all $g\in G$, where $D_{x^i}$ are the total derivatives (see \eqref{eq.contact.local}).
\end{proposition}
\begin{proof}
By restricting the exponential map
$ \g\ni g \mapsto \exp (g)\in G$
to $\g_-$, one obtains an (algebraic) isomorphism
$
\Psi:\g_-\longrightarrow U\subseteq X
$,
between the linear space $\g_-$ and an open neighborhood $U$ of the origin.
For any $v \in\g_-$, denote by $\widehat{v}$ the  vector field on $U $ induced by $v$. Then we have that
\begin{equation}\label{eqFormulaMagicaJan}
\widehat{v }_{\Psi (w )}=T_w \Psi  \left(v -\frac{1}{2}[w ,v ]\right)\, ,
\end{equation}
for all $w \in\g_-$.
Formula \eqref{eqFormulaMagicaJan} follows directly from the Baker--Campbell--Hausdorff formula
\begin{equation*}
 e^{t \widehat{v }}\Psi (w ) =  e^{tv }e^w P = e^{tv +w +\frac{1}{2}[tv ,w ]}P=\Psi \left(w +tv +\frac{1}{2}t[v ,w ]\right)
\end{equation*}
(see, e.g.,  \cite[Section 1.3]{9780199202515}).
We can  now pull--back the contact distribution $\CC$ on $X $  to a contact distribution (denoted by the same symbol $\CC$) on $\g_{-}$, by setting
$
\CC_w :=(T_w \Psi )^{-1}(\CC_{\Psi (w )})
$.
Then   \eqref{eqFormulaMagicaJan} implies
\begin{equation}\label{eqFormulaMagicaJan2}
\CC_w =\left(\id -\frac{1}{2}\ad_w \right)\g_{-1}\, .
\end{equation}
Fix now vectors $l^1,\ldots, l^n,r,l_1,\ldots, l_n$ such that
$
L=\Span{l^1,\ldots, l^n}$, $ \g_{-2}=\Span{r}$, $ L^*=\Span{l_1,\ldots, l_n}$.
Then from \eqref{eqFormulaMagicaJan2}  it follows that the vectors fields $D_i$ and $V^i$ on $\g_{-1}$ defined by
$$
\left.D_i\right|_w := \left(\id -\frac{1}{2}\ad_w \right)l_i\, ,\quad
\left.V^i\right|_w := \left(\id -\frac{1}{2}\ad_w \right)l^i\, ,
$$
for all $w \in\g_{-}$, form a basis of $\CC$.
The last step is to show that there are coordinates
 \begin{equation}\label{eqDarbSuGiMenoUnoPRE}
  \underline{x}^1,\ldots, \underline{x}^n, \underline{u}, \underline{u}_1,\ldots, \underline{u}_n,
\end{equation}
on  $\g_-$  such that
$$
 D_{\underline{x}^i} =  \frac{\partial}{\partial \underline{x}^i}+\underline{u}_i\frac{\partial}{\partial \underline{u}} \, ,\quad
V^i =  \frac{\partial}{\partial \underline{u}^i}  \, .
$$
But this is true, if one sets
$
  \underline{x}^i:= l^i$, $
    \underline{u}:= r^\vee+l^il_i$, $
      \underline{u}_i:= l_i
$.
Then the desired coordinates \eqref{eqDarbSuGiMenoUno} are just the pull--backs via $\Psi $ of \eqref{eqDarbSuGiMenoUnoPRE}.
\end{proof}

\subsubsection{Homogeneous Monge--Amp\`ere equations on  $\p T^*\C\p^n$}
Now we can apply  Proposition \ref{propDarboux}  to the bi--Lagrangian decomposition $\C^n\oplus\C^{n\ast}$ of the contact hyperplane $\CC_o=\g_{-1}$ at the origin $o$ of the $G$--homogeneous contact manifold $M$. We then obtain Darboux coordinates in such a way that the first summand $\C^n$ is spanned by the ``horizontal'' vectors $D_{x^i}$ and the second summand $\C^{n\ast}$ is spanned by the ``vertical vectors'' $\partial_{u_i}$.\par
Thanks to these coordinates we can write down explicitly the two hyperplane sections (that is, the Monge--Amp\`ere equations) obtained in Section \ref{secHomMAEonPTstarCP} above. We recall that they correspond to the first and last $1$--dimensional constituent in \eqref{eqDecPluckLinearAnp1}, which, in turn, are spanned by the $n$--forms $d x^1\wedge\cdots\wedge d x^n$ and $du_1\wedge\cdots\wedge du_n$, respectively.\par
Now it remains to build the corresponding hypersurfaces in $\E_\Omega\subset M^{(1)}$, that is, Monge--Amp\`ere equations, according to the general formula \eqref{eq.MAE.intr.Omega}, by choosing $\Omega$ to be one of the above $n$--forms. The keen reader may object that, in the first case, one would obtain an empty equation, since the $n$--form is entirely horizontal. Indeed, among all simple Lie groups, the case of type $\Asf$, i.e., the case when $G=\PGL_{n+2}(\C)$,  is a peculiar one, since it appears that one of the two hyperplane sections cannot be interpreted as a honest equation. This subtle question, extensively discussed in \cite{2016arXiv160602633A}, can be summarized as follows.\par
Let us construct first the equation corresponding to $\Omega=d u_1\wedge\cdots\wedge d u_n$. Immediate computations show that
\begin{equation}\label{eqHessIndet}
\E_{\Omega}=\{\det (u_{ij})=0\}\, ,
\end{equation}
and as simple computations show also that $\E_{\Omega}=\E_{\C^{n}}$  (recall Definition \ref{def.E.D.n.var}). This is quite standard (see  \cite{MR2985508}), but it all boils down to the observation that the Lagrangian plane
$
L=\Span{D_{x^i}+u_{ij} \partial_{u_j}\mid i=1,\ldots n}
$
intersects non--trivially $\C^n=\Span{D_{x^i} \mid i=1,\ldots n}$ if and only if $ \det (u_{ij})=0$. Hence, $\E_{\C^n}=\{\det (u_{ij})=0\}$.\par
When $\E_\Omega$ is recast as  $\E_{\C^{n}}$, it is easier to  guess how the equation looks like when $\Omega=du_1\wedge\cdots\wedge du_n$. It is the hypersurface $\E_{\C^{n\ast}}\subset M^{(1)}$.\par
 The fact that $\E_{\C^{n}}$ can be written down explicitly as the $2\Nd$ order PDE \eqref{eqHessIndet} whereas  $\E_{\C^{n\ast}}$ cannot, is just an accident due to the fact that the Darboux coordinate system provided by the Proposition \ref{propDarboux} was implicitly chosen in such a way to ``favor'' the first  equation. Even if there exist, in principle,   Darboux coordinates allowing to simultaneously
 express both $\E_{\C^{n}}$ and $\E_{\C^{n*}}$
as explicit PDEs, \emph{none of them is of the form provided by Proposition \ref{propDarboux}}, i.e., compatible with the action of $G$.
\begin{remark}
The equation \eqref{eqHessIndet} plays a role of paramount importance in classical Algebraic Geometry, as it is at the heart of the so--called Gordan--Noether conjecture (see the recent book \cite{MR3445582} for an excellent review). It concern the \emph{solutions} of the equation \eqref{eqHessIndet}. The fundamental observation is that there are some evident solutions to  \eqref{eqHessIndet}: the \emph{cones}. It is easy to convince oneself that it is so, since, de--projectivising, cones are simply graphs of surfaces $u=u(x^1,\ldots,x^n)$, which do not depend on all the variables. The Gordan--Noether conjecture says precisely that these are \emph{all} the (global, algebraic) solutions to \eqref{eqHessIndet}. Interestingly enough, the conjecture starts to fails from $n=5$.\par
Observe that the hyperplanes of $\C\p^{n+1}$ are particular instance of cones, and hence solutions to \eqref{eqHessIndet}. It is interesting to notice that the $\SL_{n+2}$--invariant Monge--Amp\`ere equation $\E_{\C^{n}}$ can be \emph{defined} as to be the unique $\SL_{n+2}$--invariant $2\Nd$ order PDE on $M=\p T^*\C\p^{n+1}$, whose global \emph{smooth} solutions include all the hyperplanes of $\C\p^{n+1}$.
\end{remark}

\nocite{MR2352610}

\bibliographystyle{plainnat}
\bibliography{BibUniver}

\end{document}